\documentclass[10pt]{amsart}
\usepackage{amssymb, amsthm, amsmath}
\usepackage[all]{xy}
\usepackage{graphicx}
\usepackage{psfrag}
\usepackage[vcentermath]{youngtab}

\newtheorem{thm}{Theorem}
\newtheorem{lemma}{Lemma}[section]
\newtheorem{prop}[lemma]{Proposition}
\newtheorem{cor}[lemma]{Corollary}
\theoremstyle{definition}
\newtheorem{idefn}{Definition}
\newtheorem{defn}[lemma]{Definition}
\newtheorem{example}[lemma]{Example}
\newtheorem{remark}[lemma]{Remark}
\newtheorem{claim}{Claim}

\newcommand\A{{\mathbb A}}

\newcommand\C{{\mathbb C}}
\newcommand\Q{{\mathbb Q}}
\newcommand\N{{\mathbb N}}
\newcommand{\cN}{{\mathcal N}}
\newcommand\IH{{\mathbb H}}
\newcommand{\ti}{\vartheta}
\newcommand{\Ti}{\Theta}
\newcommand{\oti}{\widehat{\vartheta}}
\newcommand{\oTi}{\widehat{\Theta}}
\newcommand\W{{\mathrm W}}
\newcommand\NN{{\mathfrak N}}

\newcommand\cO{{\mathcal O}}
\newcommand\cC{{\mathcal C}}

\newcommand\X{{\mathrm X}}
\newcommand\Z{{\mathbb Z}}
\newcommand\cP{{\mathcal P}}
\newcommand\cQ{{\mathcal Q}}
\newcommand\cS{{\mathcal S}}

\newcommand\AS{{\mathfrak S}}
\newcommand\CS{{\mathfrak C}}

\newcommand\al{\alpha}

\newcommand\la{\lambda}
\newcommand{\e}{e}
\newcommand{\f}{f}
\newcommand{\g}{g}
\newcommand\s{{\sigma}}
\newcommand\ta{{\tau}}
\newcommand\Sh{{\mathcal S}}

\newcommand\ssm{\smallsetminus}
\newcommand\gequ{\geq}
\newcommand\lequ{\leq}
\newcommand\noin{\noindent}

\newcommand\bull{{\scriptscriptstyle \bullet}}
\newcommand\eqto{\stackrel{\lower1.5pt\hbox{$\scriptstyle\sim\,$}}\to}
\newcommand\ov{\overline}

\newcommand\wh{\widehat}
\newcommand\wt{\widetilde}

\DeclareMathOperator{\Pf}{Pfaffian}
\DeclareMathOperator{\Sp}{Sp}

\DeclareMathOperator{\LG}{LG}
\DeclareMathOperator{\IG}{IG}

\DeclareMathOperator{\OG}{OG}
\DeclareMathOperator{\G}{G}
\DeclareMathOperator{\ev}{ev}
\DeclareMathOperator{\HH}{\mathrm{H}}

\newcommand{\ignore}[1]{}
\newcommand{\pic}[2]{\includegraphics[scale=#1]{#2}}

\psfrag{k}{$k$}
\psfrag{a}{$a$}
\psfrag{b}{$b$}
\psfrag{h}{$h$}
\psfrag{g}{$g$}
\psfrag{x}{$x$}
\psfrag{bh}{$b_h$}
\psfrag{gh}{$g_h$}
\psfrag{bx}{${\mathbf x}$}
\psfrag{(r,c)}{$(r,c)$}
\psfrag{(r',c')}{$(r',c')$}
\psfrag{la1}{$\la^1$}
\psfrag{la2}{$\la^2$}

\begin{document}

\title[A Giambelli formula for isotropic Grassmannians]
{A Giambelli formula for isotropic Grassmannians}

\date{August 3, 2010}

\author{Anders Skovsted Buch}
\address{Department of Mathematics, Rutgers University, 110
  Frelinghuysen Road, Piscataway, NJ 08854, USA}
\email{asbuch@math.rutgers.edu}

\author{Andrew Kresch}
\address{Institut f\"ur Mathematik,
Universit\"at Z\"urich, Winterthurerstrasse 190,
CH-8057 Z\"urich, Switzerland}
\email{andrew.kresch@math.uzh.ch}

\author{Harry~Tamvakis} \address{Department of Mathematics, University of
Maryland, 1301 Mathematics Building, College Park, MD 20742, USA}
\email{harryt@math.umd.edu}

\subjclass[2000]{Primary 14N15; Secondary 05E15, 14M15}

\thanks{The authors were supported in part by NSF Grants DMS-0603822
  and DMS-0906148 (Buch), the Swiss National Science Foundation
  (Kresch), and NSF Grants DMS-0639033 and DMS-0901341 (Tamvakis).}

\begin{abstract}
  Let $X$ be a symplectic or odd orthogonal Grassmannian parametrizing
  isotropic subspaces in a vector space equipped with a nondegenerate
  (skew) symmetric form. We prove a Giambelli formula which expresses
  an arbitrary Schubert class in $\HH^*(X,\Z)$ as a polynomial in
  certain special Schubert classes.  We study {\em theta polynomials},
  a family of polynomials defined using raising operators whose
  algebra agrees with the Schubert calculus on $X$. Furthermore, we
  prove that theta polynomials are special cases of Billey-Haiman
  Schubert polynomials and use this connection to express the former
  as positive linear combinations of products of Schur $Q$-functions and
  $S$-polynomials.
\end{abstract}

\maketitle

\setcounter{section}{-1}

\section{Introduction}

Let $\G=\G(m,N)$ denote the Grassmannian of $m$-dimensional subspaces
of $\C^N$.  To each integer partition $\la=(\la_1,\ldots,\la_m)$ whose
Young diagram is contained in an $m\times (N-m)$ rectangle, we
associate a Schubert class $\s_\la$ in the cohomology ring of $\G$.
The {\em special} Schubert classes $\s_1,\ldots,\s_{N-m}$ are the
Chern classes of the universal quotient bundle $\cQ$ over $\G(m,N)$;
they generate the cohomology ring $\HH^*(\G,\Z)$.  The classical {\em
Giambelli formula} \cite{G}
\begin{equation}
\label{giamS}
\s_{\la} = \det(\s_{\la_i+j-i})_{i,j}
\end{equation}
is an explicit expression for $\s_\la$ as a polynomial in the special
classes; as is customary, we agree here and in later formulas that
$\s_0=1$ and $\s_r=0$ for $r<0$.

The relation between the Schubert calculus on the Grassmannian
$\G(m,N)$ and the algebra of Schur's $S$-functions $s_\la$ (originally
defined by Cauchy \cite{C} and Jacobi \cite{J}) is well known.  Given
an infinite list $x=(x_1,x_2,\ldots)$ of commuting independent
variables, we define the elementary symmetric functions $e_r(x)$ by
the formal relation
\[
\prod_{i=1}^{\infty}(1+x_it) = \sum_{r=0}^{\infty}e_r(x)t^r
\]
and set, for any partition $\la$,
$s_{\la'}(x)=\det(e_{\la_i+j-i}(x))_{i,j}$.  Here $\la'$ is the
partition whose Young diagram is the transpose of the diagram of
$\la$.  The ring $\Lambda = \Z[e_1,e_2,\ldots]$ of symmetric functions
in $x$ has a free $\Z$-basis consisting of the Schur functions
$s_{\la}$, for all partitions $\la$.  These Schur $S$-functions enjoy
many good combinatorial properties, such as nonnegativity of their
coefficients, and multiply exactly like the Schubert classes on
$\G(m,N)$, when $m$ and $N$ are sufficiently large.

There is a closely analogous story to the above for the Lagrangian
Grassmannian $\LG(n,2n)$ which parametrizes maximal isotropic
subspaces of $\C^{2n}$, with respect to a symplectic form.
The Schubert classes in $\HH^*(\LG,\Z)$
are indexed by {\em strict} partitions, i.e., partitions with distinct
(non-zero) parts, whose diagrams fit in a square of side $n$.
The special Schubert classes $\s_r=c_r(\cQ)$ again
generate the cohomology ring, and there is a Giambelli-type formula
due to Pragacz \cite{Pra}.  This latter may be
described in two steps: For partitions $\la=(a,b)$
with only two parts, we have
\begin{equation}
\label{giamQ0}
\s_{a,b}=\s_a\s_b-2\s_{a+1}\s_{b-1}+2\s_{a+2}\s_{b-2}-\cdots
\end{equation}
while for $\la$ with $3$ or more parts,
\begin{equation}
\label{giamQ}
\s_\la = \Pf(\s_{\la_i,\la_j})_{i<j}.
\end{equation}

The identities (\ref{giamQ0}) and (\ref{giamQ}) in
fact also go back to the work of Schur \cite{S2}, who considered
a family of symmetric functions $\{Q_\la\}$ known as Schur
$Q$-functions.  We define $q_r(x)$ by the equation
\[
\prod_{i=1}^{\infty}\frac{1+x_it}{1-x_it} = \sum_{r=0}^{\infty}q_r(x)t^r
\]
and then use the same relations (\ref{giamQ0}) and (\ref{giamQ}) with
$q_r(x)$ in place of $\s_r$ to define $Q_{a,b}(x)$ and then
$Q_\la(x)$, for each strict partition $\la$.  If we let $\Gamma =
\Z[q_1,q_2,\ldots]$ denote the ring of Schur $Q$-functions, then the
$\{Q_\la\}$ for $\la$ strict form a $\Z$-basis for $\Gamma$, whose
algebra agrees with Schubert calculus on $\LG(n,2n)$, as $n \to
\infty$.  Moreover, there is a well developed combinatorial theory for
the $Q$-functions, analogous to that for the $S$-functions.

Choose $k\geq 0$ and consider now the Grassmannian $\IG(n-k,2n)$ of
isotropic $(n-k)$-dimensional subspaces of $\C^{2n}$, equipped with a
symplectic form.  We call a partition $\la$ {\em $k$-strict} if no
part greater than $k$ is repeated, i.e., $\la_j > k \Rightarrow \la_j
> \la_{j+1}$.  The Schubert classes on $\IG$ are indexed by $k$-strict
partitions whose diagrams fit in an $(n-k)\times (n+k)$
rectangle.  Given such a $\la$ and a complete flag of subspaces
$F_\bull: \, 0=F_0 \subsetneq F_1 \subsetneq \cdots \subsetneq
F_{2n}=\C^{2n}$ such that $F_{n+i}= F_{n-i}^\perp$ for $0\leq i \leq
n$, we have a Schubert variety
\[ X_\lambda(F_\bull) := \{ \Sigma \in \IG \mid \dim(\Sigma \cap
   F_{p_j(\lambda)}) \gequ j \ \ \forall\, 1 \lequ j \lequ
   \ell(\lambda) \} \,,
\]
where $\ell(\la)$ denotes the number of (non-zero) parts of $\la$ and
\begin{equation}
\label{pdefC}
p_j(\lambda) := n+k+j-\lambda_j - \#\{i<j : \lambda_i+\lambda_j
> 2k+j-i \}.
\end{equation}
This variety has codimension $|\la|=\sum \la_i$ and defines, using
Poincar\'e duality, a Schubert class $\s_{\la}=[X_{\la}(F_\bull)]$ in
$\HH^{2|\la|}(\IG,\Z)$.  As above, we consider the special Schubert classes
$\s_r=[X_r(F_\bull)]=c_r(\cQ)$ for $1\lequ r \lequ n+k$.

In \cite{BKT}, we proved a Pieri rule for the products $\s_r\s_{\la}$
in $\HH^*(\IG)$.  Equipped with this rule and the help of a computer,
we observed that (i) when $\la_j\leq k$ for all $j$, then $\s_\la$ is
given by the determinantal formula (\ref{giamS}); (ii) when $\la_j >
k$ for all non-zero $\la_j$, then $\la$ is strict and $\s_\la$ is
given by the Pfaffian formulas (\ref{giamQ0}), (\ref{giamQ}).  It is
tempting to ask for an analogous Giambelli formula for $\s_\la$ when
$\la$ is a general $k$-strict partition.  Note that the formula is
determined only up to an ideal of relations; whatever the answer, it
must naturally interpolate between the Jacobi-Trudi determinant
(\ref{giamS}) and the Schur Pfaffian (\ref{giamQ}). A similar question
was also raised by Pragacz and Ratajski \cite{PRpieri}, who were using
a different set of special Schubert classes.

The answer we give depends crucially on our choice of $k$-strict
partitions to index the Schubert classes, and uses Young's {\em
raising operators} \cite[p.\ 199]{Y}.  For any integer sequence
$\alpha=(\alpha_1,\alpha_2,\ldots)$ with finite support and $i<j$, we
define $R_{ij}(\alpha) =
(\alpha_1,\ldots,\alpha_i+1,\ldots,\alpha_j-1, \ldots)$; a raising
operator $R$ is any monomial in these $R_{ij}$'s.  Set $m_\alpha =
\prod_i \s_{\alpha_i}$ and $R\,m_{\al} = m_{R\al}$ for any raising
operator $R$ \footnote{As is customary, we slightly abuse the notation
and consider that the raising operator $R$ acts on the index $\al$,
and not on the monomial $m_\al$ itself.}.  Using these operators, the
Giambelli formulas (\ref{giamS}) and (\ref{giamQ0})--(\ref{giamQ}) can
be expressed as
\begin{equation}
\label{spgiambelli}
\s_\la = \prod_{i<j}(1-R_{ij})\,m_{\la} \ \ \ \text{and} \ \ \
\s_\la = \prod_{i<j}\frac{1-R_{ij}}{1+R_{ij}}\,m_{\la},
\end{equation}
respectively.

\begin{idefn}
\label{maindefn}
For a general $k$-strict partition $\la$, we define the
operator
\[
R^{\la} = \prod_{i<j} (1-R_{ij})\prod_{\la_i+\la_j > 2k+j-i}
(1+R_{ij})^{-1}
\]
where the first product is
over all pairs $i<j$ and second product is over pairs $i<j$ such that
$\la_i+\la_j > 2k+j-i$.
\end{idefn}

\begin{thm}
\label{mainthm}
For any $k$-strict partition $\la$, we have
$\s_\la =R^{\la}m_{\la}$
in the cohomology ring of $\IG(n-k,2n)$.
\end{thm}

For example, in the ring
$\HH^*(\IG(4,10))$ (where $k=1$) we have
\[
\s_{321} =
\frac{1-R_{12}}{1+R_{12}}(1-R_{13})(1-R_{23})\, m_{321} =
(1-2R_{12}+2R_{12}^2 - 2 R_{12}^3)(1-R_{13}-R_{23}) \, m_{321}
\]
\[
= m_{321} - 2m_{411}+m_{42} + 2m_{51} -m_{33} =
\s_3\s_2\s_1 - 2\s_4\s_1^2 + \s_4\s_2 +2\s_5\s_1 - \s_3^2.
\]
Furthermore, the theorem implies that if the $k$-strict partition
$\la$ satisfies $\la_i+\la_j \leq 2k+j-i$ for all $i<j$, then 
equation (\ref{giamS}) is valid, while if $\la_i+\la_j >
2k+j-i$ for $i<j\leq \ell(\la)$, then equations 
(\ref{giamQ0}) and (\ref{giamQ}) hold.

Our proof of Theorem \ref{mainthm} proceeds by showing directly that
the expression $R^{\la} m_{\la}$ satisfies the Pieri rule for
isotropic Grassmannians from \cite{BKT}.  This is sufficient because
the Pieri rule can be used recursively to show that a general Schubert
class may be written as a polynomial in the special Schubert classes.
The argument is challenging because the operator $R^\la$ depends on
$\la$, in contrast to the fixed raising operator expressions in
(\ref{spgiambelli}).  We remark that the equations corresponding to
(\ref{spgiambelli}) for the Schur $S$- and $Q$-functions may be
deduced from the formal identities
\begin{equation}
\label{E:formal}
\det(x_i^{\ell-j}) = \prod_{i<j}
(x_i-x_j) \ \ \ \text{and} \ \ \
\Pf\left(\frac{x_i-x_j}{x_i+x_j}\right) =
\prod_{i<j}\frac{x_i-x_j}{x_i+x_j}
\end{equation}
due to Vandermonde and Schur, respectively (see e.g.\ \cite[I.3 and
III.8]{M}).

We next use raising operators to define a family of polynomials
$\{\Theta_\la\}$ indexed by $k$-strict partitions whose algebra is the
same as the Schubert calculus in the stable cohomology ring of
$\IG$. Fix an integer $k \geq 0$ and consider a finite set of
variables $y=(y_1,\ldots,y_k)$.  For any $r\gequ 0$, define
$\vartheta_r$ by the equation
\[
\prod_{i=1}^{\infty}\frac{1+x_it}{1-x_it}\prod_{j=1}^k(1+y_jt) =
\sum_{r=0}^{\infty}\vartheta_r(x\,;y)t^r,
\]
so that $\vartheta_r(x\,;y) = \sum_i q_{r-i}(x)e_i(y)$.  We call
$\Gamma^{(k)} := \Z[\vartheta_1,\vartheta_2,\ldots]$ the ring of {\em
theta polynomials}.  For any finite integer sequence $\alpha$, let
$\vartheta_{\alpha} = \prod_i\vartheta_{\alpha_i}$, and for any
$k$-strict partition $\la$, define the theta polynomial
\[\Theta_\la := R^{\la}\vartheta_{\la}.\] 
When $k=0$, we have that $\Ti_\la(x\,;y)=Q_\la(x)$ is a Schur
$Q$-function.  As a first application of Theorem \ref{mainthm}, we
obtain the next two results. The first implies that the algebra of
theta polynomials agrees with the Schubert calculus on isotropic
Grassmannians $\IG(n-k,2n)$ when $n$ is sufficiently large.

\begin{thm}
\label{productthm}
The $\Theta_\la$, for $\la$ $k$-strict, form a $\Z$-basis of
$\Gamma^{(k)}$. There is a surjective ring homomorphism 
$\Gamma^{(k)} \to \HH^*(\IG(n-k,2n))$ such that  
$\Theta_\la$ is mapped to $\sigma_\la$, if $\la$ fits inside an
$(n-k)\times (n+k)$ rectangle, and to zero, otherwise.
\end{thm}

\begin{thm}
\label{thcor}
Let $\la$ be a $k$-strict partition.

\medskip
\noindent
{\em(a)} If $\la_i + \la_j \lequ 2k+j-i$ for all $i<j$, then
\[
\Theta_\la(x\,;y) = \sum_{\mu\subset\la}S_\mu(x)s_{\la'/\mu'}(y),
\ \ \ \text{where} \ \ \
S_\mu(x) = \det(q_{\mu_i+j-i}(x)).
\]

\noindent
{\em(b)} If $\la_i + \la_j > 2k+j-i$ for all $i<j\lequ \ell(\la)$, then
\[
\Theta_\la(x\,;y) = \sum_{\mu\subset\la}
Q_\mu(x) s_{\Sh(\la/\mu)'}(y),
\]
where the sum is over all strict partitions $\mu\subset\la$ such that
$\ell(\mu) \geq \ell(\la)-1$, and $\Sh(\la/\mu)$ denotes a shifted
skew diagram.
\end{thm}

In general, we prove that the theta polynomial $\Theta_\la(x\,;y)$ is
equal to the type C Schubert polynomial $\CS_{w_\la}(x, y)$ of
Billey and Haiman \cite{BH} indexed by the corresponding Grassmannian
element $w_\la$ of the hyperoctahedral group (Proposition
\ref{bhprop}). Our Giambelli formula may therefore be used to further
understand these and related polynomials. For instance, it follows
that the type C Stanley symmetric function $F_{w_\la}(x)$ of \cite{BH,
FK, L} is equal to $R^{\la}q_\la(x)$ (Corollary
\ref{endcor}).  Moreover, this connection implies that the coefficients
of $\Ti_\la(x\,;y)$ are nonnegative integers. These integers have
several combinatorial interpretations; the one we provide stems from
the work of Kra\'skiewicz \cite{Kr} and Lam \cite{L}.

\begin{thm}\label{bhthm}
  For any $k$-strict partition $\la$, the polynomial $\Theta_\la$ is a
  linear combination of products of Schur $Q$-functions and
  $S$-polynomials:
  \[
  \Theta_\la(x\,;y) = \sum_{\mu,\nu}
  e_{\mu\nu}^\la
  Q_\mu(x)s_{\nu'}(y)
  \]
  where the sum is over partitions $\mu$ and $\nu$ such that
  $\mu$ is strict and $\nu\subset \la$ with $\nu_1\leq k$.
  Moreover, the coefficients $e_{\mu\nu}^\la$ are nonnegative
  integers, equal to the number of Kra\'skiewicz tableaux 
  for $w_{\la}w_{\nu}^{-1}$ of shape $\mu$.
\end{thm}
\noindent
The definition of Kra\'skiewicz tableaux is recalled in \S\ref{BH}. In
\cite{T2}, an approach to tableau formulas via raising operators is
applied to obtain a different expression for $\Theta_\la(x\,;y)$,
which writes it as a sum of monomials $2^{n(U)}(xy)^U$ over all
`$k$-bitableaux' $U$ of shape $\la$.

We have described the theory here in the symplectic case, but there
are entirely analogous results for the odd orthogonal groups.  In
fact, for technical reasons, our proof of Theorem \ref{mainthm} is
obtained in the setting of orthogonal type B.  We also have analogues
of these Giambelli formulas for the quantum cohomology rings of
symplectic and odd orthogonal Grassmannians; this application will
appear in \cite{BKT2}.  In a sequel to this paper, we will discuss the
Giambelli formula for even orthogonal Grassmannians, which is more
involved.  Our results have been used in \cite{T3} to obtain a
Giambelli formula that expresses the equivariant Schubert classes on
any isotropic partial flag variety as polynomials in the special
Schubert classes.

This article is organized as follows.  The proof of Theorem
\ref{mainthm} occupies \S\ref{prelims}--\S\ref{mainpf}.  Section
\ref{thetasec} develops the theory of theta polynomials in a manner
parallel to the theory of Schur $Q$-functions, and contains our proofs
of Theorems \ref{productthm} and \ref{thcor}. Finally, in \S \ref{BH}
we show that theta polynomials are equal to certain Billey-Haiman
Schubert polynomials for the hyperoctahedral group, and prove Theorem
\ref{bhthm}.

\section{Preliminary Results}
\label{prelims}

\subsection{}
The Schubert varieties in $\IG=\IG(n-k,2n)$ are indexed by
$k$-strict partitions $\lambda$ which are contained in an $(n-k) \times
(n+k)$ rectangle; we denote the set of all such partitions by
$\cP(k,n)$.  Consider the exact sequence of vector bundles over $\IG$
\[
0 \to \cS \to E \to \cQ \to 0,
\]
where $E$ denotes the trivial bundle of rank $2n$ and $\cS$ is
the tautological subbundle of rank $n-k$.  The
special Schubert class $\sigma_p$
is equal to the Chern class $c_p(\cQ)$.

The symplectic form on $E$
gives a pairing $\cS \otimes \cQ \to \cO_{\IG}$, which in
turn produces an injection $\cS \hookrightarrow \cQ^*$.  For
$r>k$ we therefore have
\[
c_{2r}(\cQ \oplus \cQ^*) =
c_{2r}(E/\cS \oplus \cQ^*) = c_{2r}(\cQ^*/\cS) = 0,
\]
which implies that the relations
\begin{equation}
\label{presrel}
\sigma_r^2 + 2\sum_{i=1}^{n+k-r}(-1)^i \sigma_{r+i}\sigma_{r-i}= 0
\ \ \text{for} \ r > k
\end{equation}
hold in $\HH^*(\IG,\Z)$.

\subsection{}\label{S:delta0}

A {\em composition} $\al = (\al_1,\al_2,\dots,\al_r)$ is a vector of
integers from the set $\N=\{0,1,2,\ldots\}$; we let $|\al| = \sum
\al_i$.  For $\la$ any sequence of (possibly negative) integers, we
say that $\la$ has length $\ell$ if $\la_i=0$ for all $i>\ell$ and
$\ell\geq 0$ is the smallest number with this property.  All integer
sequences in this paper have finite length, and we will identify any
integer sequence of length $\ell$ with the vector consisting of its
first $\ell$ entries.  In analogy with Young diagrams of partitions,
we will say that a pair $[i,j]$ is a {\em box\/} of the integer
sequence $\lambda$ if $i \geq 1$ and $1 \leq j \leq \lambda_i$.

Let $\Delta^\circ = \{(i,j) \in \N \times \N \mid 1\leq i<j \}$ and
define a partial order on $\Delta^\circ$ by agreeing that $(i',j')\leq
(i,j)$ if $i'\leq i$ and $j'\leq j$.  We call a finite subset $D$ of
$\Delta^\circ$ a {\em valid set of pairs} if it is an order ideal,
i.e., $(i,j)\in D$ implies $(i',j')\in D$ for all $(i',j')\in
\Delta^\circ$ with $(i',j') \leq (i,j)$.

Any valid set of pairs $D$ defines the raising operator
\[
R^D = \prod_{i<j}(1-R_{ij})\prod_{i<j\, :\, (i,j)\in D}(1+R_{ij})^{-1}.
\]
Given a composition $\al$ and an integer $\ell > 0$, we denote by
$m(D,\al,\ell)$ the number of non-zero coordinates $\al_i$ such that
$(i,\ell) \in D$.  We say that $\al$ is {\em $(D,\ell)$-compatible} if
$\al_i\in\{0,1\}$ whenever $(i,\ell)\notin D$.

\begin{defn}
\label{recursedef}
For any valid set of pairs $D$ and any integer sequence $\la$ of length
$\ell$ we define a cohomology class $T_{\la}=T(D,\la)$
recursively as follows.  Set $T_p=\s_p$, and for
an arbitrary integer sequence $\mu=(\mu_1,\ldots,\mu_{\ell-1})$
and $r\in \Z$, set
\begin{equation}
\label{Tdef}
  T_{\mu,r} = \sum_\al (-1)^{|\al|} 2^{m(D,\al,\ell)}
  T_{\mu+\al} T_{r-|\al|} \,,
\end{equation}
where the sum is over all $(D,\ell)$-compatible vectors $\al \in
\N^{\ell-1}$.
\end{defn}

The sum (\ref{Tdef}) is well defined because only finitely many of its
summands are non-zero; we also have $T_{\mu,r}=0$ if $r<0$.  Notice
that definition (\ref{Tdef}) of $T(D,\lambda)$ is equivalent to
expanding the raising operator formula
\[
 R^D m_{\la} =
\prod_{i<j<\ell}(1-R_{ij}) \prod_{i<j<\ell\, :\, (i,j)\in
  D}(1+R_{ij})^{-1} \prod_{i=1}^{\ell-1}(1-R_{i\ell}) \prod_{i\, :\,
  (i,\ell)\in D}(1+R_{i\ell})^{-1} m_{\mu,r}
\]
after the last (i.e., the $\ell$-th) entry of $\lambda=(\mu,r)$.
Therefore $T_\la =  R^D m_{\la}$.

\subsection{}
If $D = \emptyset$ then for any integers $r$ and $s$ we have
\[
 T_{r,s} = T_r T_s - T_{r+1} T_{s-1}
\]
and so $T_{r,r+1} = 0$, while more generally $T_{r,s} = - T_{s-1, r+1}$.

We claim that if $D\neq \emptyset$ and $r,s\in \Z$ are
such that $r+s > 2k$, then $T_{s,r}=-T_{r,s}$;
in particular $T_{r,r}=0$ whenever $r > k$.  Indeed, from the definition
we obtain
\[
T_{r,s}=\s_r\s_s-2\,\s_{r+1}\s_{s-1}+2\,\s_{r+2}\s_{s-2}-\cdots
\]
and hence $T_{s,r}=-T_{r,s}$ whenever $r+s$ is odd.  If $r+s = 2m > 2k$
is even, we see that
\begin{equation}
\label{smallcommuteC}
T_{r,s} + T_{s,r} = (-1)^{\frac{r-s}{2}}\, 2\, (\s_m^2 -
2\,\s_{m+1}\s_{m-1} + 2\,\s_{m+2}\s_{m-2}- \cdots) = 0
\end{equation}
using the relations (\ref{presrel}) in the cohomology ring of $\IG$.

The previous observations are generalized in the next two lemmas.

\begin{lemma}\label{commuteA}
Let $\lambda=(\lambda_1,\ldots,\lambda_{j-1})$ and
$\mu=(\mu_{j+2},\ldots,\mu_\ell)$ be integer vectors.  Assume that
$(j,j+1)\notin D$ and that for each $h<j$, $(h,j)\in D$ if and only if
$(h,j+1)\in D$.  Then for any integers $r$ and $s$ we have
\[ T_{\lambda,r,s,\mu} = - T_{\lambda,s-1,r+1,\mu} \,.  \]
In particular, $T_{\lambda, r, r+1, \mu}=0$.
\end{lemma}
\begin{proof}
If $\mu = (\tau,t)$ has positive length, we set $\rho =
(\lambda,r,s,\tau)$ and the identity follows by induction, because
\[
  T_{\rho,t}
  = \sum_\alpha (-1)^{|\alpha|}2^{m(D,\alpha,\ell)} T_{\rho+\alpha}
  T_{t-|\alpha|}.
\]
Therefore, we may assume that $\mu$ is empty.  Set $\ell=j+1$.  Then
we have
\[\begin{split}
T_{\lambda,r,s} &= \sum_{\alpha} (-1)^{|\alpha|}\,
2^{m(D,\alpha,\ell)}\, T_{\lambda+\alpha,r} T_{s-|\alpha|} -
\sum_{\alpha} (-1)^{|\alpha|}\, 2^{m(D,\alpha,\ell)}\,
T_{\lambda+\alpha,r+1} T_{s-|\alpha|-1} \\ &= \sum_{\alpha,\beta}
(-1)^{|\alpha|+|\beta|}\, 2^{m(D,\alpha,\ell)+ m(D,\beta,\ell-1)}\,
T_{\lambda+\alpha+\beta} T_{r-|\beta|} T_{s-|\alpha|} \\ & \ \ \ \ -
\sum_{\alpha,\beta} (-1)^{|\alpha|+|\beta|}\, 2^{m(D,\alpha,\ell)+
m(D,\beta,\ell-1)}\, T_{\lambda+\alpha+\beta} T_{r+1-|\beta|}
T_{s-1-|\alpha|}
\end{split}\]
where the sums are over all $(D,\ell)$-compatible sequences $\alpha
\in \N^{j-1}$ and $(D,\ell-1)$-compatible sequences
$\beta\in\N^{j-1}$.  The assumptions on $D$ imply that these two sets
of sequences coincide, and this proves the lemma.
\end{proof}

\begin{lemma}\label{commuteC}
Let $\lambda=(\lambda_1,\ldots,\lambda_{j-1})$ and
$\mu=(\mu_{j+2},\ldots,\mu_\ell)$ be integer vectors, assume
$(j,j+1)\in D$, and that for each $h>j+1$, $(j,h)\in D$ if and only if
$(j+1,h)\in D$.  If $r,s\in \Z$ are such that $r+s > 2k$, then we have
\[ T_{\lambda,r,s,\mu} = - T_{\lambda,s,r,\mu} \,.  \]
In particular, $T_{\lambda,r,r,\mu} = 0$ for any $r>k$.
\end{lemma}
\begin{proof}
If $\mu = (\tau,t)$ has positive length, we set $\rho =
(\lambda,r,s,\tau)$ and $\rho' = (\lambda,s,r,\tau)$, and the identity
follows by induction, because
\[ T_{\rho,t}
  = \sum_\alpha (-1)^{|\alpha|}2^{m(D,\alpha,\ell)} T_{\rho+\alpha}
  T_{t-|\alpha|} = - \sum_\alpha (-1)^{|\alpha|}2^{m(D,\alpha,\ell)}
  T_{\rho'+\alpha} T_{t-|\alpha|} = - T_{\rho',t} \,.
\]
Thus we may assume that $\mu$ is empty.  Set $\ell=j+1$, and note
that $(h,h')\in D$ for all $h<h'\leq \ell$.  If $m>0$ is the least
integer such that $2m \geq \ell$, we claim that
$T_\rho=T_{\lambda,r,s}$ satisfies the relation
\begin{equation}
\label{pfaff1}
T_{\rho} = \sum_{i=2}^{2m} (-1)^i\,
T_{\rho_1,\rho_i}\,T_{\rho_2,\ldots,\wh{\rho_i},\ldots,\rho_{2m}}.
\end{equation}
Equation (\ref{pfaff1}) follows from the formal identity of
raising operators
\[
\prod_{1\leq h<h'\leq 2m}\frac{1-R_{hh'}}{1+R_{hh'}} =
\sum_{i=2}^{2m} (-1)^i\,
\frac{1-R_{1i}}{1+R_{1i}}\,
\prod_{\substack{2\leq h< h' \leq 2m\\h\neq i\neq h'}}
\frac{1-R_{hh'}}{1+R_{hh'}},
\]
which is equivalent the classical formula
\[
\prod_{1\leq h<h'\leq 2m}\frac{x_h-x_{h'}}{x_h+x_{h'}} =
\Pf\left(\frac{x_h-x_{h'}}{x_h+x_{h'}}\right)_{1\leq h,h' \leq 2m}
\]
due to Schur \cite[Sec.\ IX]{S2}.  The proof is completed using
induction, starting from the base case of $j=1$, which was obtained in
(\ref{smallcommuteC}).
\end{proof}

During the above discussion the set $D$ has remained fixed, but in
subsequent arguments we will need to modify it.  For this, we
use a simple observation.

\begin{lemma}
\label{easylm}
If $(i,j)\notin D$ and $D\cup (i,j)$ is a valid set of pairs, then
\[
T(D, \lambda) = T(D\cup (i,j), \lambda) +
T(D\cup (i,j), R_{ij}\lambda).
\]
\end{lemma}
\begin{proof}
The assertion follows immediately from the identity
\[
1-R_{ij} = \frac{1-R_{ij}}{1+R_{ij}} + \frac{1-R_{ij}}{1+R_{ij}}\,R_{ij}.
\qedhere
\]
\end{proof}

\section{From $\IG(n-k,2n)$ to $\OG(n-k,2n+1)$}
\label{igtoog}

\subsection{}
For each $k\geq 0$, the odd orthogonal Grassmannian
$\OG=\OG(n-k,2n+1)$ parametrizes the $(n-k)$-dimensional isotropic
subspaces in $\C^{2n+1}$, equipped with a nondegenerate symmetric
bilinear form.  Our aim is to show that if $\lambda$ is any $k$-strict
partition, then $\sigma_\lambda$ is given by the raising operator
expression of Theorem \ref{mainthm}.
For technical reasons, we will use an isomorphism to transfer this
relation to the cohomology ring of $\OG$, and work with the latter
space.

The Schubert varieties in $\OG$ are indexed by the same set of
$k$-strict partitions $\cP(k,n)$ as for $\IG(n-k,2n)$.  Given a
complete flag $F_\bull$ of subspaces of $\C^{2n+1}$ such that
$F_{n+i}= F_{n+1-i}^\perp$ for $1 \leq i \leq n+1$ and
$\la\in\cP(k,n)$, we define the codimension $|\la|$ Schubert variety
\[
   X_\lambda(F_\bull) = \{ \Sigma \in \OG \mid \dim(\Sigma \cap
   F_{\ov{p}_j(\lambda)}) \gequ j \ \ \forall\, 1 \lequ j \lequ
   \ell(\lambda) \} \,,
\]
where
\begin{equation}
\label{pdefB}
\ov{p}_j(\lambda) = n+k+1+j-\lambda_j - \#\{i\leq j : \lambda_i+\lambda_j
> 2k+j-i \}.
\end{equation}
Let $\ta_{\la} \in
\HH^{2|\lambda|}(\OG,\Z)$  be the cohomology
class dual to the cycle given by $X_\lambda(F_\bull)$.

For any $\la\in\cP(k,n)$, let $\ell_k(\la)$ be the number of parts
$\la_i$ which are strictly greater than $k$.  Let $\cQ_{\IG}$ and
$\cQ_{\OG}$ be the universal quotient vector bundles over
$\IG(n-k,2n)$ and $\OG(n-k,2n+1)$, respectively.  It is known (see
e.g.\ \cite[\S 3.1]{BS}) that the map which sends $\s_p =
c_p(\cQ_{\IG})$ to $c_p(\cQ_{\OG})$ for all $p$ extends to a ring
isomorphism $\phi:\HH^*(\IG,\Q) \to \HH^*(\OG,\Q)$.  Moreover, we have
$\phi(\s_\la) = 2^{\ell_k(\la)} \ta_{\la}$ for all $\la\in \cP(k,n)$.

We let $c_p= c_p(\cQ_{\OG})$.
The Chern classes $c_p$ are related to the special Schubert classes
$\ta_p$ on $\OG$ by the equations
\[
c_p=
\begin{cases}
\ta_p & \text{if $p\lequ k$},\\
2\ta_p & \text{if $p> k$}.
\end{cases}
\]
Using the isomorphism $\phi$, we can therefore describe the Giambelli
formula for $\OG(n-k,2n+1)$ as follows.  For any integer sequence
$\alpha$, set $m_{\alpha} = \prod_ic_{\al_i}$; then for
every $\la\in\cP(k,n)$, we have
\begin{equation}
\label{mainthmB}
\ta_\la = 2^{-\ell_k(\la)}R^{\la}\,m_{\la}
\end{equation}
in $\HH^*(\OG,\Z)$.

\subsection{}
\label{classpieri}
For $\lambda$ any $k$-strict partition, we say that the box $[r,c]$ in
row $r$ and column $c$ of $\lambda$ is {\em $k$-related\/} to the box
$[r',c']$ if $|c-k-1|+r = |c'-k-1|+r'$.  If $c \leq k < c'$, then this
is equivalent to $c+c' = 2k+2+r-r'$.  For example, in the partition
displayed below, the grey box $[r,c]$ is $k$-related to $[r',c']$.
The notion of $k$-related boxes makes sense also for boxes outside the
Young diagram of $\la$.
\[ \pic{0.65}{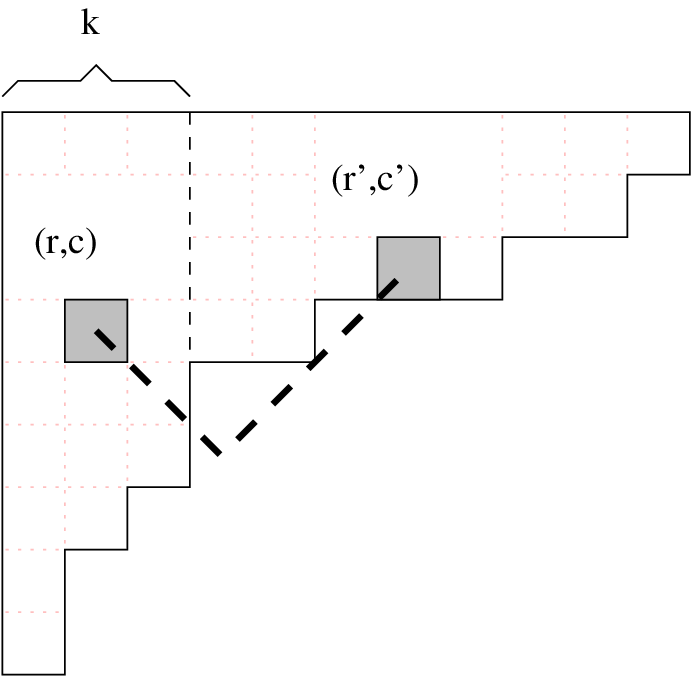} \]

Given two Young diagrams $\mu$ and $\nu$ with $\mu\subset\nu$, the
skew diagram $\nu/\mu$ is called a horizontal (resp.\ vertical) strip
if it does not contain two boxes in the same column (resp.\ row).  For
any two $k$-strict partitions $\lambda$ and $\mu$, we write $\lambda
\to \mu$ if $\mu$ may be obtained by removing a vertical strip from
the first $k$ columns of $\lambda$ and adding a horizontal strip to
the result, so that

\medskip
\noin (1) if one of the first $k$ columns of $\mu$ has the same number
of boxes as the same column of $\lambda$, then the bottom box of this
column is $k$-related to at most one box of $\mu \smallsetminus
\lambda$; and

\medskip
\noin
(2) if a column of $\mu$ has fewer boxes than the same column of
$\lambda$, then the removed boxes and the bottom box of $\mu$ in this
column must each be $k$-related to exactly one box of $\mu
\smallsetminus \lambda$, and these boxes of $\mu \smallsetminus
\lambda$ must all lie in the same row.

\medskip

Equivalently, $\la\to\mu$ means that $\la_j-1\leq \mu_j \leq
\la_{j-1}$ for each $j$, $\la_j\leq \mu_j$ when $\la_j>k$, and
conditions (1) and (2) are true.  Let $\A$ be the set of boxes of
$\mu\ssm \la$ in columns $k+1$ through $k+n$ which are not mentioned
in (1) or (2), and define $\NN(\lambda,\mu)$ to be the number of
connected components of $\A$.  Here two boxes are connected if they
share at least a vertex.  In \cite[Theorem 2.1]{BKT} we proved that the
Pieri rule
\begin{equation}
\label{bktpieri}
c_p \cdot \ta_\lambda = \sum_{\lambda \to \mu, \,
  |\mu|=|\lambda|+p} 2^{\NN(\lambda,\mu)} \, \ta_\mu
\end{equation}
holds, for any $p\in [1,n+k]$.

\subsection{}
A comparison of (\ref{pdefC}) with (\ref{pdefB}) suggests modifying
the definition of valid sets of pairs from \S\ref{prelims} to
include elements along the diagonal $\{(i,i)\ |\ i>0\}$.  This
convention will make the formalism of our proof of Theorem
\ref{mainthm} cleaner, and is in fact crucial in the corresponding
proof of Giambelli for even orthogonal Grassmannians.

Set $\Delta = \{(i,j) \in \N\times \N \mid 1\leq i\leq j \}$ with the
same partial order as in \S\ref{S:delta0}, and define the notion of a
valid set of pairs exactly as before.
Given a $k$-strict partition
$\la$ and an integer $t \geq \ell(\la)$, we obtain a valid set of
pairs $\cC_t(\la)$ by
\[
\cC_t(\la) = \{(i,j)\in\Delta\ |\ \la_i+\la_j > 2k+j-i
\ \, \text{and} \ \, j \leq t\}.
\]
Furthermore, we let $\cC(\la)=\cC_{\ell(\la)}(\la)$.  Notice that $\cC(\la)$
includes the pairs $(i,i)$ such that $\la_i>k$.  It is easy to see
that when $k>0$, a set $D\subset\Delta$ is a valid set of pairs
if and only if there exists a $k$-strict partition $\la$ for which
$\cC(\la)=D$.

An {\em outer corner\/} of a valid set of pairs $D \subset
\Delta$ is a pair $(i,j) \in \Delta \ssm D$ such that $D \cup (i,j)$
is also a valid set of pairs.  The {\em outside rim} $\partial D$ of
$D$ is the set of pairs $(i,j)\in\Delta\ssm D$ such that $i=1$ or
$(i-1,j-1) \in D$.

\begin{lemma}\label{lem:CinDinCDC}
  Let $\mu$ be a $k$-strict partition such that $\la \to \mu$.  Then
  for any $t \geq \ell(\la)$, we have $\cC_t(\la) \subset
  \cC_{t+1}(\mu) \subset \cC_t(\la) \cup
  \partial\,\cC_t(\la)$.
\end{lemma}
\begin{proof}
  If $(i,j) \in \cC_{t+1}(\mu)$, then $\la_{i-1} + \la_{j-1} \geq
  \mu_i + \mu_j > 2k+j-i$.  This proves that $\cC_{t+1}(\mu)\subset
  \cC_t(\la) \cup \partial \cC_t(\la)$.  If there exists a pair $(i,j)
  \in \cC_t(\la) \ssm \cC_{t+1}(\mu)$, then $\la_i + \la_j > 2k+j-i
  \geq \mu_i+\mu_j$, so we must have $\mu_i = \la_i$, $\mu_j =
  \la_j-1$, and $\la_i+\la_j = 2k+1+j-i$.  Condition (2) of
  \S\ref{classpieri} implies that some box $[h,c]$ of $\mu\ssm\la$ is
  $k$-related to $[j,\la_j]$, and $[h,c-1]$ is also in $\mu\ssm\la$
  since this box is $k$-related to $[j-1,\la_j]$.  The equality
  $\la_j+c = 2k+2+j-h$ implies that $(h,j) \in \cC_{t+1}(\mu)$, and
  since $\cC_{t+1}(\mu)$ is a valid set of pairs, we must have $h <
  i$.  But we also obtain $\la_h < c-1 = 2k+1+j-h-\la_j = \la_i+i-h$,
  contradicting the fact that $\lambda$ is $k$-strict.  This proves
  that $\cC_t(\la) \subset \cC_{t+1}(\mu)$.
\end{proof}

\begin{defn}
\label{recursedefB}
For any valid set of pairs $D\subset \Delta$ and any integer
sequence $\la$ we define the cohomology class $T(D,\la) \in \HH^*(\OG)$ by
\[
T(D,\la) = 2^{-\#\{i\,| \,(i,i)\in D\}} \,\phi(T(D\cap\Delta^\circ,\la)),
\]
where $T(D\cap\Delta^\circ,\la) \in \HH^*(\IG)$ is defined by
(\ref{Tdef}).
\end{defn}

To prove (\ref{mainthmB}) and hence also establish
Theorem \ref{mainthm}, it suffices to show that if $\lambda$ is a
$k$-strict partition, the Pieri rule
\begin{equation}
\label{pieriTla}
c_p \cdot T(\cC(\la),\lambda) = \sum_{\lambda \to \mu, \,
|\mu|=|\lambda|+p} 2^{\NN(\lambda,\mu)} \, T(\cC(\mu),\mu) \,
\end{equation}
holds in $\HH^*(\OG,\Z)$, for all $p$.  To see this, write $\mu \succ \la$
if $\mu$ strictly dominates $\lambda$, i.e., $\mu \neq \lambda$ and
$\mu_1 + \dots + \mu_i \geq \lambda_1 + \dots + \lambda_i$ for each $i
\geq 1$.  We deduce from (\ref{bktpieri}) and (\ref{pieriTla}) that
\[
2^{\ell_k(\la)}\,\ta_\la + \sum_{\mu \succ \la} a_{\la\mu}\,\ta_\mu =
c_{\la_1}\cdots c_{\la_\ell} =
2^{\ell_k(\la)}\,T(\cC(\la),\la) +
\sum_{\mu \succ \la} a_{\la\mu}\, T(\cC(\mu),\mu),
\]
for some constants $a_{\la\mu}\in\Z$. By induction on $\la$, it
follows that $T_\la=T(\cC(\la),\la)$, which is a restatement of
(\ref{mainthmB}).

Observe that Lemmas \ref{commuteA}, \ref{commuteC}, and \ref{easylm}
carry over verbatim to our current setting where $D\subset
\Delta$.  These lemmas are the main properties of the cohomology
classes $T(D,\lambda)$ that we use, and as such constitute the
technical core of our proof of Theorem \ref{mainthm}.  But the
non-trivial scheme that puts them to work together is an algorithm with
a substitution rule; this is explained in the next section.

\section{The Substitution Rule}
\label{subrule}

\subsection{}
\label{initnot}
Throughout the next two sections we fix $p>0$, the $k$-strict
partition $\la$ of length $\ell$, and choose $n$ sufficiently large so
that we can ignore it in the sequel.  Set $\cC=\cC(\la)$.  For any
$d\geq 1$ define the raising operator $R^{\la}_d$ by
\[
R_d^{\la} = \prod_{1\leq i<j\leq d}(1-R_{ij})\,
\prod_{i<j\, :\, (i,j)\in\cC} (1+R_{ij})^{-1}.
\]
We compute that
\[
c_p\cdot T(\cC,\la) = c_p\cdot 2^{-\ell_k(\la)}R_{\ell}^{\la}\, m_{\la} =
2^{-\ell_k(\la)} R_{\ell+1}^\la\cdot
\prod_{i=1}^\ell(1-R_{i,\ell+1})^{-1} \, m_{\la,p}
\]
\[
= 2^{-\ell_k(\la)}
R^\la_{\ell+1}\cdot\prod_{i=1}^\ell(1+R_{i,\ell+1} +
R_{i,\ell+1}^2 + \cdots)\,m_{\la,p}
\]
and therefore
\begin{equation}\label{initeq}
  c_p\cdot T(\cC,\la) = \sum_{\nu\in \cN}T(\cC,\nu),
\end{equation}
where $\cN=\cN(\la,p)$ is the set of all compositions $\nu \geq \la$
such that $|\nu| = |\la|+p$ and $\nu_j =0$ for $j > \ell+1$.  Our
strategy for proving Theorem \ref{mainthm} is to show that the right
hand side of equation (\ref{initeq}) is equal to the right hand side
of the Pieri rule (\ref{pieriTla}).

\subsection{}\label{initialdefs}

The following objects will be used as book keeping tools in a
delicate process of rewriting the right hand side of (\ref{initeq}).
Let $m\geq 1$ be minimal such that $\la_m \leq k$; we call $m$ the
{\em middle} row of $\la$.  Notice that $m$ is the smallest positive
integer for which $(m,m)\notin \cC$.

\begin{defn}
  A {\em valid 4-tuple\/} of {\em level} $h$ is a 4-tuple $\psi =
  (D,\mu,S,h)$, such that $h$ is an integer with $0 \leq h \leq
  \ell+1$, $D$ is a valid set of pairs containing $\cC$, all pairs
  $(i,j)$ in $D$ satisfy $i\leq m$ and $j \leq \ell+1$, $S$ is a
  subset of $D \ssm \cC$, and $\mu$ is an integer sequence of length
  at most $\ell+1$.  The evaluation of $\psi$ is defined by $\ev(\psi)
  = T(D,\mu) \in \HH^*(\OG,\Q)$.
\end{defn}

All valid 4-tuples encountered in this paper will also satisfy that $D
\subset \cC \cup \partial\cC$ (see Lemma~\ref{lem:DinCDC}), but for
technical reasons we do not require this in the definition.  We will
represent the set $\Delta$ as the positions on or above the main
diagonal of a matrix, and the various sets of pairs $D$ as sets of
entries in this matrix.  In Figure~\ref{CrimandD} the white dots
represent a set of pairs $\cC$, the grey dots are a subset of the
outside rim of $\cC$, and we have $m=10$. The union of the white and
grey dots form the set $D$ in a typical valid $4$-tuple $(D,\mu,S,h)$.

\begin{figure}
\centering
\includegraphics[scale=0.5]{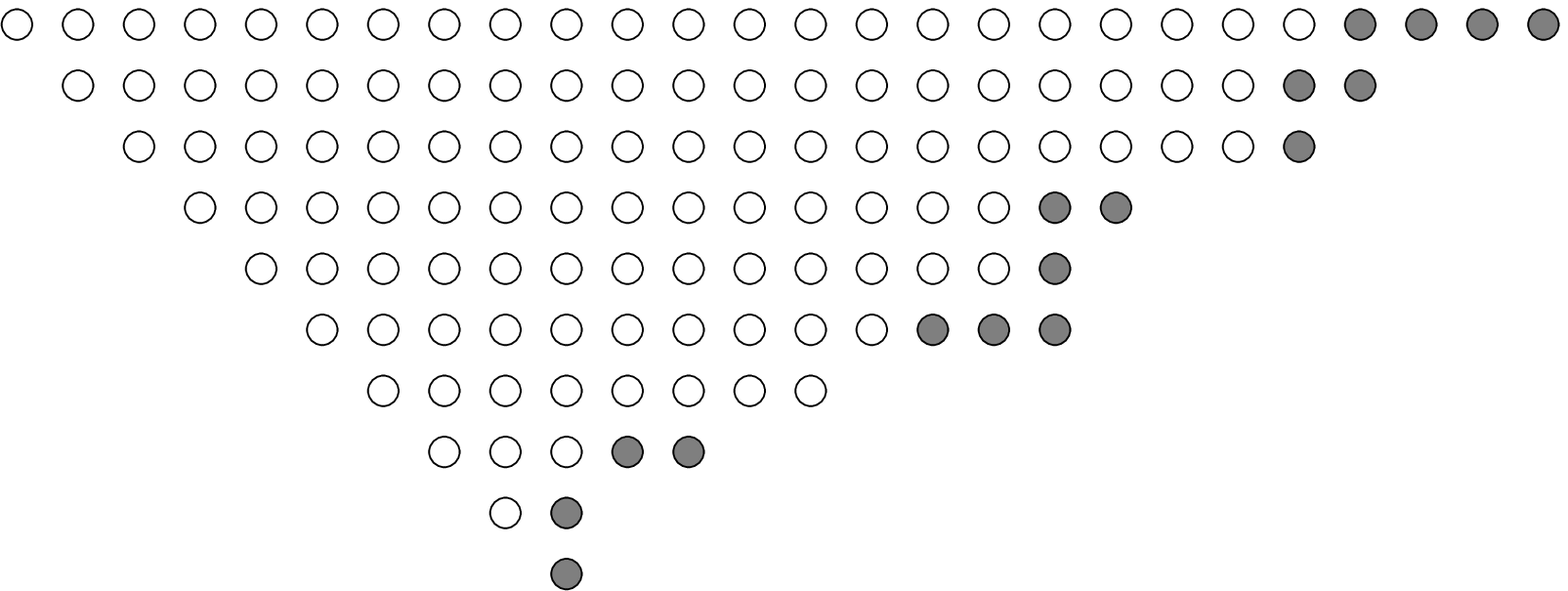}
\caption{A valid set of pairs $\cC$ (white dots) and a subset of
$\partial \cC$ (grey dots).}
\label{CrimandD}
\end{figure}

In the following we set $\mu_0 = \infty$ whenever $\mu$ is an integer
sequence.

\begin{defn}
  For any $y \in \Z$ we let $r(y)$ denote the largest integer such
  that $r(y) \leq \ell+1$ and $\la_{r(y)-1} > 2k+r(y)-y$.
\end{defn}

Since $\la_0 = \infty$ we have $r(y) \geq 1$.  Notice that for $(i,j)
\in \Delta$ and $j \leq \ell$ we have $(i,j) \in \cC \Leftrightarrow
\la_j > 2k+j-i-\la_i \Leftrightarrow j < r(i+\la_i+1)$.  This gives
the relation
\begin{equation}\label{eqn:Cfromr}
   \cC = \{ (i,j) \in \Delta \mid
   j < r(i+\la_i+1) \} \,.
\end{equation}

The function $r(y)$ is also connected to the notion of $k$-relatedness
of boxes.  Assume that some box $[i,c]$ with $c>k$ is $k$-related to a
box $[j,d]$ in the first $k$ columns of $\la$.  Then $\la_j \geq d =
2k+2+j-i-c$, which implies that $j < r(i+c)$.  Furthermore, if
$[j+1,d+1] \not\in \la$ then $r(i+c) = j+1$.

\begin{defn}\label{efg}
  Let $h\in \N$ satisfy $1\leq h \leq m$ and let $\mu$ be an integer
  sequence.\smallskip

  \noindent
  (a) We define $b_h = r(h+\la_h+1)$ and $g_h = b_{h-1}$.  By
  convention we set $g_1 = \ell+1$.  \smallskip

  \noindent
  (b) Set $R(\mu) = \{ [i,c] \in \mu\ssm\la \mid c>k \text{ and }
  \mu_{r(i+c)} \leq 2k+r(i+c)-i-c \}$.
  \smallskip

  \noindent
  (c) Assume that $h \geq 2$ and $\mu_h \geq \la_{h-1}$.  If
  $[h,\la_{h-1}] \in R(\mu)$ then set $e_h(\mu) = \la_{h-1}$.
  Otherwise choose $e_h(\mu) > \max\,\{k,\la_h\}$ minimal such that
  $[h,c] \not\in R(\mu)$ for $e_h(\mu) \leq c \leq \la_{h-1}$.
  Finally, set $f_h(\mu) = r(h+e_h(\mu))$.
\end{defn}

\begin{figure}\centering
\includegraphics[scale=0.6]{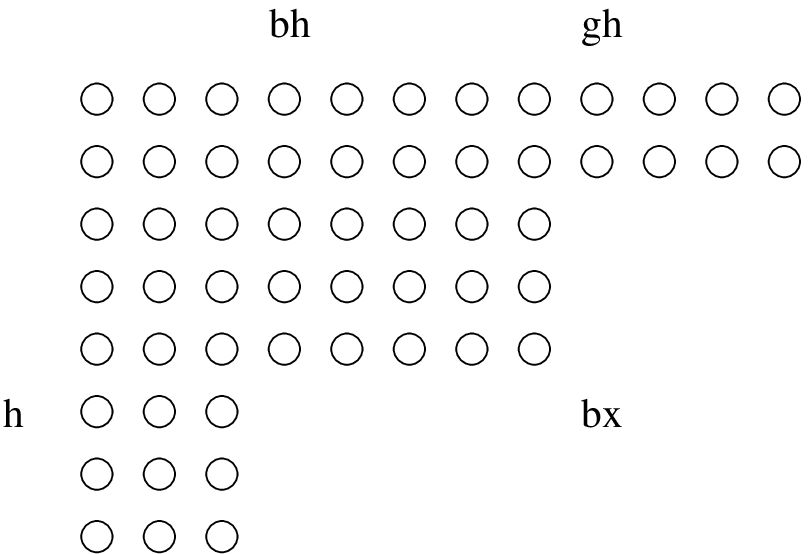}
\caption{The set $\cC$ near the pair ${\mathbf x} = (h,g_h)
  \in \partial \cC$.}
\label{abgh}
\end{figure}

Notice that for $h < m$ we have $b_h = \min \{j \geq m \mid (h,j)
\not\in \cC \}$.  The integers $b_h$ and $g_h$ are illustrated in
Figure~\ref{abgh}.  In the definition of $R(\mu)$, suppose that some
box $[i,c] \in \mu\ssm\la$ with $c>k$ is $k$-related to a box $[j,d]$
in the first $k$ columns of $\la$, such that $[j+1,d+1] \not\in \la$.
Then we have $r(i+c) = j+1$ and $d = 2k+2+j-i-c$.  It follows that
$[i,c] \in R(\mu)$ if and only if $\mu_{j+1} < d$.  In particular, if
$\mu$ is a $k$-strict partition such that $\la \to \mu$, then the set
$\A$ from \S\ref{classpieri} consists of the boxes of $\mu\ssm\la$ in
columns $k+1$ and higher which are not in $R(\mu)$.  Notice also that
since $h+\la_h+1 \leq h+e_h(\mu) \leq h+\la_{h-1}$ and $r(y)$ is a
monotone increasing function of $y$, we always have
\begin{equation}\label{eqn:bfg}
  b_h \leq f_h(\mu) \leq g_h \,.
\end{equation}
In particular we have $f_h(\mu) \geq m$ and $(h,f_h(\mu)) \notin \cC$;
when $h=m$ the inequality is true because $\la_{m-1} > 2k-e_h(\mu) =
2k+m-h-e_h(\mu)$.  Furthermore, if $[h,\la_{h-1}] \in R(\mu)$, then
$f_h(\mu) = g_h$.

\begin{example}
  Let $k=3$, $\la= (9,7,3,2,1,1)$, and $\mu = (11, 12, 7, 2, 2)$.
  Then
  \[
  \cC(\la) = \{(1,1),(1,2),(1,3),(1,4),(2,2),(2,3),(2,4) \}.
  \]
  Figure \ref{efgfig} illustrates $\la$ and $\mu$, with the boxes in
  $\mu\ssm\la$ shaded, and the boxes in $R(\mu)$ marked.  Note that
  there is one box in $\la\ssm\mu$.  We have $e_2(\mu)=8$,
  $f_2(\mu)=5$, $g_2=5$, $e_3(\mu)=6$, $f_3(\mu)=4$, and $g_3=5$.
  \begin{figure}
    \centering
    \includegraphics[scale=0.35,viewport=0 0 450 350]{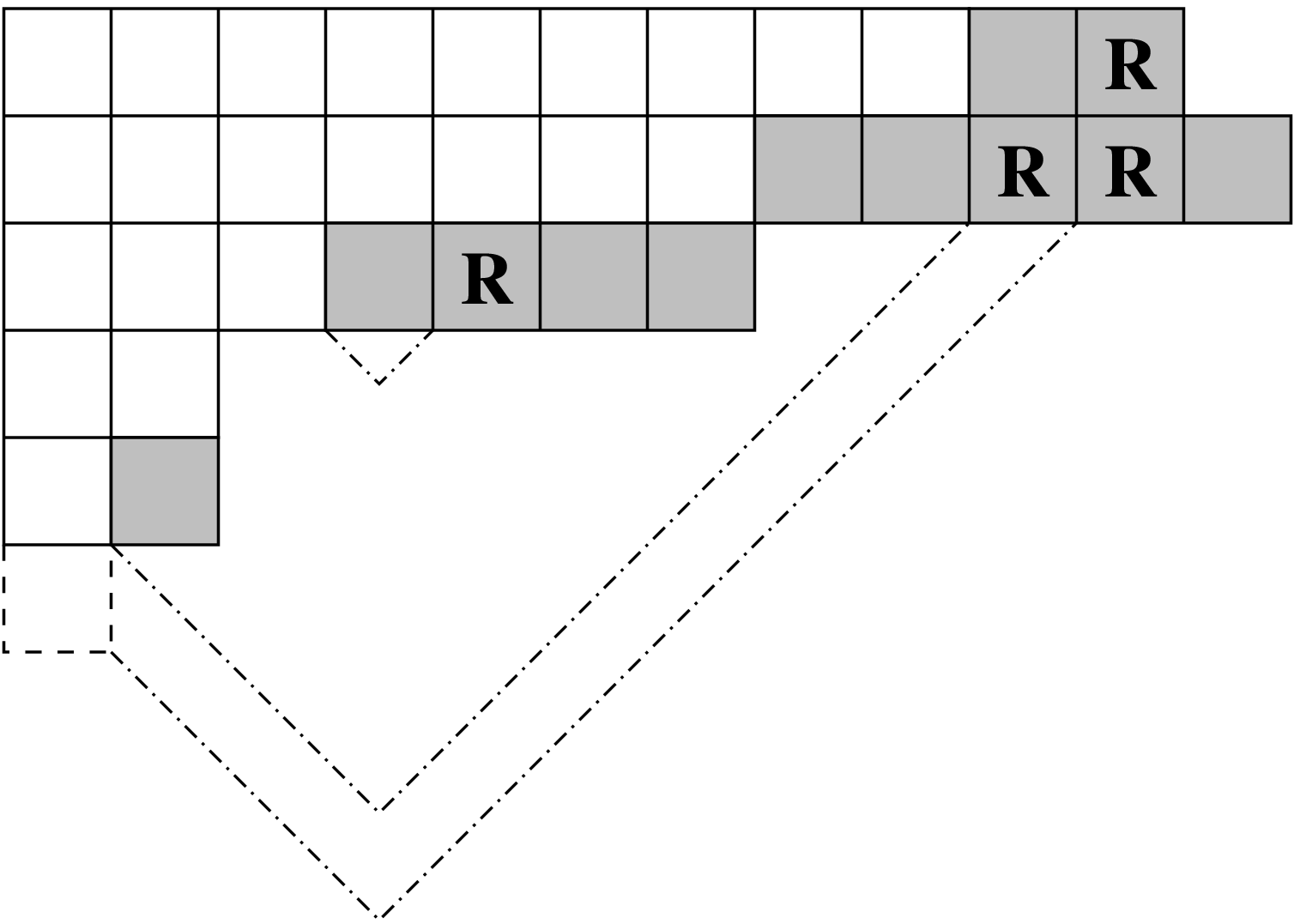}
    \caption{The shapes $\la$ and $\mu$, with $\mu\ssm\la$ shaded.}
    \label{efgfig}
  \end{figure}
\end{example}

\begin{lemma}\label{lambdabg}
  If $2\leq h\leq m$ then we have $\la_{h-1}-\la_h \geq g_h-b_h+1$.
\end{lemma}
\begin{proof}
  The inequality is clear if $b_h=g_h$, as $\la$ is $k$-strict and $h\leq
  m$.  If $b_h<g_h$, then since $b_h = r(h+\la_h+1)$ and $g_h =
  r(h+\la_{h-1})$ we have $\la_{b_h} \leq 2k+b_h-h-\la_h$ and $\la_{g_h-1} >
  2k+g_h-h-\la_{h-1}$, which implies that $\la_{h-1}-\la_h >
  g_h-\la_{g_h-1}-b_h+\la_{b_h} \geq g_h-b_h$.
\end{proof}

\begin{lemma}\label{lem:fisg}
  Let $\mu$ be an integer sequence and $2 \leq h \leq m$.  If $\mu_h
  \geq \la_{h-1}$ and $\la_{h-1} + \mu_{g_h} \leq 2k+g_h-h$, then
  $f_h(\mu)=g_h$.
\end{lemma}
\begin{proof}
  We have $\la_{h-1}>k$ and $[h,\la_{h-1}] \in \mu\ssm\la$.  Since
  $g_h = r(h+\la_{h-1})$, the inequality $\mu_{g_h} \leq 2k + g_h - h
  - \la_{h-1}$ shows that $[h,\la_{h-1}] \in R(\mu)$.  This implies
  that $e_h(\mu) = \la_{h-1}$ and $f_h(\mu) = r(h+\la_{h-1}) = g_h$,
  as required.
\end{proof}

\begin{lemma}\label{Rflemma}
  Let $2 \leq h \leq m$ and let $\mu$ and $\mu'$ be integer
  sequences such that $\mu_h \geq \la_{h-1}$, $\mu'_h \geq
  \la_{h-1}$, and $\mu_j = \mu'_j$ for $\max(m,h+1) \leq j \leq
  g_h$.  Then $[h,c]\in R(\mu)$ if and only if $[h,c]\in R(\mu')$
  for all $c\leq \la_{h-1}$.  In particular, we have
  $e_h(\mu)=e_h(\mu')$ and $f_h(\mu)=f_h(\mu')$.
\end{lemma}
\begin{proof}
  Let $[h,c] \in \mu\ssm\la$ satisfy $k<c\leq \la_{h-1}$, and set $j =
  r(h+c)$.  Since $k$-strictness of $\la$ implies that $k+m+1 \leq h+c
  \leq h+\la_{h-1}$, we obtain $m \leq j \leq g_h$ and hence $\mu_j =
  \mu'_j$ provided that $j>h$.  It follows that $[h,c]\in R(\mu)$ if
  and only if $[h,c] \in R(\mu')$, as required.
\end{proof}

If we are given a fixed valid 4-tuple $(D,\mu,S,h)$ with $1\leq h \leq
m$, we will use the shorthand notation $b = b_h$, $g=g_h$, $R =
R(\mu)$, $e = e_h(\mu)$, and $f = f_h(\mu)$; the values $e$ and $f$
will be used only when $\mu_h \geq \la_{h-1}$.  The precise value of
$f$ will play a crucial role in our proof that the Pieri terms in
(\ref{pieriTla}) appear in (\ref{initeq}) with the correct
multiplicities.  For example, it is part of the following definition
of a condition $\X$, that will be used to identify undesired valid
4-tuples.

\begin{defn}\label{wvdef}
  Let $(i,j)\in \Delta$ be arbitrary.
  We define two conditions $\W(i,j)$ and $\X$ on a valid $4$-tuple
  $(D,\mu,S,h)$ as follows.
  \[
  \W(i,j)\ :\ \mu_i+\mu_j > 2k + j-i \,.
  \]
  Condition $\X$ is true if and only if $(h,h) \in D$ and
  \[
  \mu_h \geq \mu_{h-1} \  {\mathrm{or}}  \ \mu_h > \la_{h-1} \
  {\mathrm{or}} \ (\mu_h = \la_{h-1} \ {\mathrm{and}}  \
  (h,\f)\notin S) \,.
  \]
\end{defn}

\subsection{}\label{ss:subrule}

The following {\em substitution rule\/} will be applied iteratively to
rewrite the right hand side of (\ref{initeq}).  It may be applied to
any valid 4-tuple of positive level and will result in either a
REPLACE statement, indicating that the 4-tuple should be replaced by
one or two new 4-tuples, or a STOP statement, indicating that the
4-tuple should not be replaced.

\medskip

\begin{center}
{\bf \underline{Substitution Rule}}
\end{center}

\medskip
\medskip

Let $(D,\mu,S,h)$ be a valid 4-tuple of level $h\geq 1$.  Assume first
that $(h,h)\notin D$.  If

\medskip
\begin{center}
{\bf (i)} there is an outer corner $(i,h)$ of $D$ with $i\leq m$
such that $\W(i,h)$ holds
\end{center}

\medskip
\noin
then REPLACE $(D,\mu,S,h)$ with
\[
(D \cup (i,h),\mu,S,h) \ \ \mathrm{and} \ \ (D \cup (i,h),R_{ih}\mu,
S\cup (i,h), h).
\]

\noin
Otherwise, if

\medskip
\begin{center}
{\bf (ii)} $D$ has no outer corner in column $h$ and $\mu_h > \la_{h-1}$,
\end{center}

\medskip
\noin
then STOP.

\medskip
Assume now that $(h,h)\in D$.  If

\medskip
\begin{center}
{\bf (iii)} there is an outer corner $(h,j)$ of $D$ with $j \leq \ell+1$
such that $\W(h,j)$ holds,
\end{center}

\medskip
\noin
then REPLACE $(D,\mu,S,h)$ with
\[
\begin{cases}
(D \cup (h,j),\mu,S,h)\ \ \mathrm{and} \ \ (D\cup (h,j),R_{hj}\mu, S\cup
(h,j),h) & \mathrm{if} \ \mu_j \leq \mu_{j-1}, \\
(D\cup (h,j),R_{hj}\mu, S\cup(h,j),h) & \mathrm{if} \
\mu_j > \mu_{j-1}.
\end{cases}
\]
Otherwise, if

\medskip
\begin{center}
  {\bf (iv)} $\W(h,\g)$ or $\X$ holds, and $D$ has an
  outer corner $(i,g)$ with $i \leq h$,
\end{center}

\medskip
\noin
then REPLACE $(D,\mu,S,h)$ with
\[
(D \cup (i,\g),\mu,S,h) \ \ \mathrm{and} \ \ (D \cup (i,\g),R_{i\g}\mu,
S\cup (i,\g), h).
\]
Otherwise, if

\medskip
\begin{center}
{\bf (v)}  $\X$ holds,
\end{center}

\medskip
\noin
then STOP.

\medskip

If  none of the above conditions hold, REPLACE
$(D,\mu,S,h)$ with $(D,\mu,S,h-1)$.

\medskip

\begin{defn}\label{meets}
  Let {\bf (x)} be one of the conditions {\bf (i)}--{\bf (v)} of the
  Substitution Rule.  We say that a valid $4$-tuple $\psi$ {\em meets}
  condition {\bf (x)} if $\psi$ reaches condition {\bf (x)} in the
  Substitution Rule, and condition {\bf (x)} is satisfied. Whenever
  the Substitution Rule REPLACES $\psi$ by one or two $4$-tuples
  $\psi_i$, we refer to $\psi$ as the {\em parent\/} term and the
  $\psi_i$ are its {\em children}.
\end{defn}

\subsection{}

Initially, we define the set $\Psi = \{(\cC,\nu,\emptyset,\ell+1) \mid
\nu\in \cN(\la,p)\}$; thus $\sum_{\psi\in\Psi}\ev(\psi)$ agrees with
the right hand side of (\ref{initeq}).  We then apply an {\em
  algorithm\/} which will change this set by replacing some $4$-tuples
with one or two new valid $4$-tuples.  The algorithm applies the
Substitution Rule to each element $(D,\mu,S,h)$ of level $h \geq 1$.
If the substitution rule results in a REPLACE statement, then the set
is changed accordingly; otherwise the substitution rule results in a
STOP statement, in which case the $4$-tuple $(D,\mu,S,h)$ is left
untouched.  These substitutions are iterated until no further elements
can be REPLACED, i.e., until the substitution rule results in a STOP
statement when applied to any remaining $4$-tuple with $h \geq 1$.

Since the set of pairs $D$ is not allowed to grow beyond column
$\ell+1$, the algorithm will terminate after a finite number of steps.
Notice that if $\psi = (D,\mu,S,h)$ is any 4-tuple produced by the
algorithm, then the initial 4-tuple $\psi_0 = (\cC, \nu, \emptyset,
\ell+1)$ that gave rise to $\psi$ can be recovered by the equation
$\nu = \prod_{(i,j)\in S} L_{ij}\mu$.  Here $L_{ij}$ denotes the
lowering operator which is the inverse of $R_{ij}$.  Furthermore, the
sequence of 4-tuples leading from $\psi_0$ to $\psi$ is uniquely
determined by $\psi$ because all choices made along the way are
recorded in the set $S$.  In particular, no $4$-tuple can be produced
multiple times.

Suppose that the $4$-tuple $\psi=(D,\mu,S,h)$ occurs in the
algorithm. If $\psi$ is REPLACED by two 4-tuples $\psi_1$ and
$\psi_2$, we deduce from Lemma \ref{easylm} that $\ev(\psi) =
\ev(\psi_1) + \ev(\psi_2)$.  Moreover, if $\psi$ meets {\bf (iii)} and
is REPLACED by the single 4-tuple $\psi'=(D\cup (h,j),R_{hj}\mu,
S\cup(h,j),h)$, then Lemmas \ref{commuteA} and \ref{easylm} imply that
$\ev(\psi)=\ev(\psi')$.  Indeed, it follows from Corollary
\ref{tamecor} below that $\mu_{j-1} = \mu_j - 1$ and $D\cup (h,j)$ has
no outer corner in column $j$, so Lemma~\ref{commuteA} shows that
$\ev(D\cup (h,j),\mu,S,h)=0$.

When the algorithm terminates, let $\Psi_0$ (respectively $\Psi_1$)
denote the collection of all $4$-tuples $(D,\mu,S,h)$ in the final set 
such that $h=0$ (respectively $h>0$). We say that a 4-tuple $\psi$ 
{\em survives the algorithm} if at least one of its successors lies 
in $\Psi_0$. The above analysis implies that 
\[
\sum_{\nu\in N}T(\cC,\nu) = \sum_{\psi\in\Psi_0}\ev(\psi) +
\sum_{\psi\in\Psi_1}\ev(\psi).
\]

In the next section, we will prove the following two claims.

\begin{claim}
\label{claim1}
For each 4-tuple $\psi = (D,\mu,S,0)$ in $\Psi_0$ with $\mu_{\ell+1}
\geq 0$, $\mu$ is a $k$-strict partition with $\lambda \to \mu$ and
$\ev(\psi) = T(\cC(\mu),\mu)$.  Furthermore, for each such partition
$\mu$, there are exactly $2^{\NN(\lambda,\mu)}$ such 4-tuples $\psi$,
in accordance with the Pieri rule.
\end{claim}

\begin{claim}
\label{claim2}
There exists an involution $\iota:\Psi_1 \to\Psi_1$ of the form
$\iota(D,\mu,S,h) = (D,\mu',S',h)$ such that $\ev(\psi) +
\ev(\iota(\psi)) = 0$, for every $\psi\in\Psi_1$.
\end{claim}

\noindent
We remark that the 4-tuples $\psi\in\Psi_0$ with $\mu_{\ell+1} <0$
evaluate to zero trivially, by Definition \ref{recursedef}.  The two
claims therefore suffice to prove the Pieri rule (\ref{pieriTla}).

For each initial $4$-tuple $\psi_0=(\cC,\nu,\emptyset,\ell+1)$ of the
sum (\ref{initeq}), the algorithm produces a tree of 4-tuples with
root node given by $\psi_0$.  If the Substitution Rule REPLACES a
4-tuple $\psi$ by one or two other $4$-tuples $\psi_i$, we have a
branch in the tree from $\psi$ to the $\psi_i$.  The leaves of the tree
are exactly the $4$-tuples with $h=0$ or where the Substitution Rule
STOPS.  The fate of all the terms of the sum (\ref{initeq}) is encoded
by the collection of all the trees with root nodes
$(\cC,\nu,\emptyset,\ell+1)$ for $\nu \in \cN(\la,p)$.  This
collection will be called the {\em substitution forest}; the sum of
the cohomology classes represented by the roots of the substitution
forest is equal to the sum of classes given by the leaves.

\begin{example}
  We discuss an example of the substitution forest in detail.
  Consider the Grassmannian $\OG(n-1,2n+1)$ for $n\geq 5$, and the
  Pieri product
\[
c_1\cdot \ta_{2,1,1} = \ta_{2,1,1,1} + 2\,\ta_{3,1,1} + \ta_5.
\]
For simplicity, we will omit the commas in our notation for
compositions and pairs.  Thus $\la=211$, $ k= p = 1$, and we have
$\cC(\la) = \{ 11\}$ and $\cN(\lambda,p) = \{ 2111, 2120, 2210, 3110
\}$.  The substitution forest is pictured in Figure \ref{forest}, except
that we have
omitted those nodes $(D,\mu,S,h)$ which have $(D,\mu,S,h+1)$ as parent
and $(D,\mu,S,h-1)$ as child.

\begin{figure}\centering
\begin{tiny}
\psfrag{(i)}{(i)}
\psfrag{(ii)}{(ii)}
\psfrag{(iii)}{(iii)}
\psfrag{(iv)}{(iv)}
\psfrag{(v)}{(v)}
\psfrag{n1}{$(\{11\},2210,\emptyset,4)$}
\psfrag{n11}{$(\{11\},2210,\emptyset,2)$}
\psfrag{n111}{$(\{11,12\},2210,\emptyset,2)$}
\psfrag{n1111}{$(\{11,12,22\},2210,\emptyset,2)$}
\psfrag{n1112}{$(\{11,12,22\},2210,\{22\},2)$}
\psfrag{n112}{$(\{11,12\},3110,\{12\},2)$}
\psfrag{n1121}{$(\{11,12\},3110,\{12\},0)$}
\psfrag{n2}{$(\{11\},2111,\emptyset,4)$}
\psfrag{n21}{$(\{11\},2111,\emptyset,0)$}
\psfrag{n3}{$(\{11\},2120,\emptyset,4)$}
\psfrag{n31}{$(\{11\},2120,\emptyset,3)$}
\psfrag{n4}{$(\{11\},3110,\emptyset,4)$}
\psfrag{n41}{$(\{11\},3110,\emptyset,2)$}
\psfrag{n411}{$(\{11,12\},3110,\emptyset,2)$}
\psfrag{n4111}{$(\{11,12\},3110,\emptyset,0)$}
\psfrag{n412}{$(\{11,12\},4010,\{12\},2)$}
\psfrag{n4121}{$(\{11,12\},4010,\{12\},1)$}
\psfrag{n41211}{$(\{11,12,13\},5000,\{12,13\},1)$}
\psfrag{n412111}{$(\{11,12,13\},5000,\{12,13\},0)$}
\includegraphics[scale=0.5,viewport=60 0 629 324]{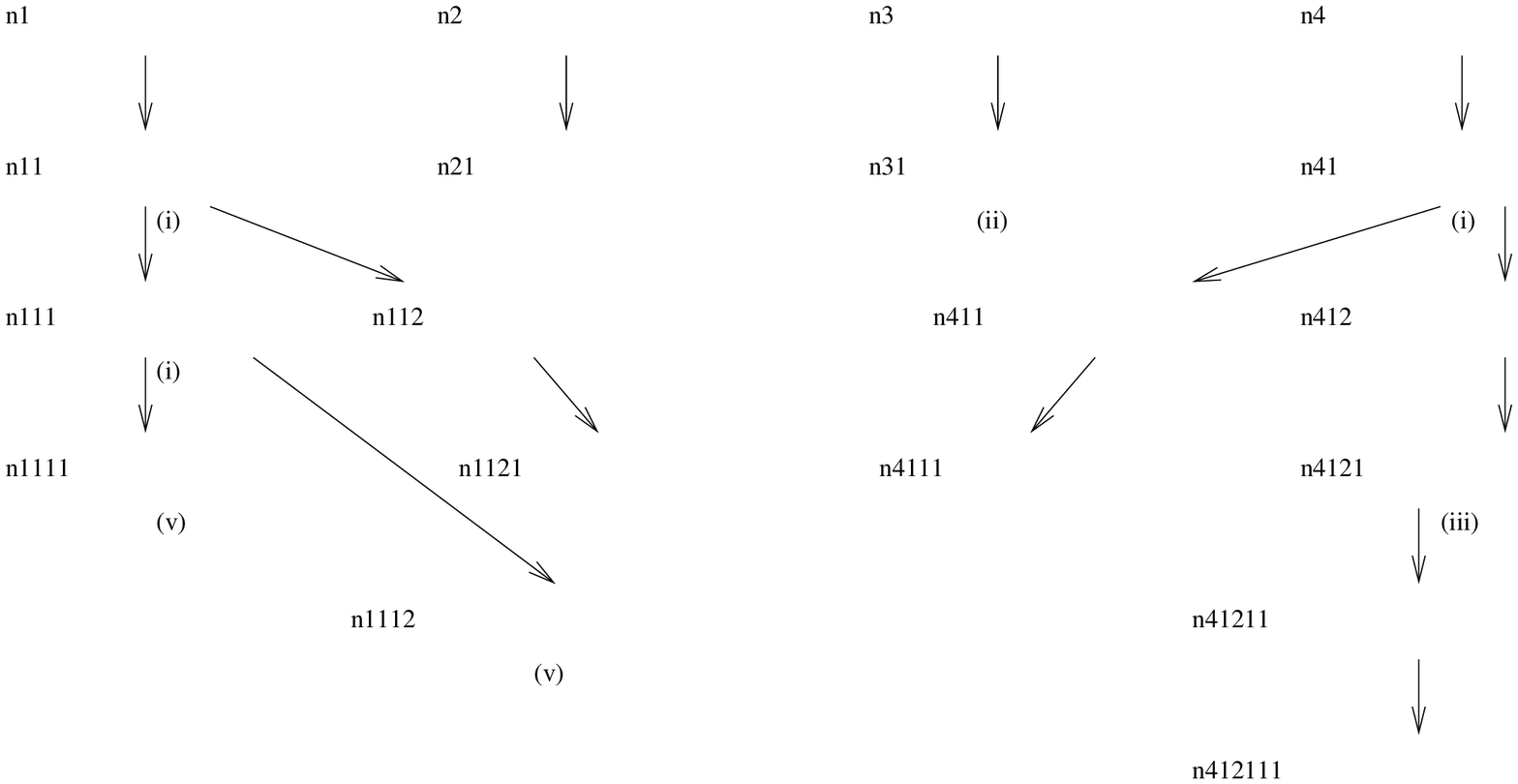}
\end{tiny}
\caption{The substitution forest for $\la=211$ and $k=p=1$.}
\label{forest}
\end{figure}

Observe that the root $(\{11\},2120,\emptyset,4)$ is the only
initial $4$-tuple that does not survive the algorithm.  We have
$\Psi_0 = \{$ $(\{11\},2111,\emptyset,0),$
$(\{11,12\},3110,\{12\},0),$ $(\{11,12\},3110,\emptyset,0),$
$(\{11,12,13\},5000,\{12,13\},0)$ $\}$, which corresponds exactly to
the terms in the Pieri product $c_1\cdot\tau_{211}$.  Furthermore,
each 4-tuple in the set $\Psi_1 = \{$
$(\{11,12,22\},2210,\emptyset,2),$ $(\{11,12,22\},2210,\{22\},2),$
$(\{11\},2120,\emptyset,3)$ $\}$
evaluates to zero in the cohomology ring of $\OG$.
\end{example}

\section{Proof of Theorem \ref{mainthm}}
\label{mainpf}

\subsection{}\label{sec1}

Recall the fixed choices of $p$, $\lambda$, $\ell$, $\cC$, and $m$
from \S \ref{initnot}.  In \S\ref{sec1} through \S\ref{sec3} we
furthermore let $\psi=(D,\mu,S,h)$ denote a $4$-tuple which occurs at
some step in the algorithm, i.e., a node of the substitution forest.
The symbols $D$, $\mu$, $S$, $h$ will refer to components of the
$4$-tuple $\psi$.  We will occasionally work with more than one valid
4-tuple.  If $(D',\mu',S',h')$ is an additional 4-tuple, then the sets
and values that Definition~\ref{efg} associates to this 4-tuple will
be called $R'$, $e'$, $f'$, and $g'$.

The algorithm has two phases.  A $4$-tuple $\psi$ is in Phase 1 if
$(h,h)\notin D$, and in Phase 2 if $(h,h)\in D$.  The level $h$ is
always used to index an entry of the integer sequence $\mu$ in $\psi$;
it begins at $h=\ell+1$ and decreases as the 4-tuple proceeds through
the algorithm.  In Phase 1 we have $h \geq m$, while $h \leq m$ in
Phase 2.  Throughout the algorithm we have $i \leq m \leq j$ for each
$(i,j)\in S$, so $\mu$ is obtained from the initial composition $\nu$
by removing boxes from rows weakly below the middle row of $\la$ and
adding them to rows weakly above the middle row.

The set $D$ is initially equal to $\cC$ and grows when REPLACE
statements are encountered.  Lemma~\ref{lem:DinCDC} below shows that
all pairs added to $D$ come from the outer rim $\partial\cC$.  In
Phase 1, pairs are added by rule {\bf(i)} to column $h$, so as the
level $h$ decreases from $\ell+1$ to $m$, these pairs are added along
vertical columns of $\partial\cC$, proceeding from top (row 1) to
bottom (row $m$) and right to left.  In Phase 2, the set $D$ mainly
grows when rule {\bf(iii)} adds pairs to row $h$, in which case the
pairs are added in horizontal rows of $\partial\cC$, from left to
right and bottom to top.  In some cases rule {\bf(iv)} will add extra
pairs $(i,g)$ to $D$, where $i \leq h$.  Lemma~\ref{lem:ivdooms}
implies that if $\psi$ meets {\bf(iv)}, then it will not survive the
algorithm, and only pairs from column $g$ of $\partial\cC$ can be
added to its successors.  In particular, all 4-tuples in $\Psi_0$ are
produced from the the initial 4-tuples by applications of rules
{\bf(i)} and {\bf(iii)}.

Our proof of Theorem \ref{mainthm} occupies the remainder of this
section.  In \S\ref{sec2} we prove some properties satisfied by
4-tuples that occur in the algorithm.  Additional properties for
4-tuples in $\Psi_0$ are proved in \S\ref{sec3}.  The proof of Claim
\ref{claim1} is then given in \S\ref{sec4}, while Claim \ref{claim2}
is justified in \S\ref{sec5}.

\subsection{}\label{sec2}

We prove some lemmas that reveal what can happen to the 4-tuple $\psi
= (D,\mu,S,h)$ during the algorithm.

\begin{lemma}\label{lem:DinCDC}
  We have $D \subset \cC \cup \partial\cC$.
\end{lemma}
\begin{proof}
  It is enough to show that if the substitution rule adds the pair
  $(i,j)$ to $D$, then $(i,j) \in \partial\cC$.  Notice first that $i
  \leq h \leq j$.  If $(i,j) \not\in \partial\cC$, then $i>1$ and
  $(i-1,j-1) \not\in \cC$.  Let $\psi' = (D',\mu',S',h')$ be the most
  recent predecessor of $\psi$ such that $(i-1,j) \not\in D'$.  Then
  $\psi'$ meets rule {\bf(i)}, {\bf(iii)}, or {\bf(iv)}, which adds
  the pair $(i-1,j)$ to $D'$.  Since the pair $(i-1,j-1) \in
  D'\ssm\cC$ was added to a predecessor of $\psi'$ of level smaller
  than $j$, it follows that $i \leq h \leq h' \leq j-1$, so $\psi'$
  does not meet {\bf(i)} or {\bf(iii)}.  But $\psi'$ also does not
  meet {\bf(iv)} because $g' \leq j-1$, a contradiction.
\end{proof}

\begin{lemma}\label{phaseonelm}
  If $j > h$ and $(j,j) \not\in D$, then $\mu_j \leq \la_{j-1}$.
\end{lemma}
\begin{proof}
  Assume that $\mu_j > \la_{j-1}$ and let $\psi' = (D',\mu',S',j)$ be
  the most recent predecessor of $\psi$ of level $j$.  Then $\mu'_j
  \geq \mu_j$, and since $\psi'$ does not meet {\bf(ii)}, it follows
  that $D'$ has an outer corner $(i,j)$ in column $j$.  If $i<j$, then
  since $(i,j-1)\in \cC$ we obtain $\mu'_j + \mu'_i > \la_{j-1} +
  \la_i > 2k+(j-1)-i$, and otherwise we have $i=j=m$ and $\mu'_j >
  \la_{m-1} > k$.  In both cases $\psi'$ satisfies $\W(i,j)$.  But
  then $\psi'$ meets {\bf(i)} and is not the most recent predecessor
  of $\psi$ of level $j$, a contradiction.
\end{proof}

\begin{lemma}\label{fglemma}
  If $2\leq h\leq m$, $\mu_h\geq \la_{h-1}$, and $f<g$, then $\mu_g =
  \la_{g-1}$ and $(h,g)\notin S$.
\end{lemma}
\begin{proof}
  By assumption we have $h \leq m \leq f < g$, and
  Lemma~\ref{phaseonelm} implies that $\mu_g \leq \la_{g-1} \leq \la_m
  \leq k$.  Set $x = 2k+g-h-\mu_g$.  Since $(h,f) \notin \cC$ we get
  $(h,g-1) \notin \cC$ which implies $\la_h < x$.  Since
  $[h,\la_{h-1}] \notin R$ we also obtain $x < \la_{h-1}$.  Assume
  that $\mu_g < \la_{g-1}$.  Then we get $\la_{g-1} > 2k+g-h-x$, which
  implies that $g \leq r(h+x) \leq r(h+\la_{h-1}) = g$.  The
  definition of $x$ now shows that $[h,x] \in R$, from which we deduce
  that $x < e$ and $g = r(h+x) \leq r(h+e) = f$, a contradiction.
  Finally assume that $(h,g) \in S$.  Since $(h,g-1)\notin\cC$ we
  deduce that the parent of $\psi$ is $\psi' = (D \ssm (h,g), \mu',
  S\ssm(h,g),h)$, where $\mu' = L_{hg}\mu$.  But Lemma~4.2 then
  implies that $\mu_g < \mu'_g \leq \la_{g-1}$, a contradiction.
\end{proof}

We next make some observations concerning condition $\X$ and rules
{\bf (iv)} and {\bf(v)}.

\begin{lemma}\label{lem:stickyx}
  If condition $\X$ holds for $\psi$, then $\X$ also holds for the
  children of $\psi$.  In particular, $\psi$ does not survive the
  algorithm, and all its successors have level $h$.
\end{lemma}
\begin{proof}
  Assume that $\psi= (D,\mu,S,h)$ satisfies condition $\X$ and let
  $\psi' = (D',\mu',S',h)$ be a child of $\psi$.  If $S' = S$ then
  $\mu'=\mu$ and $\X$ also holds for $\psi'$, so we may assume that
  $S'\ssm S = \{(i,j)\}$ and $\mu' = R_{ij}\mu$, where $i \leq h$.  We
  can also assume that $\la_{h-1} = \mu_h = \mu'_h < \mu'_{h-1}$.
  Since $(h,h)\in D$ and $(i,j) \not\in D$, we get $i < h < j$.  In
  particular, $\psi$ meets {\bf(iv)} and $j=g$.  Let $\ov{\psi} = (\ov
  D, \ov\mu, \ov S, g)$ be the most recent predecessor of $\psi$ of
  level $g$, and let $(a,g)$ be the outer corner of $\ov D$ in column
  $g$.  Then $a \leq h-1$, and $\W(a,g)$ fails for $\ov\psi$ since it
  does not meet {\bf(i)}.  Using this and $k$-strictness of $\la$, we
  obtain $\la_{h-1}+\mu'_g < \la_{h-1}+\mu_g \leq
  (\la_a+a-h+1)+\ov\mu_g \leq \ov\mu_a+\ov\mu_g+1+a-h \leq 2k+1+g-h$.
  Lemma~\ref{lem:fisg} now implies that $f'=g$.  We conclude that
  $\psi'$ satisfies condition $\X$ as $\mu'_h = \la_{h-1}$ and $(h,f')
  = (h,g) \not\in S'$.
\end{proof}

\begin{lemma}\label{lem:hgbox}
  Assume that $(h,h)\in D$ and $\psi$ satisfies $\X$ or $\W(h,g)$.

  \smallskip\noin{\em(a)} If $h<g$ and $(h,g-1)\notin D$, then $\psi$
  meets {\bf(iii)}.

  \smallskip\noin{\em(b)} If $\psi$ meets {\bf(v)}, then $(h,g) \in D$.
\end{lemma}
\begin{proof}
  Assume that $h<g$ and $(h,g-1) \notin D$, and let $(h,d)$ be the
  outer corner of $D$ in row $h$.  Then $d \leq g-1$ and
  $(h-1,d)\in\cC$.  Using Lemma~\ref{phaseonelm} and the fact that
  $\cC$ and $D$ have equally many pairs in column $d$, we obtain
  $\mu_g \leq \la_{g-1} \leq \la_d \leq \mu_d$.  If $\psi$ satisfies
  condition $\X$, then $\mu_h \geq \la_{h-1}$ and it follows that
  $\mu_h+\mu_d \geq \la_{h-1}+\la_d > 2k+d-h+1$.  If $\psi$ satisfies
  $\W(h,g)$, then $\mu_h+\mu_d \geq \mu_h+\mu_g > 2k+g-h \geq 2k+d-h$.
  This shows that $\psi$ satisfies $\W(h,d)$ and (a) follows.  If
  $\psi$ meets {\bf(v)} and $(h,g)\notin D$, then part (a) implies
  that $D$ has an outer corner in column $g$, so $\psi$ meets
  {\bf(iii)} or {\bf(iv)}.  This contradiction proves (b).
\end{proof}

The following result implies that no term meeting {\bf(iv)} survives
the algorithm, and also that applications of {\bf(iv)} happen in an
uninterrupted sequence.

\begin{lemma}\label{lem:ivdooms}
  Assume that $\psi$ meets {\bf(iv)} and let $\psi' = (D',\mu',S',h)$
  be any successor of $\psi$ of level $h$.

  \smallskip\noin{\em(a)} We have $(h,g-1) \in D$.

  \smallskip\noin{\em(b)} If $(h,g)\notin D'$ then $\psi'$ meets
  {\bf(iii)} or {\bf(iv)}.

  \smallskip\noin{\em(c)} If $(h,g)\in D'$ and $\psi'$ does not meet
  {\bf(v)}, then $S'=S$, the child $(D',\mu,S,h-1)$ of $\psi'$ meets
  {\bf(v)}, and $g_{h-1} = g$.
\end{lemma}
\begin{proof}
  Part (a) follows from Lemma~\ref{lem:hgbox}.  If $\psi$ satisfies
  condition $\X$ then assertions (b) and (c) follow from Lemma
  \ref{lem:stickyx}, so we may assume that $\X$ fails and $\W(h,g)$
  holds for $\psi$.  Without loss of generality we can also replace
  $\psi$ with the most recent predecessor whose parent did not meet
  {\bf(iv)}.  Let $(i,g)$ be the outer corner of $D$ in column $g$.
  Since $\psi$ does not meet {\bf(iii)} we have $i<h$.  Furthermore
  $\W(i,g)$ fails for $\psi$, since otherwise the most recent
  predecessor of $\psi$ of level $g$ meets {\bf(i)}.  We obtain
  $\mu_h+\mu_g \leq \la_{h-1}+\mu_g \leq \la_i-h+1+i + \mu_g \leq
  \mu_i+\mu_g+1+i-h \leq 2k+1+g-h$.  Since $\psi$ satisfies $\W(h,g)$
  we deduce that $\la_i-h+1+i = \la_{h-1} = \mu_h = 2k+1+g-h-\mu_g$.

  If $S' \supsetneq S$, then $\mu'_g < \mu_g$, and
  Lemma~\ref{lem:fisg} implies that $f'=g$.  In addition, we have
  either $(h,g) \notin S'$ or $\mu'_h > \mu_h = \la_{h-1}$.  Since
  both possibilities imply that $\psi'$ satisfies condition $\X$, it
  follows that assertions (b) and (c) are true for $\psi'$.

  Otherwise, we have $S' = S$ and $\mu' = \mu$.  Then $\psi'$
  satisfies $\W(h,g)$ and (b) is true.  Assume that $(h,g)\in D'$ and
  $\psi'$ does not meet {\bf(v)}.  Then $\X$ fails for $\psi'$, so we
  must have $\la_{h-2}-1 = \la_{h-1} = \mu_h < \mu_{h-1}$.  Since
  $(h-1,g)\notin S$, we deduce that $(D',\mu,S,h-1)$ satisfies
  condition $\X$.  Furthermore, since $(h-1,g)\notin\cC$ and
  $\la_{h-2}=\la_{h-1}+1$, we obtain $(h-2,g)\notin\cC$, so
  $g_{h-1}=g$.  We conclude that $(D',\mu,S,h-1)$ meets {\bf(v)},
  which completes the proof of (c).
\end{proof}

\begin{defn}
  Let $\partial^1\cC = \{ (i,j) \in \partial\cC \mid 
  (i,j-1)\in\cC \text{ or } i=j=m \}$.
\end{defn}

\begin{cor}\label{cor:addorder}
  Let $(i,j) \in D \ssm \cC$, and if $2 \leq h \leq m$ then assume
  that $j \neq g$.  If $(i,j) \in \partial^1\cC$ then this pair was
  added to $D$ by rule {\bf(i)}, and otherwise it was added by rule
  {\bf(iii)}.
\end{cor}
\begin{proof}
  Let $\psi' = (D',\mu',S',h')$ be the most recent predecessor of
  $\psi$ for which $(i,j) \notin D'$.  Then $(i,j)$ is an outer corner
  of $D'$.  If $\psi'' = (D'',\mu'',S'',h'')$ is any predecessor of
  $\psi'$ meeting {\bf(iv)}, then Lemma~\ref{lem:ivdooms} implies that
  that $h'' = h'$, $g_{h'}=g$, and $(h',g-1)\in D'$, so we must have
  $j=g$, a contradiction.  It follows that no predecessor of $\psi'$
  meets {\bf(iv)}.  If $\psi'$ meets {\bf(i)} then $h'=j$, and since
  $D'\ssm\cC$ contains no pairs in column $j-1$ we deduce that $(i,j)
  \in \partial^1\cC$.  Finally assume that $\psi'$ meets {\bf(iii)}
  and $(i,j)\in \partial^1\cC$.  Let $\wt\psi = (\wt D, \wt S, \wt\mu,
  j)$ be the most recent predecessor of $\psi'$ of level $j$.  Then
  $(i,j)=(h',j)$ is an outer corner of $\wt D$, since otherwise
  $(h'-1,j) \in D' \ssm \wt D$ was added to a 4-tuple on the path from
  $\wt\psi$ to $\psi'$, which is impossible.  But then the inequality
  $\wt\mu_i+\wt\mu_j = \mu'_i+\mu'_j > 2k+j-i$ implies that $\wt\psi$
  meets {\bf(i)}.  This contradiction finishes the proof.
\end{proof}

We now prove some results that will be used later to show that
surviving 4-tuples $\psi = (D,\mu,S,0) \in \Psi_0$ satisfy $\la \to
\mu$.

\begin{lemma}\label{lem:narrownoeat}
  Let $\psi=(D,\mu,S,h)$ and let $j\leq \ell$ be a positive integer.

  \smallskip\noin{\em(a)} If $h\geq 1$ and $(h,j) \not\in \cC$ and
  $(h+1,j) \in D$, then we have $\mu_j \geq \la_j$.

  \smallskip\noin{\em(b)} If $h \leq 1$ or $(h-1,j) \in \cC$,
  then we have $\mu_j \geq \la_j - 1$.  Moreover, if $\mu_j = \la_j -
  1$, then $D \smallsetminus \cC$ contains exactly one pair in column
  $j$, and this pair is also in $S$.
\end{lemma}
\begin{proof}
  Suppose that $\mu_j < \la_j$ and choose $i>h$ maximal such that
  $(i,j) \in D$.  Let $\psi' = (D',\mu',S',h')$ be the most recent
  predecessor of $\psi$ with $(i,j) \not\in D'$.  Then $\psi'$ meets
  rule {\bf(i)}, {\bf(iii)}, or {\bf(iv)}, which adds the pair $(i,j)$
  to $D'$.  Let $\ov{\psi} = (D' \cup (i,j), \ov{\mu}, \ov{S}, h')$ be
  the child of $\psi'$ that is a predecessor of $\psi$.  Notice that
  $\mu'_t \leq \ov{\mu}_t \leq \la_{t-1}$ for all integers $t$ such
  that $h < t \leq h'$ and $(t,t) \in D' \cup (i,j)$, since otherwise
  condition $\X$ holds for every successor of $\ov{\psi}$ of level
  $t$.  We also have $\mu'_j \leq \la_j$, and if $\mu'_j = \la_j$ then
  $(i,j) \in \ov{S}$, $i<j$, $\ov{\mu}_i>\mu'_i$, and
  $\ov{\mu}_j<\mu'_j$.

  If $\psi'$ meets {\bf(i)}, then $h'=j$.  Since $(i-1,j) \not\in \cC$
  we have $\mu'_i + \mu'_j \leq \la_{i-1}+\la_j \leq 2k+j-i+1$.  As
  $\W(i,j)$ holds for $\psi'$, it follows that $\mu'_i=\la_{i-1}$ and
  $\mu'_j = \la_j$.  But this implies that $\ov{\mu}_i > \la_{i-1}$, a
  contradiction.

  Therefore $\psi'$ meets {\bf(iii)} with $h'=i$, or it meets
  {\bf(iv)} with $h' \geq i$.  In either case we have $g' = j$, and
  since $\psi'$ does not satisfy condition $\X$, it must satisfy
  $\W(h',j)$.  Since $(h'-1,j) \not\in \cC$ and thus $\mu'_{h'}+\mu'_j
  \leq \la_{h'-1} + \la_j \leq 2k+j-h'+1$, it follows that
  $\mu'_{h'}=\la_{h'-1}$ and $\ov{\mu}_j < \mu'_j = \la_j$.  We obtain
  $\la_{h'-1} + \ov{\mu}_j \leq 2k+j-h'$, so Lemma~\ref{lem:fisg}
  shows that $\ov{f} = j$.  Since $\ov{\mu}_i > \mu'_i$, we must also
  have $i<h'$, so $(h',\ov{f}) \not\in \ov{S}$ and $\ov{\psi}$
  satisfies condition $\X$.  This contradiction completes the proof of
  part (a).

  If $\mu_j \leq \la_j - 2$, then $D \smallsetminus \cC$ contains at
  least two pairs in column $j$, say $(a+1,j)$ and $(a,j)$, and the
  assumptions in (b) imply that $a \geq h$.  Let $\psi' =
  (D',\mu',S',a)$ be the most recent predecessor of $\psi$ of level
  $a$.  Part (a) applied to $\psi'$ implies that $\mu'_j \geq \la_j$,
  a contradiction since $\mu'_j = \mu_j$.
\end{proof}

\begin{cor}\label{tamecor}
  Assume that $\psi$ meets {\bf(iii)} and let $(h,j)$ be the outer
  corner of $D$ in row $h$.  If $\mu_j > \mu_{j-1}$, then $\mu_{j-1} =
  \mu_j-1$ and $D\cup(h,j)$ has no outer corner in column $j$.
\end{cor}
\begin{proof}
  Lemma~\ref{phaseonelm} implies that $\mu_{j-1} < \mu_j \leq
  \la_{j-1}$.  Since $(h-1,j-1) \in \cC$, it follows from
  Lemma~\ref{lem:narrownoeat}(b) that $\mu_j = \la_{j-1} =
  \mu_{j-1}+1$ and $D \ssm \cC$ contains a unique pair $(i,j-1)$ in
  column $j-1$.  Since $i<j-1$, it is enough to show that $i=h$.
  Lemma~\ref{lem:ivdooms} implies that $(i,j-1)$ was added by {\bf(i)}
  or {\bf(iii)}, so $\psi$ satisfies $\W(i,j-1)$, and we obtain $\mu_i
  + \mu_j = \mu_i + \mu_{j-1}+1 > 2k+j-i$.  Assume that $i>h$ and let
  $\psi' = (D',\mu',S',i)$ be the most recent predecessor of $\psi$ of
  level $i$.  Then $(i,j-1)\in D'$ and $g'=j$.  Since $\mu'_i = \mu_i$
  and $\mu'_j \geq \mu_j$, $\W(i,j)$ holds for $\psi'$.  But then
  $\psi'$ meets {\bf(iv)} and is not the most recent predecessor, a
  contradiction.
\end{proof}

\begin{lemma}\label{noragged}
  Assume that $j>m$.  If $h=0$ or $(h,j) \in D$, then $\mu_j \leq
  \mu_{j-1}$.
\end{lemma}
\begin{proof}
  Assume that $\mu_j > \mu_{j-1}$.  Then Lemmas \ref{phaseonelm}
  and \ref{lem:narrownoeat}(b) imply that $\mu_j = \la_{j-1} =
  \mu_{j-1}+1$, and $D \smallsetminus \cC$ contains a unique pair
  $(i,j-1)$ in column $j-1$, with $i \geq h$.  Let $\psi' =
  (D',\mu',S',i)$ be the most recent predecessor of $\psi$ of level
  $i$ for which $(i,j) \not\in D'$.  The assumptions of the lemma then
  imply that $\psi' \neq \psi$.  Lemma~\ref{lem:ivdooms} shows that
  $(i,j-1)$ was added to $D$ by {\bf(i)} or {\bf(iii)}, so $\psi'$
  satisfies $\W(i,j-1)$.  Since $\mu'_j \geq \mu_j > \mu_{j-1} =
  \mu'_{j-1}$, $\psi'$ also satisfies $\W(i,j)$, so $\psi'$ meets
  {\bf(iii)} or {\bf(iv)}.  The choice of $\psi'$ implies that $(i,j)$
  must be the outer corner added to $D'$, so in fact $\psi'$ meets
  {\bf(iii)}.  Now the statement of rule {\bf(iii)} implies that
  $(i,j) \in S$, so $\mu'_j > \mu_j > \mu'_{j-1}$.  This contradicts
  Corollary~\ref{tamecor}.
\end{proof}

\begin{lemma}\label{lem:decrease}
  Assume $h<m$.  Then $\mu_j \le \la_{j-1}$ and $\mu_j <\mu_{j-1}$
  for $h<j\leq m$.
\end{lemma}
\begin{proof}
  If the statement is false, then choose $i>h$ minimal such that
  $\mu_i>\la_{i-1}$ or $\mu_i\geq \mu_{i-1}$, and let $\psi' =
  (D',\mu',S',i)$ be the most recent predecessor of $\psi$ of level
  $i$.  Since $\mu'_i>\la_{i-1}$ or $\mu'_i\geq\mu'_{i-1}$ and
  condition $\X$ fails for $\psi'$, we have $(i,i) \not\in D'$, so
  $i=m$ and $\mu'_m \geq \la_{m-1} > k$.  But then $\mu'_t+\mu'_m \geq
  \la_t+\mu'_m > (k+m-t)+k$ for $1 \leq t < m$, so $\psi'$ meets
  {\bf(i)}, a contradiction.
\end{proof}

\begin{lemma}\label{Rlemma}
  Assume that $(h,h)\in D$, $\mu_h = \la_{h-1}$, and $[h,\la_{h-1}]\in
  R$.  Then $\psi$ satisfies condition $\X$ and does not survive the
  algorithm.
\end{lemma}
\begin{proof}
  Since $[h,\la_{h-1}] \in R$ and $r(h+\la_{h-1}) = g$, we have $\mu_g
  \leq 2k+g-h-\la_{h-1}$, hence $\W(h,g)$ fails for $\psi$.
  Lemma~\ref{lem:decrease} shows that $\mu_h > \mu_{h+1} > \dots >
  \mu_m$, so $\W(d,g)$ fails for $h \leq d \leq m$.  If $(h,g) \not\in
  D$ then condition $\X$ holds because $(h,f) = (h,g) \not\in S$.
  Otherwise choose $i \geq h$ maximal such that $(i,g) \in D$, and let
  $\psi' = (D',\mu',S',h')$ be the most recent predecessor of $\psi$
  for which $(i,g) \not\in D'$.  Then $\psi'$ meets {\bf(iv)} and
  $\psi$ is a successor of the child $\psi'' = (D,\mu,S,h')$ of
  $\psi'$.  If $h' > h$, then Lemma~\ref{lem:stickyx} implies that
  $\X$ fails and $\W(h',g)$ holds for $\psi'$.  Since $\W(h',g)$ fails
  for $\psi''$, it follows that $i<h'$ and $\mu_{h'} + \mu_g =
  2k+g-h'$.  We also have $\la_{h'-1} + \mu_g \leq \la_{h-1} + \mu_g +
  h-h' \leq 2k+g-h'$, so $\mu_{h'} \geq \la_{h'-1}$, and using
  Lemma~\ref{lem:fisg} we get $(h',f'') = (h',g) \not\in S$.  But then
  $\psi''$ satisfies condition $\X$, a contradiction.  We therefore
  have $h' = h$, $\W(h,g)$ fails for $\psi'$, $\X$ holds for $\psi'$,
  and the result follows from Lemma~\ref{lem:stickyx}.
\end{proof}

In our applications of Lemma~\ref{Rlemma} we only need the fact that
$\psi$ does not survive the algorithm, so it is enough to know that a
predecessor of $\psi$ meets {\bf(iv)}.  The last six lines of the above
proof could therefore be omitted.

\subsection{}\label{sec3}

In this section we will study a $4$-tuples $\psi = (D,\mu,S,0) \in
\Psi_0$ with $\mu_{\ell+1}\geq 0$.  For such a 4-tuple,
Corollary~\ref{cor:addorder} implies that each pair $(i,j) \in D \ssm
\cC$ was added by {\bf(i)} or {\bf(iii)}.  More precisely, the pair
$(i,j)$ was added by {\bf(i)} if $(i,j)\in \partial^1\cC$, and
otherwise the pair was added by {\bf(iii)}.

\begin{prop}\label{survivors:muD}
  Suppose that $\psi=(D,\mu,S,0)$ and $\mu_{\ell+1}\geq 0$.  Then
  $\mu$ is a $k$-strict partition with $|\mu|=|\la|+p$, satisfying
  $\la_j-1 \leq \mu_j \leq \la_{j-1}$ for every $j\geq 1$, and
  $\la_j\leq\mu_j$ when $\la_j > k$.  Furthermore, we have
  $D=\cC_{\ell+1}(\mu)$.
\end{prop}
\begin{proof}
  By Lemma \ref{lem:decrease} we have $\mu_j\leq\la_{j-1}$ and $\mu_j
  < \mu_{j-1}$ for $1\leq j\leq m$, and Lemmas \ref{phaseonelm} and
  \ref{noragged} show that $\mu_j\leq\min(\la_{j-1},\mu_{j-1})$ for
  $j>m$.  We deduce that $\mu$ is a $k$-strict partition.
  Lemma~\ref{lem:narrownoeat}(b) implies that $\la_j-1 \leq \mu_j$ for
  every $j$.  Clearly $\la_j\leq\mu_j$ when $\la_j > k$, and
  $|\mu|=|\la|+p$.

  It remains to show that the set $\cC_{\ell+1}(\mu)$ is equal to $D$.
  If $D \not\subset \cC_{\ell+1}(\mu)$, then since $\cC_{\ell+1}(\mu)$
  and $D$ are both valid sets of pairs, we can find an inner corner
  $(i,j) \in D \ssm \cC_{\ell+1}(\mu)$ such that $(i+1,j)\not\in D$
  and $(i,j+1)\not\in D$.  Since $(i,j)\not\in\cC$ by
  Lemma~\ref{lem:CinDinCDC}, the pair $(i,j)$ was added by {\bf(i)} or
  {\bf(iii)}, and $\W(i,j)$ holds since $\mu_i$ and $\mu_j$ did not
  change since this event.

  On the other hand, if $\cC_{\ell+1}(\mu) \not\subset D$, then we can
  find an outer corner $(i,j)$ of $D$ such that $(i,j) \in
  \cC_{\ell+1}(\mu)$.  If $(i,j) = (m,m)$ or if $(i,j-1) \in \cC$,
  then the most recent predecessor of $\psi$ of level $j$ meets {\bf
    (i)}, and otherwise we deduce from Lemma~\ref{lem:CinDinCDC} that
  the most recent predecessor of $\psi$ of level $i$ meets {\bf(iii)}.
  This contradiction finishes the proof.
\end{proof}

\begin{lemma}\label{lem:neighbor}
  Assume that $\psi = (D,\mu,S,0)$ and $j$ are such that $\mu_j =
  \la_j-1 \geq 0$, and let $(i,j)$ be the unique pair in column $j$ of
  $D\ssm\cC$.  Then the removed box $[j,\la_j]$ and the above box
  $[j-1,\la_j]$ are $k$-related to the boxes $[i,c]$ and $[i,c-1]$,
  respectively, where $c = 2k + 2 + j - i - \la_j$, and these latter
  boxes belong to $R$.
\end{lemma}
\begin{proof}
  Let $\psi' = (D',\mu',S',h')$ be the most recent predecessor of
  $\psi$ for which $(i,j) \not\in D'$.  Since $\mu'_j = \la_j$ and
  $\psi'$ satisfies $\W(i,j)$, we obtain $\mu_i \geq \mu'_i+1 \geq c$,
  and since $(i,j) \not\in \cC$ we similarly have $\la_i \leq c-2$.
  The boxes $[i,c]$ and $[i,c-1]$ belong to $R$ because
  $\mu_{j+1}<\la_j$ and $\mu_j<\la_j$ (see \S\ref{initialdefs}).
\end{proof}

\begin{lemma}\label{lem:combine}
  Assume that $\psi=(D,\mu,S,0)$, $\mu_{\ell+1} \geq 0$, and let $i
  \leq m$ be any integer such that $\mu_i = \la_{i-1}$.  Then
  $[i,\la_{i-1}] \not\in R$ and $(i,f_i(\mu)) \in S$.
\end{lemma}
\begin{proof}
  Let $\psi' = (D',\mu',S',i)$ be the most recent predecessor of
  $\psi$ of level $i$.  Then $\mu'_i = \mu_i$ and $(i,i)\in D'$; if
  $i=m$ this follows because $\mu'_m = \la_{m-1} > k$.
  Lemma~\ref{Rlemma} implies that $[i,\la_{i-1}] \not\in R'$.  Since
  all pairs $(c,d) \in D\smallsetminus D'$ were added by {\bf(iii)}
  and satisfy $d > g'$, it follows that $\mu'_j = \mu_j$ for $m \leq j
  \leq g'$, so Lemma~\ref{Rflemma} shows that $[i,\la_{i-1}] \not\in
  R$ and $f' = f_i(\mu)$.  Finally, we must have $(i,f') \in S'
  \subset S$ since $\psi'$ does not satisfy condition $\X$.
\end{proof}

\begin{prop}\label{survivors:12}
  If $\psi=(D,\mu,S,0)$ and $\mu_{\ell+1}\geq 0$, then we have
  $\la\to\mu$.
\end{prop}
\begin{proof}
  By Proposition \ref{survivors:muD}, it suffices to check that
  conditions (1) and (2) of \S\ref{classpieri} are true.  Condition
  (1) follows from Lemma~\ref{lem:combine} since $[i,\la_{i-1}]
  \not\in R$ for each $i$.  Suppose that $\mu_j+1 = \la_j = d$ for
  $j_1 \leq j \leq j_2$.  According to Lemma~\ref{lem:neighbor}, each
  removed box $[j,d]$ for $j_1 \leq j \leq j_2$ is $k$-related to some
  box $[i_j,c_j] \in \mu\ssm\la$, and the box $[i_j,c_j-1]$ is also in
  $\mu\ssm\la$.  Condition (1) implies that each box $[j,d]$ is
  $k$-related to at most one box of $\mu\ssm\la$.  It follows that if
  $j<j_2$, then $[i_j,c_j] = [i_{j+1},c_{j+1}-1]$, so all the boxes
  $[i_j,c_j]$ lie in the same row of $\mu\ssm\la$.  Condition (2)
  follows from this since we also know that the box $[j_1-1,d]$ is
  $k$-related to $[i_{j_1},c_{j_1}-1]$.
\end{proof}

\subsection{}\label{sec4}

Propositions \ref{survivors:muD} and \ref{survivors:12} tell us that
if $\psi = (D,\mu,S,0)$ is any 4-tuple in $\Psi_0$ with $\mu_{\ell+1}
\geq 0$, then $\mu$ is a $k$-strict partition with $\la \to \mu$, $D =
\cC_{\ell+1}(\mu)$ is uniquely determined by $\mu$, and $\ev(\psi) =
T(\cC(\mu),\mu)$ is a term appearing in the Pieri rule
(\ref{pieriTla}).  To account for the multiplicities, we give an
explicit construction of the possible sets $S$ in these 4-tuples.  In
this section we fix an arbitrary $k$-strict partition $\mu$ such that
$\la \to \mu$ and $|\mu|=|\la|+p$.

A {\em component} means an (edge or vertex) connected component of the
set $\A$ of \S \ref{classpieri}.  We say that a box $B$ of $\A$ is
{\em distinguished} if the box directly to the left of $B$ does not
lie in $\A$.  We say that $B$ is {\em optional} if it is the rightmost
distinguished box in its component.  Notice that $\NN(\la,\mu)$ is
equal to the number of optional distinguished boxes in $\A$.  To each
distinguished box $B = [i,c]$ we associate the pair $(i,j) =
(i,r(i+c))$.  The inequality $\la_{i-1} > 2k+i-(i+c)$ implies that
$i\leq j$, so $(i,j) \in \Delta$.  Let $E$ (respectively $F$) be the
set of pairs associated to optional (respectively non-optional)
distinguished boxes.  We furthermore let $G$ be the set of all pairs
$(i,j) \in \Delta$ for which some box in row $i$ of $\mu\ssm\la$ is
$k$-related to a box in row $j$ of $\la\ssm\mu$.

\begin{lemma}\label{lem:assoc}
  {\rm(a)} We have $E \cup F \cup G \subset \cC_{\ell+1}(\mu)
  \cap \partial\cC$. \smallskip

  \noindent{\rm(b)} Each pair in $E \cup F$ is associated to exactly
  one distinguished box of $\A$. \smallskip

  \noindent{\rm(c)} The sets $E$, $F$, and $G$ are pairwise disjoint.
  \smallskip

  \noindent{\rm(d)} If $(i,j) \in F$, then $j = f_i(\mu)$.
  \smallskip

  \noindent{\rm(e)} If $(i,j) \in E \cup F \cup G$, $i<j$, and
  $(i,j-1) \not\in \cC$, then $\mu_j < \la_{j-1}$.
\end{lemma}
\begin{proof}
  Let $(i,j) \in G$.  Then $\mu_j = \la_j-1$ and the boxes $[j,\la_j]$
  and $[j-1,\la_j]$ are $k$-related to $[i,d]$ and $[i,d-1]$, where $d
  = 2k+2+j-i-\la_j$.  We also have $\la_i+1 < d \leq \mu_i$.
  Therefore $\la_i + \la_j < d+\la_j-1 = 2k+1+j-i$ and $\mu_i+\mu_j
  \geq d+\mu_j = 2k+1+j-i$, so $(i,j) \in \cC_{\ell+1}(\mu) \ssm \cC$.
  Assume that $(i,j)$ is associated to a distinguished box $[i,c] \in
  \A$.  Since $r(i+c) = j < j+1 = r(i+d)$, we must have $c < d$, hence
  $\mu_j = 2k+1+j-i-d \leq 2k+j-i-c$.  But then $[i,c] \in R(\mu)$, a
  contradiction.  It follows that $G \cap (E \cup F) = \emptyset$.

  Now let $[i,c] \in \A$ be distinguished and set $j = r(i+c)$.  Then
  $j \geq r(i+\la_i+1)$, so (\ref{eqn:Cfromr}) shows that $(i,j)
  \not\in \cC$.  Since $[i,c] \not\in R(\mu)$ we get $\mu_j > 2k+j-i-c
  \geq 2k+j-i-\mu_i$, hence $(i,j) \in \cC_{\ell+1}(\mu)$.  Using
  Lemma \ref{lem:CinDinCDC}, this establishes (a).  If $[i,c'] \in
  R(\mu)$ is any box with $c' > c$, then we must have $j < r(i+c')$,
  since otherwise $\mu_j \leq 2k+j-i-c' < 2k+j-i-c$.  This proves (b)
  and finishes the proof of (c).  If $(i,j) \in F$, then $\mu_i =
  \la_{i-1}$, $c = e_i(\mu)$, and $j = f_i(\mu)$.  This establishes
  part (d).

  In the situation of (e), notice that if $\mu_j = \la_{j-1}$, then
  $(i,j) \in E \cup F$ is associated to a distinguished box $[i,c] \in
  \A$.  We must have $c > k+1$, since otherwise $\la_i\leq k$, $i=m$,
  and $j = r(m+k+1) = i$.  Since $[i,c]\not\in R(\mu)$ and
  $(i,j-1)\not\in \cC$, we also have $c > 2k+j-i-\mu_j =
  2k+j-i-\la_{j-1} \geq \la_i+1$.  It follows that $[i,c-1] \in
  R(\mu)$.  Set $j' = r(i+c-1)$.  Then $j' \leq j$ and $\mu_{j'}-j'
  \leq 2k-i-c+1 \leq \mu_j-j$, which shows that $j'=j$ and $\mu_j =
  2k+j-i-c+1 < \la_{j-1}$.  This contradiction proves (e).
\end{proof}

To every subset $E'$ of $E$ we associate the set of pairs $S(E') :=
E'\cup F\cup G$.  This is a disjoint union, and there are exactly
$2^{\NN(\la,\mu)}$ sets of this form.  The following proposition
therefore completes the proof of Claim 1.

\begin{prop}
  Let $S \subset \Delta$ be any subset.  Then
  $(\cC_{\ell+1}(\mu),\mu,S,0) \in \Psi_0$ if and only if $S = S(E')$
  for some subset $E' \subset E$.
\end{prop}
\begin{proof}
  We first assume that $\psi = (\cC_{\ell+1}(\mu),\mu,S,0) \in
  \Psi_0$.  Lemmas \ref{lem:combine} and \ref{lem:assoc}(d) then imply
  that $F \subset S$.  We next show that $G = \{(i,j) \in S \mid \mu_j
  < \la_j \}$.  If $\mu_j < \la_j$ then $S$ contains a unique pair
  $(i,j)$ in column $j$.  Lemma~\ref{lem:neighbor} shows that
  $[j,\la_j]$ is $k$-related to a box in row $i$ of $\mu\ssm\la$, and
  condition (1) implies that no other box in $\mu\ssm\la$ is
  $k$-related to $[j,\la_j]$.  It follows that $(i,j)$ is also the
  unique pair of $G$ in column $j$.

  Let $(i,j) \in S \ssm G$.  We will show that $(i,j)$ is the pair
  associated to a distinguished box of $\A$.  If $i=j=m$, then $\la_m
  \leq k < \mu_m$ and $(m,m)$ is associated to the distinguished box
  $[m,k+1] \in \A$.  We can therefore assume that $i<j$, hence $\mu_i
  > \la_i$.  Since $(i,j) \not\in G$ we also have $\la_j \leq \mu_j$.
  If $\la_i+\mu_j \geq 2k+j-i$, then the inequality $\la_{j-1} \geq
  \mu_j > 2k+j-i-\la_i-1$ implies that $j \leq r(i+\la_i+1)$.  Since
  $(i,j) \not\in \cC$, it follows from (\ref{eqn:Cfromr}) that $j =
  r(i+\la_i+1)$.  We deduce that $[i,\la_i+1] \in \A$ is a
  distinguished box and $(i,j)$ is the associated pair.

  Otherwise we have $\la_i+\mu_j < 2k+j-i$.  In this case we set $c =
  2k+j-i-\mu_j$.  Since $(i,j) \in \cC_{\ell+1}(\mu)$ we have $\la_i <
  c < \mu_i$.  We also have $c > k$; if $i=m$ this follows because
  $\mu_j \leq \la_m \leq k$.  We claim that $\mu_j < \la_{j-1}$.  If
  $(i,j-1) \in \cC$, then this follows because $\mu_j < 2k+j-i-\la_i
  \leq \la_{j-1}$, so assume that $(i,j-1) \not\in \cC$.  Then we must
  have $j>m$, and $(i,j)$ was added to $S$ in Phase 2 of the
  algorithm.  By Lemma~\ref{phaseonelm} the first predecessor
  $(D',\mu',S',i)$ of $\psi$ of level $i$ satisfies that $\mu'_j \leq
  \la_{j-1}$.  Since $(i,j) \not\in S'$, this implies that $\mu_j <
  \la_{j-1}$, as claimed.

  The inequality $\la_{j-1} > \mu_j = 2k+j-i-c$ implies that $r(i+c)
  \geq j$, and since $\la_j \leq \mu_j$ we also have $r(i+c+1) \leq
  j$.  We deduce that $r(i+c) = r(i+c+1) = j$, $[i,c] \in R(\mu)$, and
  $[i,c+1] \in \A$.  This shows that $[i,c+1]$ is distinguished and
  $(i,j)$ is the associated pair.  We conclude that the set $E' := S
  \ssm (F \cup G)$ is a subset of $E$, hence $S = S(E')$ has the
  required form.

  Now let $E' \subset E$ be an arbitrary subset and set $S = S(E')$.
  We must show that $(\cC_{\ell+1}(\mu), \mu, S, 0) \in \Psi_0$.  Set
  $\nu = \prod_{(i,j)\in S} L_{ij}\, \mu$.  The definition of $S$
  ensures that $\nu \geq \la$, so $\nu\in \cN(\la,p)$.  We now
  construct a path $\cP$ in the substitution forest by applying the
  substitution rule of \S\ref{ss:subrule} repeatedly to the initial
  4-tuple $(\cC,\nu,\emptyset,\ell+1)$.  Whenever the substitution
  rule assigns two children to a 4-tuple $\psi'$ of $\cP$, we use the
  set $S$ to determine which child is to follow $\psi'$ on the path.
  More precisely, if $\psi' = (D',\mu',S',h')$ meets {\bf(i)} or
  {\bf(iii)} and has two children, and if $(i,j)$ is the outer corner
  being added to $D'$, then we choose the child $\psi'' =
  (D'',\mu'',S'',h')$ for which $S'' \ssm S' = S \cap \{(i,j)\}$.  We
  will show that $\cP$ terminates in the 4-tuple $(\cC_{\ell+1}(\mu),
  \mu, S, 0)$.

  For $h \geq 0$ we set $D_h = \cC \cup \{ (i,j) \in \cC_{\ell+1}(\mu)
  \mid i>h \text{ or } (j>h \text{ and } (i,j-1) \in \cC) \}$.
  Lemma~\ref{lem:CinDinCDC} implies that this is a valid set of pairs.
  We will say that the 4-tuple $\psi' = (D',\mu',S',h')$ is {\em
    good\/} if it satisfies $D_{h'} \subset D' \subset D_{h'-1}$ and
  $S' = S \cap D'$.  It is enough to show that if $\psi'$ is any
  good 4-tuple on $\cP$ with $h'>0$, then $\psi'$ has a good child
  that also belongs to $\cP$.

  Let $\psi'$ be a good 4-tuple of $\cP$.  We then have $\mu =
  \prod_{(i,j)\in S\ssm D'} R_{ij}\, \mu'$.  We first show that if
  $\psi'$ meets {\bf(i)} or {\bf(iii)}, and $(i,j)$ is the pair being
  added to $D'$, then $(i,j)$ is also in $D_{h'-1}$.  If $\mu_i+\mu_j
  = \mu'_i+\mu'_j$ then this is true because $\psi'$ satisfies
  $\W(i,j)$, and otherwise $S \ssm D'$ must contain at least one pair
  in row $i$ or column $j$, which implies that $(i,j) \in D_{h'-1}$ by
  Lemma~\ref{lem:assoc}(a).  On the other hand, assume that $D'
  \subsetneq D_{h'-1}$.  Then $D_{h'-1} \ssm D_{h'}$ contains an outer
  corner $(i,j)$ of $D'$.  If we choose $c \geq j$ maximal such that
  $(i,c) \in \cC_{\ell+1}(\mu)$, then the inequalities $\mu'_i +
  \mu'_j \geq \mu_i + \mu_j - (c-j) \geq \mu_i + \mu_c -c+j > 2k+j-i$
  show that $\psi'$ satisfies $\W(i,j)$.  If $h'=j$ then $\psi'$ meets
  {\bf(i)}, and otherwise we have $h'=i<j$ and $\psi'$ meets
  {\bf(iii)}.  Notice also that if $\psi'$ meets {\bf(iii)} and
  $\mu'_j > \mu'_{j-1}$, then we must have $(h',j) \in S$ since $\mu$
  is a partition.  These observations show that $\psi'$ meets {\bf(i)}
  or {\bf(iii)} if and only if $D' \subsetneq D_{h'-1}$, and in this
  case $\psi'$ is succeeded on $\cP$ by a good child.

  Now consider a good 4-tuple $\psi'$ of $\cP$ such that $D' =
  D_{h'-1}$.  It remains to show that the substitution rule simply
  decreases the level of $\psi'$, i.e.\ $\psi'$ does not meet
  {\bf(ii)}, {\bf(iv)}, or {\bf(v)}.  Assume that $\psi'$ meets
  {\bf(ii)} and choose $i \geq 1$ minimal such that $(i,h') \not\in
  D'$.  Then $\mu'_{h'} > \la_{h'-1}$ and $(i,h')$ is not an outer
  corner of $D'$.  We have $i<h'$ and $(i,h'-1) \not\in \cC$.  Using
  Lemma~\ref{lem:CinDinCDC} we deduce that $(i,h')\in S$ and $\mu_{h'}
  = \la_{h'-1}$; however this contradicts Lemma~\ref{lem:assoc}(e).

  If $\psi'$ satisfies condition $\X$, then $h' \leq m$ and $\mu'_{h'}
  = \mu_{h'} = \la_{h'-1}$.  It follows that $\A$ contains a
  non-optional distinguished box in row $h'$, and
  Lemma~\ref{lem:assoc}(d) implies that $(h',f_{h'}(\mu)) \in F$ is
  the associated pair.  But Lemma~\ref{Rflemma} shows that $f' =
  f_{h'}(\mu)$, so $(h',f') \in S \cap D_{h'-1} = S'$.  We conclude
  that condition $\X$ fails for $\psi'$.  In particular, $\psi'$ does
  not meet {\bf(v)}.  Finally, if $\psi'$ meets {\bf(iv)}, then
  $(h',g') \not\in \cC_{\ell+1}(\mu)$, and since $\mu'_{h'}=\mu_{h'}$
  and $\mu'_{g'}=\mu_{g'}$, we deduce that $\W(h',g')$ fails for
  $\psi'$.  This contradiction finishes the proof that the level of
  $\psi'$ is decreased.
\end{proof}

\begin{example}
\label{bigexample}
  Consider the partitions $\la = (22,21,18,16,14,7,5,4,3,3,1)$
  and $\mu = (25,21,19,17,15,14,6,5,3,2,2)$, and set $k=5$.
  Then $\la \to \mu$.  The diagrams of $\la$ and $\mu$ are displayed
  in Figure \ref{figureforbigexample}, with boxes from $R(\mu)$
  labeled with {\bf R} and distinguished boxes labeled with {\bf O}
  for optional and {\bf N} for non-optional.  The figure also shows
  the pairs in $\cC(\la)$ and $\cC_{\ell+1}(\mu)$,
  where $\ell=11$.  The pairs from $E$, $F$, and $G$ are labeled
  accordingly, and the one additional pair from $\cC_{\ell+1}(\mu)
  \ssm \cC(\la)$ is labeled {\bf D}.  The skew dotted lines
  help to identify the pairs in $E$ and $F$ associated to
  the distinguished boxes.  The compositions $\nu$ for which some
  4-tuple $(\cC_{\ell+1}(\mu), \mu, S, 0)$ originates from
  $(\cC,\nu,\emptyset,\ell+1)$ may or may not include the boxes
  labelled with question marks, which can be traded for boxes from the
  rows of corresponding optional distinguished boxes.  There are $2^5$
  such compositions $\nu$, and for each of them there are two sets
  $S$, one of which contains the diagonal pair $(7,7) \in E$.
  \vspace{3mm}

\begin{figure}
\centering
\[
\begin{array}{cc}
\includegraphics[scale=0.45]{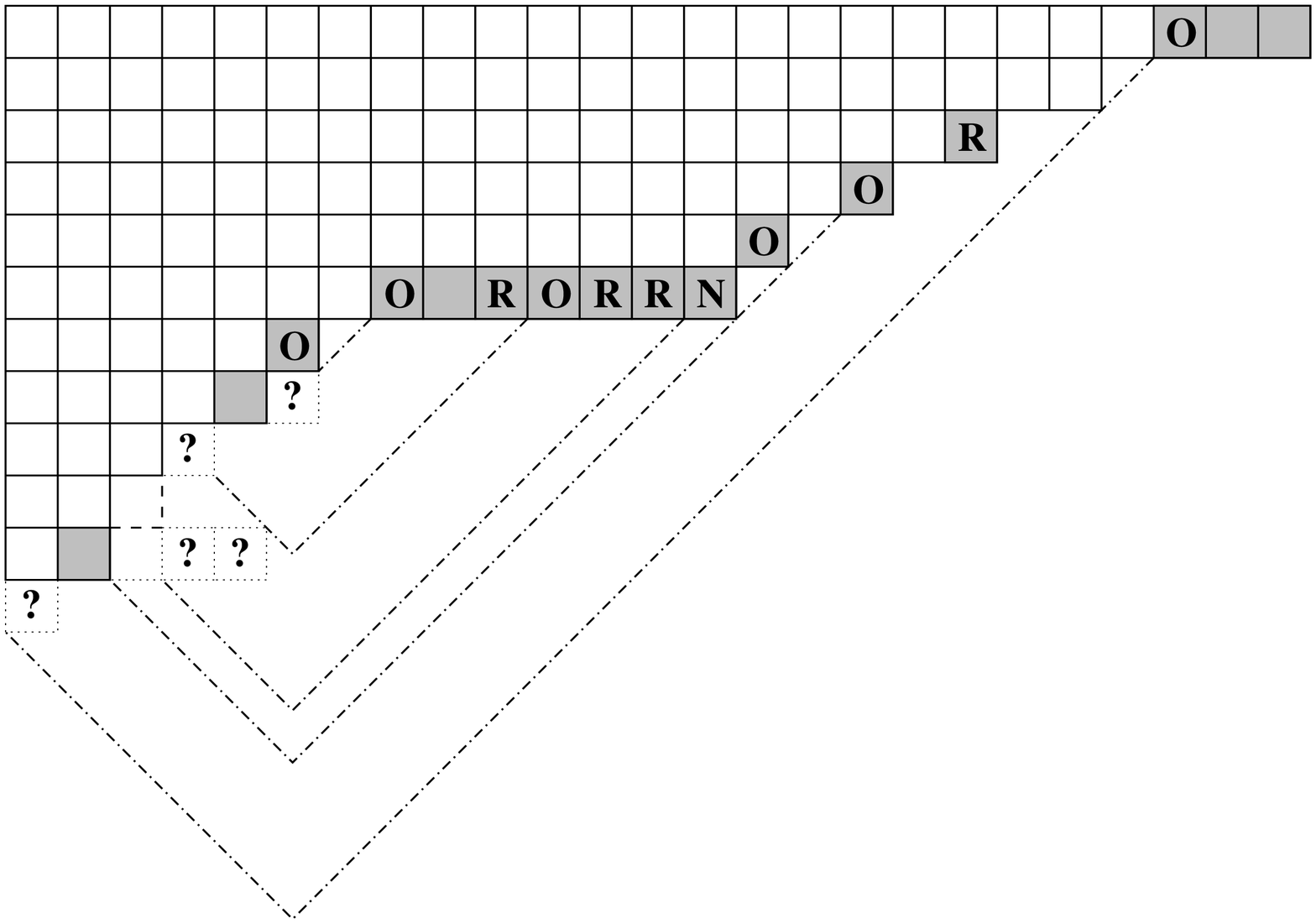} & \raisebox{140pt}{(i)} \\
\includegraphics[scale=0.65,viewport=0 10 50 50]{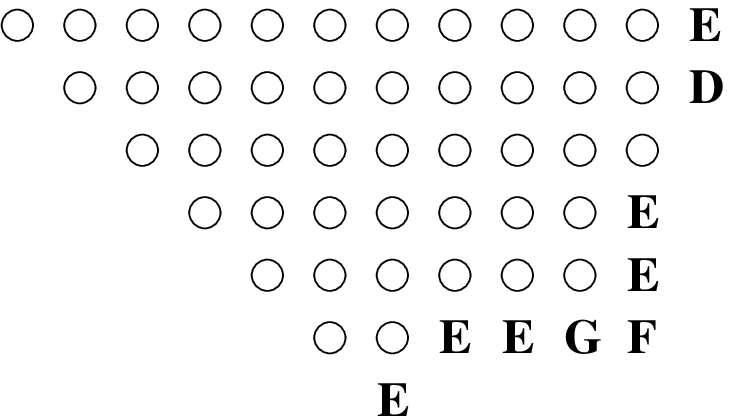} & \raisebox{26pt}{(ii)}
\end{array}
\]
\caption{(i) $\lambda$ and $\mu$, (ii) $\cC(\la)$ and
$\cC_{\ell+1}(\mu)$ in Example \ref{bigexample}.}
\label{figureforbigexample}
\end{figure}
\end{example}

\begin{remark}
  The most subtle ingredient of the Substitution Rule is the reference
  to condition $\W(h,g)$ in rule {\bf (iv)}.  In fact, if we modify
  the algorithm so that 4-tuples can meet {\bf(iv)} only when they
  satisfy condition $\X$, then Claim 1 still holds but Claim 2 fails.
  To see this, let $\psi = (D,\mu,S,h)$ be a 4-tuple that meets
  {\bf(iv)} but does not satisfy condition $\X$, and assume that the
  parent of $\psi$ does not meet {\bf(iv)}.  Then the modified
  algorithm replaces $\psi$ with $(D,\mu,S,h-1)$, and the arguments in
  the proofs of Lemmas \ref{lem:stickyx} and \ref{lem:ivdooms} can be
  used to show that the latter 4-tuple satisfies $\X$ and does not
  survive the modified algorithm.  It follows from this that Claim 1
  is true.  However, if the modified algorithm is applied with $\la =
  (4,3,1,1)$, $p=6$, and $k=1$, then the resulting set $\Psi_1$ of
  4-tuples meeting {\bf(ii)} or {\bf(v)} contains 1119 elements, and
  543 of these have non-zero evaluations.  This implies that there is
  no involution $\iota : \Psi_1 \to \Psi_1$ such that $\ev(\psi) +
  \ev(\iota(\psi)) = 0$ for every $\psi\in\Psi_1$.
\end{remark}

\subsection{}
\label{sec5}

In this section we construct a sign-reversing involution $\iota :
\Psi_1 \to \Psi_1$ and show that it has the properties required by
Claim 2.

Given a valid 4-tuple $\psi = (D,\mu,S,h)$ with $h \geq 2$ and $\mu_h
\geq \la_{h-1}$, we define a new 4-tuple $\iota\psi$ as follows.  If
$(h,h) \notin D$, then set $\iota\psi = (D,\wt\mu,S,h)$, where the
composition $\wt\mu$ is defined by $\wt\mu_{h-1} = \mu_h-1$,
$\wt\mu_h = \mu_{h-1}+1$, and $\wt\mu_t = \mu_t$ for $t \notin
\{h-1,h\}$.  If $(h,h) \in D$ and $\mu_{h-1} = \mu_h$, then set
$\iota\psi = \psi$.  Assume that $(h,h) \in D$ and $\mu_{h-1} \neq
\mu_h$.  Let $\varpi$ be the involution of $\Delta$ that exchanges
$(h-1,g)$ with $(h,f)$, and fixes all other pairs.  Then set
$\iota\psi = (D,\wt\mu,\wt S,h)$, where $\wt S = \varpi(S)$,
and $\wt\mu$ is the composition obtained from $\mu$ by switching
the parts $\mu_{h-1}$ and $\mu_h$.

\begin{lemma}\label{lem:ii}
  If $\psi\in \Psi_1$, then $\iota(\iota\psi)=\psi$ and $\ev(\psi) +
  \ev(\iota\psi) = 0$.
\end{lemma}
\begin{proof}
  Assume that $(h,h)\notin D$.  Then $\iota(\iota\psi)=\psi$ is clear,
  and since $\psi$ meets {\bf (ii)}, it follows from Lemma
  \ref{commuteA} that $\ev(\psi) + \ev(\iota\psi) = 0$.  Next, assume
  that $(h,h)\in D$.  Then $\psi$ meets {\bf(v)} and
  Lemma~\ref{lem:hgbox} shows that $(h,g) \in D$.  It follows that
  $\wt S \subset D$ and $\iota\psi$ is a valid 4-tuple.  Lemma
  \ref{Rflemma} implies that the same values of $\e$, $\f$, and $\g$
  are assigned to $\psi$ and $\iota\psi$, so we have $\iota(\iota\psi)
  = \psi$.  Finally, since $\psi$ satisfies condition $\X$, we have
  $\mu_h \geq \la_{h-1} > k$, so it follows from Lemma~\ref{commuteC}
  that $\ev(\psi) + \ev(\iota\psi) = 0$.
\end{proof}

Given any valid 4-tuple $\psi = (D,\mu,S,h)$, a valid set of pairs $D'
\subset D$, and an integer $h' \geq h$, we define the 4-tuple
$\psi(D',h') = (D',\prod_{(i,j)\in S\ssm D'} L_{ij} \mu, S\cap D',
h')$.  Notice that if $\psi$ occurs in the substitution forest, then
every predecessor of $\psi$ can be written as $\psi(D',h')$.  In
particular, the initial 4-tuple leading to $\psi$ is
$\psi(\cC,\ell+1)$.

To finish the proof of Claim 2, we will show that $\iota(\Psi_1)
\subset \Psi_1$.  Fix an element $\psi = (D,\mu,S,h) \in \Psi_1$.  We
will show that the substitution forest has a path leading to
$\iota\psi$ and that this 4-tuple meets {\bf(ii)} or {\bf(v)}.  Define
compositions $\nu$ and $\wt\nu$ by $(\cC, \nu, \emptyset, \ell+1) =
\psi(\cC,\ell+1)$ and $(\cC, \wt\nu, \emptyset, \ell+1) =
\iota\psi(\cC,\ell+1)$.

The following lemma shows that $\iota\psi \in \Psi_1$ whenever $\psi$
meets {\bf(ii)}.  We will say that two valid 4-tuples {\em meet the
same rule}, if both meet the same rule among {\bf(i)}--{\bf(v)} in the
substitution rule, or if the substitution rule decreases the level of
both 4-tuples.

\begin{lemma}
  Assume that $\psi \in \Psi_1$ meets {\bf(ii)}.\smallskip

  \noin{\em(a)} We have $\wt\nu \in \cN(\la,p)$.\smallskip

  \noin{\em(b)} Let $\psi' = \psi(D',h')$ be any predecessor of $\psi$.
  Then $\iota\psi(D',h')$ meets the same rule as $\psi'$.  In
  particular, $\iota\psi$ meets {\bf(ii)}.
\end{lemma}
\begin{proof}
  Notice first that $h>m$, as $D$ has no outer corner in column $h$
  and $(h,h)\notin D$.  Part (a) is true because $\wt\nu_{h-1} =
  \wt\mu_{h-1} = \mu_h-1 \geq \la_{h-1}$, $\wt\nu_h \geq \wt\mu_h =
  \mu_{h-1}+1 > \la_{h-1} \geq \la_h$, and $\wt\nu_t = \nu_t \geq
  \la_t$ for $t \notin \{h-1,h\}$.  The inequality $\wt\mu_h >
  \la_{h-1}$ implies that $\iota\psi$ meets {\bf(ii)}.  Let
  $\psi(D',h')$ be a strict predecessor of $\psi$.  If $h'>h$ and
  $(i,h')$ is an outer corner of $D'$, then $i \leq m-1 < h-1$ and
  $\W(i,h')$ holds for $\iota\psi(D',h')$ if and only if it holds for
  $\psi(D',h')$.  Part (b) follows from this when $h'>h$, and when
  $h'=h$ it follows from Lemma~\ref{phaseonelm}.
\end{proof}

From now on we assume that $\psi$ meets {\bf(v)} and that $\mu_{h-1}
\neq \mu_h$, so that $\iota\psi \neq \psi$.  We have $(h,h) \in D$,
$\psi$ satisfies condition $\X$, and $\mu_h \geq \la_{h-1}$.  We also
have $\wt\mu_h = \mu_{h-1} \geq \la_{h-1}$, and in case of equality
$\mu_{h-1}=\la_{h-1}$ we must have $(h-1,g)\notin S$ or equivalently
$(h,f) \notin \wt S$.  This shows that $\iota\psi$ satisfies
condition $\X$.

\begin{lemma}
  We have $\wt\nu \in \cN(\la,p)$.
\end{lemma}
\begin{proof}
  Notice that $\nu \geq \la$ and $\wt\nu_i = \nu_i$ for $i \notin
  \{h-1,h,f,g\}$. Observe also that row $h-1$ of $D \ssm \cC$ contains
  the single pair $(h-1,g)$.  If $\wt\mu_{h-1} > \la_{h-1}$ we
  therefore obtain $\wt\nu_{h-1} \geq \wt\mu_{h-1}-1 \geq \la_{h-1}$.
  Otherwise we have $\mu_h = \wt\mu_{h-1} = \la_{h-1}$, and since
  $\mu_{h-1}\neq\mu_h$ and $\psi$ satisfies condition $\X$, it follows
  that $(h,f) \notin S$, so $(h-1,g) \notin \wt S$ and $\wt\nu_{h-1} =
  \la_{h-1}$.  Using Lemma~\ref{lambdabg} we also obtain $\wt\nu_h
  \geq \wt\mu_h - (g-b+1) \geq \la_h$.  Notice that if $h < f=g$,
  then $\wt\nu_g = \nu_g \geq \la_g$, so we may assume that $f<g$.
  Lemma~\ref{fglemma} then implies that $\wt\nu_g \geq \wt\mu_g =
  \mu_g = \la_{g-1} \geq \la_g$.  Using that $[h,e] \notin R$ and the
  definition of $f$, we obtain $\mu_f \geq 2k+f+1-h-e \geq \la_f$.  If
  $f > h$ then we also have $\wt\nu_f \geq \wt\mu_f$; when $h < m = f$
  this is true because $(h,f)\notin\cC$ implies $(f,f+1)\notin D$.  We
  conclude that $\wt\nu_f \geq \wt\mu_f = \mu_f \geq \la_f$, as
  required.
\end{proof}

For $t \in \N$ we define the valid set of pairs $D_t = \cC \cup
\{(i,j) \in D \mid i < t\}$.  Let $z \geq 1$ be the smallest positive
integer for which $(z,g)\notin\cC$ and $\psi(D_z,g)$ does not meet
{\bf(i)}, and choose $\wt z \geq 1$ minimal such that $(\wt
z,g)\notin\cC$ and $\iota\psi(D_{\wt z},g)$ does not meet {\bf(i)}.
We also write $z_1 = \min(z,\wt z)$ and $z_2 = \max(z,\wt z)$, and
define the sets $F = \{(i,j) \in D\ssm\cC \mid j<g\}$ and $G = \{(i,g)
\mid z_1 \leq i < z_2 \}$.  Let $h_0$ be maximal such that $(h_0,g)
\in D$.  Notice that $h \leq h_0$, $D_{z_2} = D_{z_1}\cup G$, and
$G\subset\partial^1\cC$.

\begin{lemma}\label{lem:valpred}
  {\em(a)} Both $\psi(D_z,g)$ and $\psi(D_z\cup F,h_0)$ are
  predecessors of $\psi$.\smallskip

  \noin{\em(b)} We have $z_2 \leq h_0+1$.
\end{lemma}
\begin{proof}
  All pairs $(i,j) \in D\ssm\cC$ with $i < z$ were added to $D$ in
  Phase 1 of the algorithm, while all pairs $(i,g) \in D\ssm\cC$ with
  $i \geq z$ were added in Phase 2.  Lemma~\ref{lem:ivdooms} implies
  that the latter pairs were added to predecessors of $\psi$ of level
  $h_0$, and this happened after all pairs of $F$ were added.  Part
  (a) follows from this.

  For (b) we may assume that $(h_0+1,g) \in \partial^1\cC$.  Since
  $\wt\mu_{h_0+1} = \mu_{h_0+1}$ and $\wt\mu_g = \mu_g$, it is enough
  to show that $\W(h_0+1,g)$ fails for $\psi$.  If $h_0+1 < m$, then
  this follows because the most recent predecessor of $\psi$ of level
  $h_0+1$ does not meet {\bf(iii)} or {\bf(iv)}.  Otherwise we have
  $h_0+1=g=m$, in which case the inequality $\mu_m \leq k$ follows
  because $\W(z,m)$ fails for $\psi(D_z,m)$.
\end{proof}

Given two valid 4-tuples $\psi'$ and $\psi''$, we write $\psi' \leq
\psi''$ if $\psi'$ is a predecessor of $\psi''$, with $\psi'=\psi''$
allowed.  We will write $\psi' < \psi''$ if $\psi'\leq\psi''$ and
$\psi'\neq\psi''$.  The next proposition implies that $\iota\psi$
appears in the substitution forest; in fact we have
$\iota\psi(\cC,\ell+1) \leq \iota\psi(D_{\wt z},g) \leq
\iota\psi(D_{\wt z}\cup F,h_0) \leq \iota\psi$.

\begin{prop}\label{prop:invol}
  Let $\psi' = \psi(D',h')$ be a predecessor of $\psi$.\smallskip

  \noin{\em(a)} If $\psi' < \psi(D_{z_1},g)$ then $\iota\psi(D',h')$
  meets the same rule as $\psi'$.\smallskip

  \noin{\em(b)} If $z_1 \leq t < {\wt z}$, then
  $\iota\psi(D_t,g)$ meets {\bf(i)}.\smallskip

  \noin{\em(c)} If $\psi(D_{z},g) \leq \psi' <
  \psi(D_{z}\cup F,h_0)$, then $\iota\psi(D' \,\triangle\, G,
  h')$ meets the same rule as $\psi'$.  Here $D' \,\triangle\, G =
  (D'\cup G) \ssm (D'\cap G)$ is the symmetric difference.\smallskip

  \noin{\em(d)} If ${\wt z} \leq t \leq h_0$, then
  $\iota\psi(D_t\cup F,h_0)$ meets {\bf(iii)} or {\bf(iv)}.\smallskip

  \noin{\em(e)} If $\psi(D,h_0) \leq \psi'$, then $\iota\psi(D',h')$
  meets the same rule as $\psi'$.
\end{prop}

\begin{example}\label{ex:invol}
  Let $\la=(3,1)$, $p=4$, and $k=1$.  Then $\cC=\{11,12\}$, and the
  4-tuple $\psi = (D,341,\{22\},2) \in \Psi_1$ meets {\bf(v)}, where
  $D=\{11,12,13,22,23\}$.  We compute $f=2$, $g=3$, and $\iota\psi =
  (D,431,\{13\},2)$.  We also have $z = 1$, $\wt z = 2$, $h_0=2$, $F =
  \{22\}$, and $G = \{13\}$.  The paths of the substitution forest
  leading to $\psi$ and $\iota\psi$ are displayed below.  Notice how
  the path from $\psi(D_z,g)$ to $\psi(D_z\cup F,h_0)$ translates to a
  path from $\iota\psi(D_{\wt z},g)$ to $\iota\psi(D_{\wt z}\cup
  F,h_0)$, and the path from $\psi(D_z\cup F,h_0)$ to $\psi(D_{\wt
    z}\cup F,h_0)$ translates to a path from $\iota\psi(D_z,g)$ to
  $\iota\psi(D_{\wt z},g)$.\bigskip \medskip

  \psfrag{(i)}{{\bf(i)}} \psfrag{(iv)}{{\bf(iv)}}
  \psfrag{(iii)}{{\bf(iii)}}
  \psfrag{L1}{\hspace{-25mm}$\scriptstyle\psi(D_z,g) \ = \
    (\{11,12\},341,\emptyset,3)$}
  \psfrag{L2}{\hspace{-11mm}$\scriptstyle(\{11,12\},341,\emptyset,2)$}
  \psfrag{L3}{\hspace{-35mm}$\scriptstyle\psi(D_z\cup F,h_0) \ = \
    (\{11,12,22\},341,\{22\},2)$}
  \psfrag{L4}{\hspace{-37mm}$\scriptstyle\psi(D_{\wt z}\cup F,h_0) \ =
    \ (\{11,12,13,22\},341,\{22\},2)$}
  \psfrag{L5}{\hspace{-24mm}$\scriptstyle\psi \ = \
    (\{11,12,13,22,23\},341,\{22\},2)$}

  \psfrag{R1}{\hspace{-10mm}$\scriptstyle(\{11,12\},332,\emptyset,3) \
    = \ \iota\psi(D_z,g)$}
  \psfrag{R2}{\hspace{-14mm}$\scriptstyle(\{11,12,13\},431,\{13\},3) \
    = \ \iota\psi(D_{\wt z},g)$}
  \psfrag{R3}{\hspace{-14mm}$\scriptstyle(\{11,12,13\},431,\{13\},2)$}
  \psfrag{R4}{\hspace{-16mm}$\scriptstyle(\{11,12,13,22\},431,\{13\},2)
    \ = \ \iota\psi(D_{\wt z}\cup F,h_0)$}
  \psfrag{R5}{\hspace{-16mm}$\scriptstyle(\{11,12,13,22,23\}, 431,
    \{13\}, 2) \ = \ \iota\psi$}
  \mbox{}\hspace{36mm}\pic{.25}{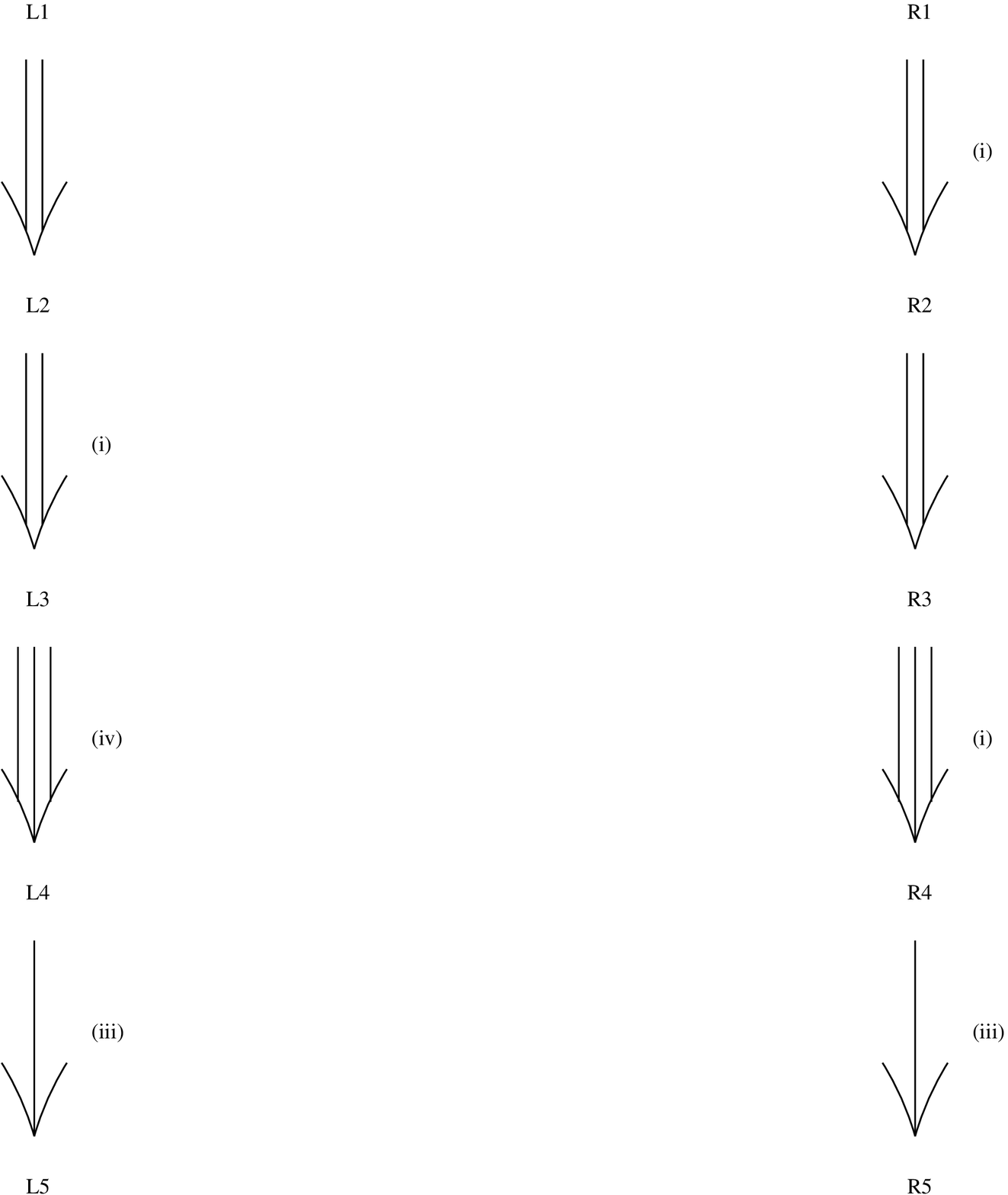}
\end{example}
\medskip

The proof of Proposition~\ref{prop:invol} is based on the following
lemmas.

\begin{lemma}\label{lem:ieq25c}
  If $h_0 < h' \leq m$ and $g_{h'} = g$, then $\wt\mu_{h'} = \mu_{h'}
  < \la_{h'-1}$.
\end{lemma}
\begin{proof}
  Let $\psi' = (D',\mu',S',h')$ be the most recent predecessor of
  $\psi$ of level $h'$, and notice that $\wt\mu_{h'} = \mu_{h'} =
  \mu'_{h'}$.  If $(h',h') \notin D'$ then $h'=m$, and since $\psi'$
  does not meet {\bf(i)} we obtain $\mu'_m \leq k < \la_{m-1}$.  We
  may therefore assume that $(h',h') \in D'$.  Since condition $\X$
  fails for $\psi'$ by Lemma~\ref{lem:stickyx}, we have $\mu'_{h'}
  \leq \la_{h'-1}$.  Suppose that $\mu'_{h'} = \la_{h'-1}$.  Then
  $(h',f_{h'}(\mu')) \in S'$, and since $(h',g)\notin D$ we must have
  $f_{h'}(\mu') < g = g_{h'}$.  Lemma~\ref{fglemma} then implies that
  $\mu'_g = \la_{g-1}$, and since $(h'-1,g-1) \in \cC$ we deduce that
  $\W(h',g)$ holds for $\psi'$.  But then Lemma~\ref{lem:hgbox}(a)
  implies that $\psi'$ meets {\bf(iii)} or {\bf(iv)}, contradicting
  the choice of $\psi'$.  We conclude that $\mu'_{h'} < \la_{h'-1}$,
  as required.
\end{proof}

In the next three lemmas, we let $\psi(D',h') = (D',\mu',S',h')$
denote a predecessor of $\psi$, let $\ov D$ be a valid set of pairs
such that $D'\ssm G \subset \ov D \subset D' \cup G$, and write
$\iota\psi(\ov D,h') = (\ov D,\ov\mu,\ov S, h')$.

\begin{lemma}\label{lem:cancelweight}
  Choose $(i,j)\in \partial\cC\ssm \ov D$ with $j \neq g$.  Assume
  that {\em (a)} $(i,j-1)\in\cC$ or $i \geq h$, and $(i,j)$ is an
  outer corner of $\ov D$, or {\em (b)} $(i,j) = (h',g_{h'}) \neq
  (m,m)$.  Then $\psi(D',h')$ satisfies $\W(i,j)$ if and only if
  $\iota\psi(\ov D,h')$ satisfies $\W(i,j)$.
\end{lemma}
\begin{proof}
  Notice first that $i<g$, since otherwise $i=m=g<j$ and
  $(i,j)\notin\partial\cC$.  We also have $j \geq z_2$, as otherwise
  $z_2-1=m=j$, $(m,g)$ is an outer corner of $D_m$, and $g=m$.
  Furthermore, if $z_1 \leq i < z_2$, then we obtain
  $(i,g)\in\partial^1\cC$, $j>g$, $i<h$, and $(i,j-1) \in \cC$, which
  is a contradiction.  This shows that $i,j \notin
  \{z_1,z_1+1,\dots,z_2-1,g\}$, so condition $\W(i,j)$ holds for
  $\iota\psi(\ov D,h')$ if and only if it holds for
  $\iota\psi(D',h')$.  Henceforth we will work with the latter
  4-tuple, which we denote $\iota\psi(D',h') = (D',\wt\mu',\wt
  S',h')$.  We then have $\mu'_t = \wt\mu'_t$ for $t \notin
  \{h-1,h,f,g\}$.

  Assume first that $i=h$.  Then we have $m \leq j < g$, and $(h,j)$
  is an outer corner of $D'$.  Since $(h,j)\in D$ it follows that
  $\W(h,j)$ holds for $\psi(D',h')$.  We must prove that also
  $\iota\psi(D',h')$ satisfies $\W(h,j)$.  If $j=h$, then we have
  $i=j=h=m$, and since $(h-1,g-1)\in\cC$ we obtain
  \[
  \wt\mu'_h \geq \wt\mu_h-g+h = \mu_{h-1}-g+h \geq \la_{h-1}-g+h >
  2k-\la_{g-1} \geq k \,.
  \]

  We can therefore assume that $i=h<j$.  We have $\wt\mu'_j =
  \wt\nu_j \geq \la_j$, and since $\iota\psi$ satisfies condition
  $\X$ we also have $\wt\mu'_h+g-j+1 \geq \wt\mu_h \geq \la_{h-1}$.
  Since $(h-1,g-1)\in\cC$ we obtain 
  \[
  \wt\mu'_h+\wt\mu'_j \geq
  \la_{h-1}+\la_j-g+j-1 \geq \la_{h-1}+\la_{g-1}-g+j-1 \geq 2k+j-h \,.
  \]
  If $\W(h,j)$ fails for $\iota\psi(D',h')$, then we must have
  equality $\wt\mu'_h+g-j+1 = \wt\mu_h = \la_{h-1}$, hence $(h,t) \in
  \wt S$ for $j \leq t \leq g$.  Since $\la_{h-1} \leq
  \wt\mu_{h-1} \neq \wt\mu_h$ we also obtain $\wt\mu_h <
  \wt\mu_{h-1}$.  Since $\iota\psi$ satisfies condition $\X$, we
  deduce that $(h,f) \notin \wt S$.  But then $f<g$, and
  Lemma~\ref{fglemma} implies that $(h,g)\notin S$ and hence $(h,g)
  \notin \wt S$, a contradiction.

  We next show that if $i \neq h$, then the identities
  $\mu'_i=\wt\mu'_i$ and $\mu'_j=\wt\mu'_j$ hold.  If $i<h$, then
  $j>g$ and $(i,j-1)\in\cC$.  Since this implies that $i<h-1$, the
  identities are true in this case.  Otherwise we have $h < i \leq j
  \leq b \leq f$, and the identities are clear unless $j=f$.  In this
  case we have $b=f<g$.  Notice also that $(h,f)\in D'$; in case (a)
  this is true because $(i,f)$ is an outer corner of $D'$, and in case
  (b) it follows from Corollary~\ref{cor:addorder} because $h'=i<j=f$
  and $(h,f) \in D\cap \partial^1\cC$.  The sets of pairs $S \ssm D'$
  and $\wt S \ssm D'$ therefore agree in column $f$, and since $\mu_f
  = \wt\mu_f$ by construction, we deduce that $\mu'_f = \wt\mu'_f$, as
  required.
\end{proof}

\begin{lemma}\label{lem:cancel2}
  The 4-tuple $\iota\psi(\ov D,h')$ does not meet {\bf(ii)}.
\end{lemma}
\begin{proof}
  Suppose that $\iota\psi(\ov{D},h')$ meets {\bf(ii)}.  Without loss
  of generality, we may also assume that $\psi(D',h')$ is the most
  recent predecessor of $\psi$ of level $h'$.  We have $h'>m$, and
  since $D\ssm\ov{D}$ contains at most one pair in column $h'$, we
  obtain $\la_{h'-1} \leq \ov\mu_{h'}-1 \leq \wt\mu_{h'} = \mu_{h'}
  \leq \mu'_{h'}$.  Lemma~\ref{phaseonelm} therefore implies that
  $\ov\mu_{h'}-1 = \mu_{h'} = \mu'_{h'} = \la_{h'-1}$.  In particular,
  we have $h' \in \{f,g\}$.

  Assume that $h'=f$ and $(h,f-1)\notin\cC$.  Then $[h,\la_{h-1}]
  \notin R(\mu)$, since otherwise $f=r(h+\la_{h-1})$ and $\mu_f \leq
  2k+f-h-\la_{h-1} < \la_{f-1}$.  We also have $e>k+1$, since
  otherwise $\la_h\leq k$, $h=m$, and $f=r(m+k+1)=m$.  Since
  $[h,e]\notin R(\mu)$ and $(h,f-1)\notin\cC$, we obtain $e >
  2k+f-h-\mu_f = 2k+f-h-\la_{f-1} \geq \la_h+1$.  It follows that
  $[h,e-1] \in R(\mu)$.  Set $f_1 = r(h+e-1)$.  Then $f_1 \leq f$ and
  $\mu_{f_1}-f_1 \leq 2k-h-e+1 \leq \mu_f-f$.  Since
  Lemma~\ref{noragged} implies that $\mu_f \leq \mu_{f_1}$, we
  therefore obtain $f_1 = f$ and $\mu_f = 2k+f-h-e+1 < \la_{f-1}$, a
  contradiction.
  
  In view of the above, the absence of an outer corner in
  column $h'$ of $\ov{D}$ implies that either $h'=f$ and $(h,f)\in
  \ov{D}$, or $h'=g$ and $(h-1,g)\in\ov{D}$.  It follows that the sets
  of pairs $S\ssm\ov{D}$ and $\wt S\ssm\ov{D}$ agree in column $h'$,
  so $\ov\mu_{h'} \leq \mu'_{h'}$.  This contradiction finishes the
  proof.
\end{proof}

\begin{lemma}\label{lem:cancelX}
  If $h'>h_0$, or if $h'>h$ and $\ov D = D'$, then $\iota\psi(\ov
  D,h')$ does not satisfy condition $\X$.
\end{lemma}
\begin{proof}
  We first show that $\mu'_h > \la_h$ and $\ov\mu_h > \la_h$.  In
  fact, Lemma~\ref{lambdabg} implies that $\mu_h = \wt\mu_{h-1} \geq
  \la_{h-1} \geq \la_h+g-b+1$.  If $\la_h \geq \mu'_h$, then we must
  have $\wt\mu_{h-1} = \la_{h-1}$ and $(h,j) \in S$ for $b \leq j \leq
  g$.  The equality shows that $(h-1,g) \notin \wt S$, which in turn
  implies that $(h,f) \notin S$, a contradiction.  The inequality
  $\ov\mu_h > \la_h$ is proved by interchanging $\psi(D',h')$ and
  $\iota\psi(\ov D,h')$.
  
  For proving the lemma we may assume that $(h',h') \in \ov{D}$ and
  $\ov{\mu}_{h'} \geq \la_{h'-1}$.  The assumptions imply that
  $\ov{D}$ and $D'$ agree in all rows $i$ with $i \geq h'$.  In
  particular $(h',h') \in D'$.  Since $\wt\mu_{h'} = \mu_{h'}$ and
  $\wt S$ agrees with $S$ in row $h'$, we deduce that $\ov{\mu}_{h'} =
  \mu'_{h'}$.  Since $\psi(D',h')$ does not satisfy condition $\X$ by
  Lemma~\ref{lem:stickyx}, we obtain $\ov{\mu}_{h'} = \mu'_{h'} =
  \la_{h'-1} < \mu'_{h'-1}$ and $(h',f_{h'}(\mu')) \in S'$.  We claim
  that $\ov{\mu}_{h'} < \ov{\mu}_{h'-1}$.  If $h'-1=h$, then this is
  true because $\ov{\mu}_{h'} = \la_h < \ov{\mu}_h$.  On the other
  hand, if $h'-1>h$ and the claim fails, then $\ov{D}$ and $D'$ differ
  in row $h'-1$, and this implies that $h'>h_0$ and $(h'-1,g)\in G$,
  hence $h'-1=h_0>h$.  Since $(h_0,g) \in G \subset \partial^1\cC$ we
  obtain $g_{h'} = b_{h_0} = g$, so Lemma~\ref{lem:ieq25c} implies
  that $\mu'_{h'} \leq \mu_{h'} < \la_{h'-1}$, a contradiction.
  
  We claim that $\ov{\mu}_{g_{h'}} = \mu'_{g_{h'}}$.  Notice that
  $g_{h'} \leq f$, and the claim is clearly true if $g_{h'} < f$.  If
  $g_{h'} = f < g$, then Corollary~\ref{cor:addorder} implies that
  $(h,f)$ was added to $D'$ by rule {\bf(i)}, hence $\ov{\mu}_f =
  \mu'_f$.  If $g_{h'} = f = g$, then Lemma~\ref{lem:ieq25c} implies
  that $h' \leq h_0$ and $\ov{D}=D'$, so the claim is clear unless
  $(h-1,g)\in D'$ and $(h,g)\notin D'$.  In this case the inclusion
  $(h',f_{h'}(\mu')) \in D'$ implies that $f_{h'}(\mu') < g$.
  Lemma~\ref{fglemma} then implies that $\mu'_g = \la_{g-1}$, and
  since $(h,g-1)\in\cC$ we obtain $\mu'_h + \mu'_g > \la_h + \la_{g-1}
  \geq 2k+g-h$.  Since the outer corner $(h,g)$ of $D'$ was not added
  by {\bf(i)}, it follows that $(h-1,g)$ was added by {\bf(iv)} when
  the substitution rule was applied to the parent of $\psi(D',h')$.
  By applying Lemma~\ref{lem:ivdooms}(c) to the parent of
  $\psi(D',h')$ and using that $h'>h$, we obtain $(h-1,g)\notin S$ and
  $(h,g) \notin S$.  This shows that $\mu'_g = \ov{\mu}_g$, proving
  the claim.  Finally, Lemma~\ref{Rflemma} implies that
  $f_{h'}(\ov{\mu}) = f_{h'}(\mu')$, therefore $(h',f_{h'}(\ov\mu))
  \in \ov{S}$, and we conclude that $\iota\psi(\ov{D},h')$ does not
  satisfy $\X$.
\end{proof}

\begin{proof}[Proof of Proposition~\ref{prop:invol}]
  Write $\psi' = \psi(D',h') = (D',\mu',S',h')$.

  (a).  If $\psi' < \psi(D_{z_1},g)$, then $(h',h')\notin D'$.  If
  $h'>g$, then Lemmas \ref{lem:cancelweight} and \ref{lem:cancel2}
  imply that $\iota\psi(D',h')$ meets the same rule as $\psi'$, and if
  $h'=g$, then both $\psi'$ and $\iota\psi(D',h')$ meet {\bf(i)}.

  (b).  This part follows from the definition of $\wt z$.

  (c).  Set $\ov{D} = D' \,\triangle\, G$, which is a valid set of
  pairs.  If $\psi'$ meets {\bf(i)} or {\bf(iii)}, then the pair
  $(i,j)$ that is added to $D'$ belongs to $F$, so $j<g$, and it
  follows from Lemma~\ref{lem:cancelweight} that
  $\iota\psi(\ov{D},h')$ meets the same rule as $\psi'$.  Otherwise,
  the substitution rule decreases the level of $\psi'$, so $h' \geq
  h_0+1 \geq \wt z$.  If $h'=g$, then $D' = D_z$, and it follows from
  Lemma~\ref{lem:cancel2} that the level of $\iota\psi(\ov{D},h') =
  \iota\psi(D_{\wt z},g)$ is decreased.  We may therefore assume that
  $h' < g$.  If $(h',h')\notin D'$, then Lemmas \ref{lem:cancelweight}
  and \ref{lem:cancel2} imply that the level of $\iota\psi(\ov{D},h')$
  is decreased, so assume that $(h',h') \in D'$.
  Lemma~\ref{lem:cancelX} implies that $\iota\psi(\ov{D},h')$ does not
  satisfy condition $\X$.

  Write $\iota\psi(\ov{D},h') = (\ov{D},\ov{\mu},\ov{S},h')$.  If the
  level of $\iota\psi(\ov{D},h')$ is not decreased, then this 4-tuple
  meets {\bf(iii)} or {\bf(iv)}, and Lemma~\ref{lem:cancelweight}
  implies that a pair from column $g$ is added to $\ov{D}$.  It
  follows that $\ov{D}=D_{\wt z} \cup F$ and $\iota\psi(\ov{D},h')$
  satisfies $\W(h',g)$.  Notice that $(\wt z,g) \in \partial^1\cC$,
  since otherwise we obtain $\wt z=h'=h_0+1$, $\ov{D}=D$, and
  $\ov{\mu}=\wt\mu$, hence $\psi$ satisfies $\W(h',g)$, and this
  contradicts that the most recent predecessor of $\psi$ of level $h'$
  does not meet {\bf(iii)} or {\bf(iv)}.  The definition of $\wt z$
  therefore implies that $\W(\wt z,g)$ fails for $\iota\psi(D_{\wt
  z},g)$, and the same is true for $\iota\psi(\ov{D},h')$.  Since the
  latter 4-tuple satisfies $\W(h',g)$, we deduce that $\ov\mu_{\wt z}
  - \ov\mu_{h'} < h' - \wt z$.  Now notice that $\la_{\wt z} \leq
  \ov\mu_{\wt z}$; if $\wt z = m$ then this follows because $h'=m<g$
  and hence $(m-1,m) \in \cC$.  Lemma~\ref{lem:ieq25c} therefore
  implies that $\la_{\wt z} - \la_{h'-1} < \ov\mu_{\wt z} -
  \wt\mu_{h'} < h'-\wt z$, which contradicts the fact that $\la$ is
  $k$-strict.
  
  (d).  We first show that $(h_0,h_0) \in D_t\cup F$.  If this is
  false, then we must have $h_0=g=m$, so the pair $(m,m)$ was added to
  $D$ by rule {\bf(i)}.  This implies that $\mu_m>k$ and therefore
  $\wt\mu_m>k$.  Since $\wt\mu_i \geq \la_i \geq k+m-i$ for all $i<m$,
  it follows that $\iota\psi$ satisfies $\W(i,m)$ for $1 \leq i \leq
  m$, hence ${\wt z} = m+1 > h_0$, a contradiction.

  Write $\iota\psi(D_t\cup F,h_0) = (D_t\cup F,\ov\mu,\ov S,h_0)$ and
  assume that both $\W(h_0,g)$ and $\X$ fail for this 4-tuple.  Then
  $\W(h_0,g)$ also fails for $\iota\psi$.  If $h<h_0$, then
  $\W(h_0,g)$ fails for $\psi$ as well, so the pair $(h_0,g)$ was
  added to $D$ by rule {\bf(iv)}.  It follows that the predecessor
  $(D,\mu,S,h_0)$ of $\psi$ meets {\bf(v)}, which is a contradiction.

  We therefore have $h=h_0$ and $(h,g)\notin \ov{S} \subset D_t\cup
  F$.  Since $\iota\psi$ satisfies $\X$ we have $\wt\mu_h \geq
  \la_{h-1}$, and Lemma~\ref{lem:fisg} applied to $\iota\psi$ shows
  that $f=g$.  It follows that $\ov\mu_h \geq \la_{h-1}$, since
  otherwise we must have $\ov\mu_h < \la_{h-1} = \wt\mu_h$ and
  $(h,f)=(h,g)\in \wt S$, contradicting that $\iota\psi$ satisfies
  $\X$.  Now Lemma~\ref{lem:fisg} implies that $f_h(\ov\mu) = g$, so
  $\iota\psi(D_t\cup F,h_0)$ satisfies $\X$ anyway and meets
  {\bf(iii)} or {\bf(iv)}.

  (e).  If $h'>h$, then the substitution rule decreases the level of
  $\psi'$, and Lemma~\ref{lem:cancelX} implies that the same thing
  happens to $\iota\psi(D,h')$.  Finally, if $h'=h$, then both $\psi'=\psi$
  and $\iota\psi(D',h')=\iota\psi$ satisfy condition $\X$ and meet
  {\bf(v)}.
\end{proof}

This completes the proof of Claim 2, and of Theorem \ref{mainthm}.

\section{Theta Polynomials}
\label{thetasec}

\subsection{}
\label{theta1}
In this section we develop the theory of theta polynomials
systematically; the exposition is influenced by that in Macdonald's
text \cite[III.8]{M}. We shall show that these polynomials share many
common features with the Schur $Q$-functions. One exception is that when 
$k>0$, we do not have a natural Hopf algebra structure.

Given any power series $\sum_{i \geq 0} c_i t^i$ in the variable $t$
and an integer sequence $\al = (\al_1,\al_2,\dots,\al_\ell)$, we write
$c_\al = c_{\al_1} c_{\al_2} \cdots c_{\al_\ell}$ and set $R\,c_{\al}
= c_{R\al}$ for any raising operator $R$.  We will always work with
power series with constant term 1, so that $c_0=1$ and $c_i=0$ for
$i<0$. The formal identities (\ref{E:formal}) imply that the equations
\begin{equation}
\label{jtformal}
\prod_{i<j}(1-R_{ij})\, c_\al = 
\det (c_{\al_i+j-i})_{i,j}
\end{equation}
and 
\[
\prod_{i<j}\frac{1-R_{ij}}{1+R_{ij}}\, c_\al = 
\Pf(C_{\al_i,\al_j})_{i<j}
\]
where
\[
C_{\al_i,\al_j} = \frac{1-R_{12}}{1+R_{12}} \, c_{\al_i,\al_j} = 
c_{\al_i}c_{\al_j} - 2 c_{\al_i+1}c_{\al_j-1} + 2 c_{\al_i+2}c_{\al_j-2}
- \cdots
\]
are valid in the polynomial ring $\Z[c_1,c_2,\ldots]$.

Let $x=(x_1,x_2,\ldots)$ be a list of commuting
independent variables and let $\Lambda=\Lambda(x)$ be the ring of
symmetric functions in $x$. Consider the generating functions
\[
E(x\,;t) = \prod_{i=1}^{\infty}(1+x_it) = \sum_{r=0}^{\infty}
e_r(x)t^r\ \ \mathrm{and} \ \
H(x\,;t) = \prod_{i=1}^{\infty}\frac{1}{1-x_it} = \sum_{r=0}^{\infty}h_r(x)t^r
\]
for the elementary and complete symmetric functions $e_r$ and $h_r$,
respectively.  Fix an integer $k\geq 0$, let $y=(y_1,\ldots,y_k)$, and
for each $r$ define $\ti_r=\ti_r(x\,;y)$ by
\[
\ti_r = \sum_{i\geq 0} q_{r-i}(x) e_i(y).
\]
We let $\Gamma^{(k)}$ be the subring of
$\Lambda\otimes\Z[y_1,\ldots,y_k]^{S_k}$
generated by the $\ti_r$:
\[
\Gamma^{(k)}= \Z[\ti_1,\ti_2,\ti_3,\ldots].
\]

Set $\Ti(t)=\sum_{r\geq 0} \ti_r t^r$; we then have
\[
\Ti(t)= \prod_{i=1}^\infty\frac{1+tx_i}{1-tx_i}
\prod_{j=1}^k(1+y_jt) = E(x\,;t)H(x\,;t)E(y\,;t)
\]
and hence
\[
\Ti(t)\Ti(-t)= E(y\,;t)E(y\,;-t) = \sum_{m=0}^{2k} (-1)^m e_m(y^2) t^{2m}\,,
\]
where $y^2$ denotes $(y_1^2,\ldots,y_k^2)$.  It follows that
\begin{equation}
\label{first}
\sum_{i+j=r}(-1)^i \ti_i\ti_j = \begin{cases}
0 & \text{if $r$ is odd} \\
(-1)^{r/2}e_{r/2}(y^2) & \text{if $r$ is even.}
\end{cases}
\end{equation}
In particular, when $r=2m>2k$, equation (\ref{first}) gives
\begin{equation}
\label{later}
\ti_m^2 = 2\sum_{i=1}^m(-1)^{i+1}\ti_{m+i}\ti_{m-i}.
\end{equation}

\begin{defn}
Given any $k$-strict partition $\la$, we let $\la^1$ be the strict
partition obtained by removing the first $k$ columns of $\la$, and let
$\la^2$ be the partition of boxes contained in the first $k$ columns
of $\la$.
\[ \pic{0.50}{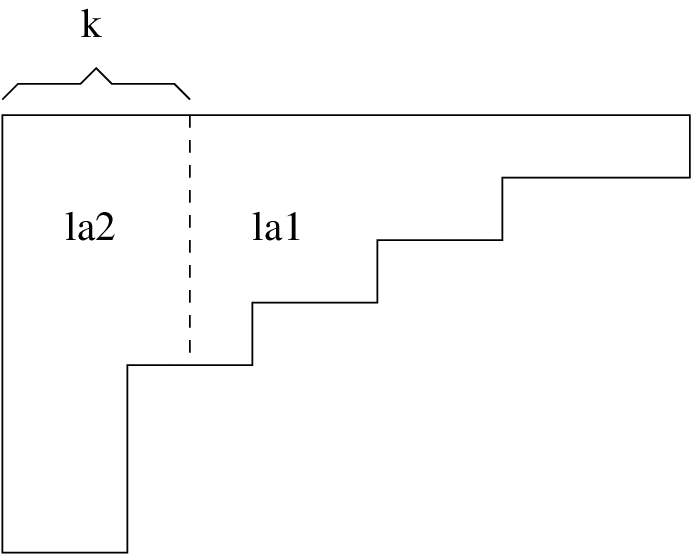} \]
We say that a partition $\lambda$ is {\em $k$-odd} if all its parts
which are greater than $2k$ are odd. 
\end{defn}

\begin{prop}
\label{thbasis}
{\em (a)} The $\ti_{\la}$ for $\la$ $k$-strict form a $\Z$-basis of
$\Gamma^{(k)}$.

\medskip
\noindent 
{\em (b)} The $\ti_{\la}$ for $\la$ $k$-odd form a $\Q$-basis of
$\Gamma_{\Q}^{(k)} := \Gamma^{(k)} \otimes_{\Z} \Q$.
\end{prop}
\begin{proof}
  It follows from (\ref{later}) that for any partition $\la$, either
  $\la$ is $k$-strict, or $\ti_{\la}$ is a $\Z$-linear combination of
  the $\ti_{\mu}$ such that $\mu$ is $k$-strict and $\mu \succ \la$
  (dominance order).  Furthermore, we have
  \[
  \ti_{\la}(x\,;y)= \sum_{\alpha} q_{\la-\alpha}(x)e_{\alpha}(y),
  \]
  the sum over all compositions $\alpha$ with $0\leq \alpha_i\leq k$
  for all $i$.  If $\la$ is $k$-strict, we deduce that the
  homogeneous summand of $\ti_\la$ of lowest $x$-degree is equal to
  $q_{\la^1}(x)e_{\la^2}(y)$. Part (a) follows because the set of all
  products $q_{\la^1}(x) e_{\la^2}(y)$, given by $k$-strict partitions
  $\la$, is linearly independent over $\Z$.

  Equation (\ref{first}) implies that $\ti_{2m} \in
  \Q[\ti_1,\ldots,\ti_{2m-1}]$ for $m>k$. By induction on $m$ it
  follows that $\ti_{2m} \in \Q[\ti_1,\ldots,\ti_{2k}, \ti_{2k+1},
  \ti_{2k+3},\ldots,\ti_{2m-1}]$ for all $m>k$, hence the monomials
  $\ti_\la$ indexed by $k$-odd partitions $\la$ span $\Gamma^{(k)}_\Q$
  as a vector space over $\Q$.  Finally, for each $d \in \N$, the
  number of $k$-odd partitions of $d$ is equal to the number of
  $k$-strict partitions of $d$, as verified by the equality of
  generating functions
  \begin{gather*}
    \sum_{\la \ k\text{-odd}} t^{|\la|} =
    \prod_{r=1}^{2k}\frac{1}{1-t^r}\prod_{r > k}\frac{1}{1-t^{2r-1}}
    = \prod_r \frac{1}{1-t^r} \prod_{r>k}(1-t^{2r}) \\
    = \prod_{r=1}^k\frac{1}{1-t^r}\prod_{r>k}(1+t^r)=
    \sum_{\la \ k\text{-strict}} t^{|\la|}\,.
  \end{gather*}
  This completes the proof of the proposition.
\end{proof}

\subsection{}\label{theta2}
Recall the notation of \S\ref{S:delta0}.

\begin{defn}
\label{Thdef}
For any valid set of pairs $D\subset \Delta^\circ$ and
integer sequence $\la$, define the
polynomial $\Ti(D,\la)$ by the raising operator formula
$\Ti(D,\la) = R^D\ti_\la$.  Equivalently, we recursively set
\[
  \Ti(D,\la) = \sum_\al (-1)^{|\al|} 2^{m(D,\al,\ell)}
  \Ti(D,\mu+\al) \ti_{r-|\al|}\,,
\]
where $\la=(\mu,r)$ has length $\ell$
and the sum is over all $(D,\ell)$-compatible
vectors $\al \in \N^{\ell-1}$.  For any $k$-strict partition $\la$,
the {\em theta polynomial} $\Ti_\la(x\,;y)$ is defined by
$\Ti_{\la} = \Ti(\cC(\la),\la)=R^{\la}\ti_\la$.
\end{defn}

Definition~\ref{Thdef} and equation (\ref{later}) imply that
each polynomial $\Ti_{\lambda}$ can be written in the form
  \[
  \Ti_{\la} = \ti_{\la} + \sum_{\mu \succ \la} a_{\la\mu} \ti_{\mu}
  \]
where the sum is over $k$-strict partitions $\mu \succ \la$ and
$a_{\la\mu} \in \Z$.  We deduce from
Proposition~\ref{thbasis}(a) that the polynomials $\Theta_\la$ indexed
by $k$-strict partitions $\la$ form a $\Z$-basis of $\Gamma^{(k)}$.

Let
\[
\IH(\IG_k) = \lim_{\longleftarrow}\HH^*(\IG(n-k,2n),\Z)
\]
be the stable cohomology ring of $\IG$; that is, the inverse limit in
the category of {\em graded} rings of the system
\[
\cdots \leftarrow \HH^*(\IG(n-k,2n),\Z) \leftarrow
\HH^*(\IG(n+1-k,2n+2),\Z) \leftarrow \cdots
\]
From the presentation of $\HH^*(\IG(n-k,2n),\Z)$ given in
\cite[Thm.\ 1.2]{BKT}, we deduce that $\IH(\IG_k)$ is isomorphic to
the polynomial ring $\Z[\s_1,\s_2,\ldots]$ modulo the relations
\[
\s_m^2+ 2\sum_{i=1}^m(-1)^i\s_{m+i}\s_{m-i} = 0
\]
for all $m>k$.  Since the generators $\ti_r$ of $\Gamma^{(k)}$ satisfy
(\ref{later}), we have a surjective ring homomorphism
$\phi:\IH(\IG_k)\to\Gamma^{(k)}$ sending $\s_r$ to $\ti_r$ for each
$r$.  Theorem \ref{mainthm} implies that $\phi(\s_{\la}) = \Ti_\la$
for any $k$-strict partition $\la$.  Since the $\Ti_\la$ form a basis
of $\Gamma^{(k)}$, we conclude that $\phi$ is an isomorphism.  This
completes the proof of Theorem \ref{productthm}.

\begin{remark}
For any $k$-strict partition $\la$ and formal power series 
$c=\sum_{i\geq 0} c_i t^i$, define the polynomial $\Ti_{\la}(c) =
R^{\la}c_\la$. The $\Ti_\la(c)$ are Giambelli polynomials for both the
classical and quantum cohomology of isotropic Grassmannians (Theorem
\ref{mainthm} and \cite{BKT2}) and appear in more general Giambelli
formulas which hold in the equivariant cohomology ring of isotropic
partial flag varieties \cite{T3}.
\end{remark}

\subsection{}
\label{theta3}
Consider the analogues of the polynomials $\ti_r$ when the $e_r(y)$
are replaced by complete symmetric functions $h_r(y)$.  Define
for each $r$ a function $\oti_r=\oti_r(x\,;y)$ by
\[
\oti_r = \sum_i q_{r-i}(x) h_i(y)
\]
and set $\oTi(t)=\sum_{r\gequ 0} \oti_r t^r$.  We then have
$\Ti(t)\oTi(-t)=1$, or equivalently,
\begin{equation}
\label{duality}
\sum_{r=0}^n(-1)^r\ti_r\oti_{n-r}=0, \ \ \ \ n\gequ 1.
\end{equation}
As in \cite[I.2, (2.$9'$)]{M}, the equations (\ref{duality}) imply that
for any partition $\lambda$,
\begin{equation}
\label{oti}
\det\left(\ti_{\lambda_i+j-i}\right) =
\det\left(\oti_{\lambda'_i+j-i}\right).
\end{equation}
Here $\lambda'$ is the partition conjugate to $\lambda$, i.e.,
$\la'_i=\#\{h\ |\ \la_h \geq i\}$ for all $i$.

If $k=0$, then $\oti_r=\ti_r=q_r$ for every $r\geq 0$. Let $(1^r)$
denote the partition $(1,1, \ldots,1)$ of length $r$.
\begin{prop}
\label{otiTh}
Assume that $k\geq 1$ and $r\in\N$. Then  
$\oti_r(x\,;y) = \Theta_{(1^r)}(x\,;y)$.  
\end{prop}
\begin{proof}
Observe that $\cC(1^r) = \emptyset$. It follows from this, the
identity (\ref{jtformal}), and equation (\ref{oti}) that
\[
\Theta_{(1^r)} = \prod_{i<j}(1-R_{ij})\,\ti_{(1^r)} =
\det(\ti_{1+j-i})_{1\leq i,j \leq r} = \oti_r.
\qedhere
\]
\end{proof}

Equation (\ref{duality}) and the Whitney sum formula prove that the
polynomials $\oti_r=\Theta_{(1^r)}$ map to the Chern classes of the
dual of the tautological subbundle $\Sh\to\IG$ under the isomorphism
$\phi$ of \S \ref{theta2}. A Pieri rule for the products $\oti_r \cdot
\Theta_\la$ was obtained by Pragacz and Ratajski \cite{PRpieri} (see
also \cite[Ex.\ 4]{T3}).  

\begin{prop}
The $\oti_{\la}$ for $\la$ $k$-strict form a $\Q$-basis of
$\Gamma_{\Q}^{(k)}$.
\end{prop}
\begin{proof}
It is clear from the equations (\ref{oti}) that
$\Gamma^{(k)}=\Z[\oti_1,\oti_2,\oti_3,\ldots]$. Since
$\oti_{\la}(x\,;y)= \sum_{\alpha\geq 0}
q_{\la-\alpha}(x)h_{\alpha}(y)$, we deduce that if $\la$ is
$k$-strict, the homogeneous summand of $\oti_\la$ of lowest $x$-degree
is equal to $q_{\la^1}(x)h_{\la^2}(y)$. Moreover, the set of
products $q_{\la^1}(x) h_{\la^2}(y)$ for all $k$-strict partitions
$\la$ is linearly independent over $\Q$. The result now follows by
a dimension count.
\end{proof}

\begin{example} When $k=1$, we have
\[
3\,\Ti_{31} = 2 \,\oti_4 - 5\, \oti_{31} + 4 \,\oti_{211} - \oti_{1111}.
\]
We deduce that the $\oti_{\la}$ for $\la$ $k$-strict do {\em not}
form a $\Z$-basis of $\Gamma^{(k)}$. Furthermore, the transition matrix
between the $\Q$-bases $\{\oti_\la\}$ and $\{\Ti_\la\}$ of
$\Gamma^{(k)}_\Q$ is not triangular with respect to the dominance
order.
\end{example}

\subsection{}
\label{theta4}
We next introduce an analogue of the Schur $S$-functions in the ring
$\Gamma^{(k)}$.
\begin{defn}
\label{Sdef}
For any two finite integer sequences $\la$, $\mu$, define the function
$S^{(k)}_{\la/\mu}\in \Gamma^{(k)}$ by setting
\[
S^{(k)}_{\la/\mu}(x\,;y) = \det(\ti_{\la_i-\mu_j+j-i}(x\,;y))_{i,j}.
\]
\end{defn}

Assume that $\la$ and $\mu$ are two partitions.  Then, arguing as in
\cite[I.5]{M}, the skew function $S_{\la/\mu}^{(k)}(x\,;y)$ is zero
unless $\la_i\geq \mu_i$ for each $i$. The functions $S_{\la/\mu}(x)
:= S^{(0)}_{\la/\mu}(x\,;y)$ are well known (see \cite[III.8, Ex.\
7]{M} and \cite[Sec.\ 2.7]{W}).  We also let
\[
s_{\la'/\mu'}(y) = \det(e_{\la_i-\mu_j+j-i}(y))_{i,j}
\]
denote the (ordinary) skew Schur polynomial in the variables $y$.  We
have that $s_{\la'/\mu'}(y) = 0$ unless $0\leq \la_i-\mu_i\leq k$ for
each $i$.  The functions $S_{\la/\mu}(x)$ (respectively,
$s_{\la'/\mu'}(y)$) are known to be linear combinations of Schur
$Q$-functions $Q_{\nu}(x)$ (respectively, Schur $S$-polynomials
$s_{\nu'}(y)$) with positive integer coefficients.

\begin{prop}
\label{Sprop}
For any partitions $\la$, $\mu$ with $\mu\subset\la$, we have
\begin{equation}
\label{skewA}
S^{(k)}_{\la/\mu}(x\,;y) =
\sum_{\nu}S_{\la/\nu}(x)s_{\nu'/\mu'}(y)=
\sum_{\nu}S_{\nu/\mu}(x)s_{\la'/\nu'}(y)
\end{equation}
summed over all partitions $\nu$ such that $\mu\subset\nu
\subset\la$.
\end{prop}
\begin{proof}
  Let $\wt{x}=(\wt{x}_1,\wt{x}_2,\ldots)$ be another 
  infinite list of variables and define the ring
  $\wt{\Lambda}=\Z[e_1(\wt{x}),e_2(\wt{x}),\ldots]
  \otimes\Z[e_1(y),\ldots,e_k(y)]$. According to
  \cite[I.(5.10)]{M}, we have
\[
s_{\la'/\mu'}(\wt{x},y) =
\sum_{\nu}s_{\la'/\nu'}(\wt{x})s_{\nu'/\mu'}(y)=
\sum_{\nu}s_{\nu'/\mu'}(\wt{x})s_{\la'/\nu'}(y)
\]
in $\wt{\Lambda}$.  This is mapped to (\ref{skewA}) under the
ring homomorphism $\wt{\Lambda} \to \Gamma^{(k)}$ defined by sending
$e_i(\wt{x})$ to $q_i(x)$ and $e_j(y)$ to $e_j(y)$.
\end{proof}

The definition of $S_{\la}^{(k)}$ implies that $S_{\la}^{(k)} =
\ti_{\la}+ \sum_{\mu \succ \la} d_{\la\mu}\ti_\mu$ for some integers
$d_{\la\mu}$, and therefore that the set of $S_{\la}^{(k)}$ for $\la$
$k$-strict forms another $\Z$-basis of $\Gamma^{(k)}$. For any integer
sequence $\al$ and raising operator $R$, set $R\, S^{(k)}_\al =
S^{(k)}_{R\al}$.  The next result follows from the identity
$S^{(k)}_\la(x;y) = \Theta(\emptyset,\la)$, which is derived from
(\ref{jtformal}).

\begin{prop}
\label{kstrictS}
For any $k$-strict partition $\la$, we have
\[
\Theta_{\la}(x\,;y) = \prod_{(i,j)\in \cC(\la)}\,
(1-R_{ij}+R_{ij}^2-\cdots)\,S^{(k)}_{\la}(x\,;y).
\]
\end{prop}

\subsection{}
\label{theta5}
In this section, we give the proof of Theorem \ref{thcor}.  Let $\la$
be a $k$-strict partition of length $\ell$.  Note that if $\la_i +
\la_j \lequ 2k+j-i$ for all $i<j$, then $\cC(\la)=\emptyset$, and we
deduce from (\ref{jtformal}) that
$\Ti_{\la} = S^{(k)}_{\la}$.  Part (a) of the theorem then follows by
setting $\mu=0$ in (\ref{skewA}).

Suppose now that $\la_i + \la_j > 2k+j-i$ for all $i<j\lequ \ell$.
Then $\la$ is a {\em strict\/} partition. For any strict partition
$\mu \subset \la$ with $\ell(\mu) \geq \ell-1$, we define the shifted
skew shape
\[ \Sh(\la/\mu) = (\la+\epsilon_\ell) / (\mu+\epsilon_\ell) \,, \]
where $\epsilon_\ell = (0,1,\dots,\ell-1)$. 

Following \cite[Chp.\ 9]{HH} and \cite[III.8, (8.8) and Ex.\ 8(c)]{M},
for any integer sequence $\gamma$ of length $\ell$, the (generalized)
Schur $Q$-function $Q_\gamma$ is defined by
\[
Q_\gamma = \prod_{1\leq i<j \leq \ell} \frac{1-R_{ij}}{1+R_{ij}}\,
q_\gamma = R^\la q_\gamma.
\]
Given any raising operator $R$, we have
\[
R\,\ti_{\la} (x\,;y)=\ti_{R\la} (x\,;y)=
\sum_{\alpha} e_{\alpha}(y)\, q_{R\la-\alpha}(x)
=\sum_{\alpha} e_{\alpha}(y) \,R \, q_{\la-\alpha}(x).
\]
It follows that
\begin{equation}
\label{key}
\Ti_{\lambda} = \sum_{\alpha} e_{\alpha}(y) \,
R^{\la} q_{\la-\alpha}(x) = \sum_{\alpha} Q_{\la-\alpha}(x) e_{\al}(y).
\end{equation}
where the sums run over all compositions $\alpha$ with
$0\leq \alpha_i\leq k$ for each $i$.

Since $\la$ is strict and $\la_{\ell-1}+\la_{\ell} > 2k+1$, we see
that $\la_i >\al_i$ for all compositions $\al$ indexing the sum
(\ref{key}) and every $i$ except possibly $i =\ell$. 
We deduce from \cite[III.8, Ex.\ 8(c)]{M} or Lemma \ref{commuteC} that
$Q_\gamma$ is skew symmetric in $\gamma$ for all integer vectors
$\gamma=\la-\al$ which appear.  It follows that we may rewrite
(\ref{key}) as
\[
\Ti_{\lambda}= \sum_{\mu}
\sum_{w\in S_{\ell}} (-1)^w Q_{\mu}(x)e_{\la-w(\mu)}(y)
\]
summed over strict partitions $\mu$ with $\ell(\mu) \in
\{\ell-1,\ell\}$.  Part (b) follows from this because
\[
\sum_{w \in S_\ell} (-1)^w e_{\la-w(\mu)}(y) =
\det(e_{\la_i-\mu_j}(y))_{1 \leq i,j \leq \ell} = s_{\Sh(\la/\mu)'}(y) \,.
\]

\section{Schubert Polynomials for Isotropic Grassmannians}\label{BH}

\subsection{}\label{bh1}

The polynomials $\Ti_\la(x\,;y)$ fall within the Billey-Haiman theory
of type C Schubert polynomials $\CS_w(x,z)$.  We will prove and
discuss this in detail in this section.  Let $W_n$ be the
hyperoctahedral group of signed permutations on the set
$\{1,\ldots,n\}$, and define $W_\infty = \bigcup_nW_n$.  The group
$W_{\infty}$ is generated by the simple transpositions $s_i=(i,i+1)$
for $i\geq 1$, and the sign change $s_0$ defined by $s_0(1)=\ov{1}$
and $s_0(p)=p$ for $p>1$.  Let $w \in W_\infty$.  A {\em reduced
  factorization\/} of $w$ is a product $w = u v$ in $W_n$ such that
$\ell(w) = \ell(u) + \ell(v)$.  We say that $w$ has a {\em descent\/}
at position $i$ if $\ell(w s_i) < \ell(w)$; this is equivalent to the
inequality $w(i) > w(i+1)$ if we set $w(0)=0$.  The signed permutation
$w$ is called $k$-{\em Grassmannian\/} if $k$ is the only descent
position for $w$.

The elements of $W_n$ index the Schubert classes in the cohomology
ring of the flag variety $\Sp_{2n}/B$, which contains
$\HH^*(\IG(n-k,2n),\Z)$ as the subring spanned by Schubert classes
given by $k$-Grassmannian elements.  In particular, each $k$-strict
partition $\la$ in $\cP(k,n)$ corresponds to a $k$-Grassmannian
element $w_\la \in W_n$ which we proceed to describe; more details
and relations to other indexing conventions can be found in \cite[\S
4]{T1}.

Notice that a $k$-strict partition $\la$ belongs to $\cP(k,n)$ if and
only if its Young diagram fits inside the shape $\Pi$ obtained by
attaching an $(n-k)\times k$ rectangle to the left side of a staircase
partition with $n$ rows.  When $n=7$ and $k=3$, this shape looks as
follows.
\[ \Pi \ \ = \ \ \ \raisebox{-36pt}{\pic{.6}{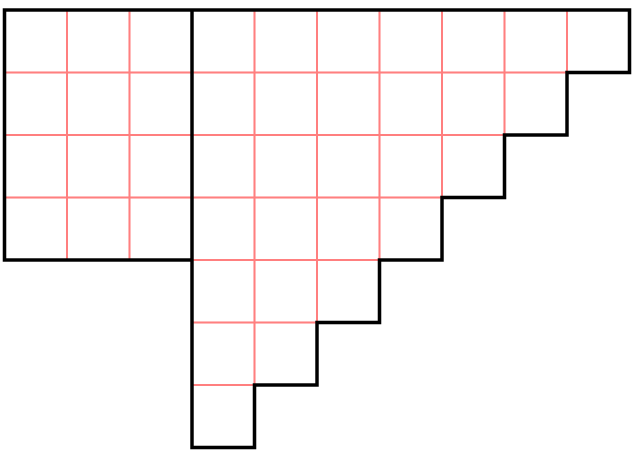}} \] The boxes of the
staircase partition that are outside $\la$ are organized into
south-west to north-east diagonals.  Such a diagonal is called {\em
  related\/} if it is $k$-related to one of the bottom boxes in the
first $k$ columns of $\la$, or to any box $[1,i+1]$ for which $\la_1 <
i \leq k$; the remaining diagonals are {\em non-related}.  The
$k$-Grassmannian element for $\la$ is defined by 
\[
w_\la =
(r_1,\dots,r_k, \ov{(\la^1)_1}, \dots, \ov{(\la^1)_p},
u_1,\dots,u_{n-k-p}),
\] 
where $r_1 < \dots < r_k$ are the lengths of the related diagonals,
$p=\ell(\la^1) = \ell_k(\la)$, and $u_1 < \dots < u_{n-k-p}$ are the
lengths of the non-related diagonals.  For example, the partition $\la
= (7,4,2) \in \cP(3,7)$ corresponds to the element $w_\la =
(2,5,6,\ov{4},\ov{1},3,7)$.
\[
\lambda \ = \ \ \raisebox{-53pt}{\pic{.6}{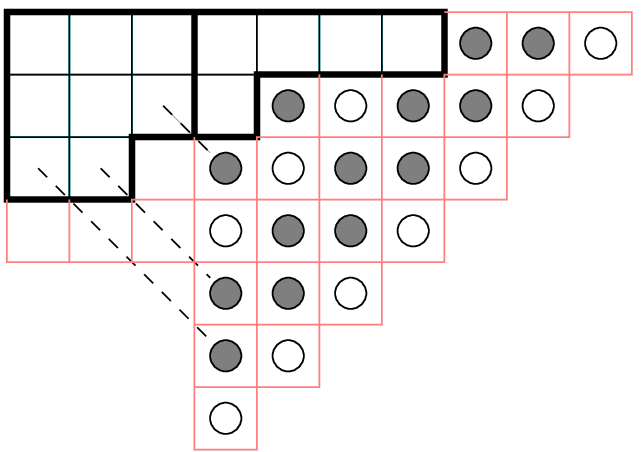}}
\]
The signed permutation $w_\la \in W_\infty$ depends on $\la$ and $k$,
but is independent of $n$.  Furthermore, if $\la_1 \leq k$, then
$w_\la \in S_\infty$ is the type A Grassmannian permutation for the
conjugate partition $\la'$ with descent at position $k$.

\subsection{}\label{bh2}

A sequence $a=(a_1,\ldots,a_m)$ is called {\em unimodal} if for some
index $r$ we have
\[
a_1 > a_2 > \cdots > a_{r-1} > a_r < a_{r+1} < \cdots < a_m .
\]
A {\em subsequence\/} of $a$ is any sequence $(a_{i_1}, \dots,
a_{i_p})$ with $1 \leq i_1 < \dots < i_p \leq m$.

Let $w \in W_{\infty}$ and let $\la$ be a strict partition such that
$|\la| = \ell(w)$.  A {\em Kra\'skiewicz tableau} \cite{Kr} for $w$ of
shape $\la$ is a filling $T$ of the boxes of $\lambda$ with
nonnegative integers such that, if $T_i$ is the sequence of integers
in row $i$ from left to right, then (a) the row word $T_{\ell(\la)}
\ldots T_1$ is a reduced word for $w$; and (b) for each $i$, $T_i$ is
a unimodal subsequence of maximum length in the word $T_{\ell(\la)}
\ldots T_{i+1} T_i$.

For each $w\in W_{\infty}$ one has a {\em type C Stanley symmetric
function} $F_w(x)$, which is a positive linear combination of Schur
$Q$-functions \cite{BH,FK,L}. There exist several combinatorial
interpretations for the coefficients in this expression.  We will use
a result of Lam \cite{L} stating that
\begin{equation}
\label{Feq}
F_w(x) = \sum_{\la} e_w^{\la} \, Q_{\la}(x)
\end{equation}
where $e_w^{\la}$ equals the number of Kra\'skiewicz
tableaux for $w$ of shape $\la$.

\begin{example}\label{maxGr}
  Assume that $k=0$ and let $w_{\la}$ be the $0$-Grassmannian element
  defined by a strict partition $\la$.  In this case there exists a
  unique Kra\'skiewicz tableau $T_\la$ for $w_\la$.  This tableau has
  shape $\la$ and its $i$th row contains the integers between $0$ and
  $\la_i-1$ in decreasing order.  For example, we have
  \[
  T_{(6,5,2)} =
  \young(543210,43210,10)
  \ .
  \]
  To see this, one checks that every reduced word for $w_\la$ can be
  obtained from the row word of $T_\la$ by using the commuting
  relations $s_i s_j = s_j s_i$ for $|i-j|>2$; condition (b) above
  then implies that any Kra\'skiewicz tableau for $w_\la$ has the same
  top row as $T_\la$, and the remaining rows are determined by
  induction on $\ell(\la)$.  We deduce that $F_{w_{\la}}(x) =
  Q_{\la}(x)$.
\end{example}

\subsection{}\label{bh3}

Following Billey and Haiman, each $w\in W_{\infty}$ defines a type C
Schubert polynomial $\CS_w(x,z)$.  Here $z=(z_1,z_2,\ldots)$ is
another infinite set of variables and each $\CS_w$ is a polynomial in
the ring $A=\Z[q_1(x),q_2(x),\ldots;z_1,z_2,\ldots]$.  The polynomials
$\CS_w$ for $w\in W_{\infty}$ form a $\Z$-basis of $A$, and their
algebra agrees with the Schubert calculus on symplectic flag varieties
$\Sp_{2n}/B$, when $n$ is sufficiently large.  According to
\cite[Thm.\ 3]{BH}, for any $w\in W_n$ we have
\begin{equation}
\label{BHgeneq}
\CS_w(x,z) = \sum_{uv=w} F_u(x) \AS_v(z) \,,
\end{equation}
summed over all reduced factorizations $w = uv$ in $W_n$ for which
$v\in S_n$.  Here $\AS_v(z)$ denotes the type A Schubert polynomial of
Lascoux and Sch\"utzenberger \cite{LS}.

We next show that each theta polynomial $\ti_r$ agrees with the
Billey-Haiman Schubert polynomial indexed by the $k$-Grassmannian
element $w_{(r)}\in W_\infty$ corresponding to $\la = (r)$.  It is
easy to see that $w_{(r)}$ has a unique reduced expression, given by
$w_{(r)} = s_{k-r+1} s_{k-r+2} \cdots s_k$ when $1 \leq r \leq k$, and
by $w_{(r)} = s_{r-k-1} s_{r-k-2} \cdots s_1 s_0 s_1 \cdots s_k$ when
$r \geq k+1$.  It follows that if $w_{(r)} = u v$ is any reduced
factorization of $w_{(r)}$ with $v \in S_\infty$, then $v = w_{(i)}$
for some integer $i$ with $0 \leq i \leq k$.  The type A Schubert
polynomial for $w_{(i)}$ is given by $\AS_{w_{(i)}}(z) =
e_i(z_1,\dots,z_k)$, and (\ref{Feq}) implies that the type C Stanley
symmetric function for $u = w_{(r)} w_{(i)}^{-1}$ is $F_u(x) =
q_{r-i}(x)$.  We conclude from (\ref{BHgeneq}) that
\begin{equation}\label{theqC}
  \CS_{w_{(r)}}(x,z) = \sum_{i=0}^k q_{r-i}(x_1,x_2,\ldots)
  e_i(z_1,\ldots,z_k) = \ti_r(x\,;z) \,,
\end{equation}
as required. Since the Schubert polynomials $\CS_w$ multiply like the
Schubert classes on symplectic flag varieties, Theorem \ref{mainthm}
and (\ref{theqC}) imply the following result.

\begin{prop}\label{bhprop}
  The ring $\Gamma^{(k)}$ of theta polynomials is a subring of the
  ring of Billey-Haiman Schubert polynomials of type C.  For every
  $k$-strict partition $\la$ we have $\Ti_{\la}(x\,;y)=
  \CS_{w_{\la}}(x,y)$.
\end{prop}
\noindent
Notice that Proposition \ref{bhprop} may be used to get a different
proof of Theorem \ref{productthm}.

Proposition~\ref{bhprop} and (\ref{BHgeneq}) imply that for every
$k$-strict partition $\la$ we have
\begin{equation}\label{BHeq}
  \Ti_\la(x\,;y) = \sum_{uv=w_{\la}} F_u(x) \AS_v(y) \,,
\end{equation}
where the sum over all reduced factorizations $w_\la = uv$ in
$W_\infty$ with $v\in S_\infty$.  The right factor $v$ in any such
factorization must be a Grassmannian permutation with descent at
position $k$.  In fact, it is not hard to check that the right reduced
factors of $w_\la$ that belong to $S_\infty$ are exactly the
permutations $w_\nu$ given by partitions $\nu \subset \la^2$.  Since
the Schubert polynomial $\AS_{w_\nu}(y)$ is equal to the Schur
polynomial $s_{\nu'}(y)$, we deduce from (\ref{Feq}) that
\begin{equation}\label{E:thm4}
  \Theta_\la(x;y) = \sum_{\mu,\nu} e^\la_{\mu \nu} Q_\mu(x) s_{\nu'}(y) \,,
\end{equation}
where the sum is over partitions $\mu$ and $\nu$ such that $\mu$ is
strict and $\nu \subset \la^2$, and $e^\la_{\mu\nu}$ is the number of
Kra\'skiewicz tableaux for $w_\la w_\nu^{-1}$ of shape $\mu$.  This
completes the proof of Theorem~\ref{bhthm}.

\begin{cor}\label{endcor} 
  Let $\la$ be a $k$-strict partition.\smallskip

  \noin{\em(a)} The homogeneous summand of $\Ti_{\la}(x\,;y)$ of highest
  $x$-degree is the type C Stanley symmetric function $F_{w_\la}(x)$,
  and satisfies $F_{w_{\la}}(x) = R^{\la}q_{\la}(x)$.\smallskip

  \noin{\em(b)} The homogeneous summand of $\Ti_{\la}(x\,;y)$ of
  lowest $x$-degree is $Q_{\la^1}(x)\,s_{(\la^2)'}(y)$.
\end{cor}
\begin{proof}
  Part (a) is deduced by setting $y=0$ in (\ref{BHeq}) and also in the
  raising operator expression $\Ti_{\la}(x\,;y) =
  R^{\la}\ti_{\la}(x\,;y)$.  Part (b) follows from (\ref{E:thm4}),
  Example~\ref{maxGr}, and the observation that $w_\la w_{\la^2}^{-1}$
  is the 0-Grassmannian Weyl group element corresponding to the strict
  partition $\la^1$.
\end{proof}

\begin{example}
  Let $k=1$ and $\la=(3,2,1)$, with corresponding Weyl group element
  $w_{\la} = (4, \ov{2}, \ov{1}, 3) \in W_4$.  Then we have
  \begin{equation*}
    \Ti_{321} = (Q_{42}+Q_{321}) +(Q_{41}+2\,Q_{32})\,s_{1'} +
    2\,Q_{31}\,s_{11'} + Q_{21}\,s_{111'}
  \end{equation*}
  (with the variables $x$ and $y$ omitted).  The terms in this
  expansion are accounted for by the Kra\'skiewicz tableaux in the
  following table.
  \[
  \begin{array}{c|c|c}
    \nu & w_\la w_\nu^{-1} & \text{Kra\'skiewicz tableaux for $w_\la
      w_\nu^{-1}$}\\ \hline
    \emptyset & (4,\ov{2},\ov{1},3) &
    \,\young(3201,01)\ \ \ \ \ \young(321,10,0)
    \raisebox{8mm}{\mbox{}}\raisebox{-6mm}{\mbox{}}
    \\ \hline
    (1) & (\ov{2},4,\ov{1},3) &
    \,\young(3102,0)\ \ \ \ \ \young(320,01)\ \ \ \ \ \young(302,01)
    \raisebox{6mm}{\mbox{}}\raisebox{-4mm}{\mbox{}}
    \\ \hline
    (1,1) & (\ov{2},\ov{1},4,3) &
    \,\young(310,0)\ \ \ \ \ \young(103,0)
    \raisebox{6mm}{\mbox{}}\raisebox{-4mm}{\mbox{}}
    \\ \hline
    (1,1,1) & (\ov{2},\ov{1},3,4) &
    \,\young(10,0)
    \raisebox{6mm}{\mbox{}}
  \end{array}
  \]
\end{example}

\begin{remark}
  The polynomials $2^{-\ell_k(\la)}\Ti_\la$ given by $k$-strict
  partitions $\la$ multiply like the Schubert classes on odd
  orthogonal Grassmannians $\OG(n-k,2n+1)$ and agree with the
  Billey-Haiman Schubert polynomials of type B indexed by
  $k$-Grassmannian elements $w_\la$.  For the even orthogonal
  Grassmannians $\OG(n-k,2n)$, both the Giambelli formula and the
  corresponding family of polynomials are more involved; we plan to
  develop this theory elsewhere.
\end{remark}


\begin{thebibliography}{BKT2}


\bibitem[BS]{BS} N.  Bergeron and F.  Sottile :
{\em A Pieri-type formula for isotropic flag manifolds}, Trans.
Amer.  Math.  Soc.  {\bf 354} (2002), 4815--4829.



\bibitem[BH]{BH} S.  Billey and M.  Haiman :
{\em Schubert polynomials for the classical groups},
J.  Amer.  Math.  Soc.  {\bf 8} (1995), 443--482.



\bibitem[BKT1]{BKT} A.  S.  Buch, A.  Kresch, and H.  Tamvakis :
{\em Quantum Pieri rules for isotropic Grassmannians},
Invent.  Math.  {\bf 178} (2009), 345--405.


\bibitem[BKT2]{BKT2} A.  S.  Buch, A.  Kresch, and H.  Tamvakis :
{\em Quantum Giambelli formulas for isotropic Grassmannians},
Preprint (2008).

\bibitem[C]{C} A.  L.  Cauchy :
{\em M\'emoire sur les fonctions qui ne peuvent obtenir que deux
valeurs \'egales et de signes contraires par suite des transpositions
op\'er\'es entre les variables qu'elles renferment},
J.  \'Ecole Polyt.  {\bf 10} (1815), 29--112; Oeuvres, ser.  2,
vol.  1, 91--169.


\bibitem[FK]{FK} S.  Fomin and A.  N.  Kirillov :
{\em Combinatorial $B_n$-analogs of Schubert polynomials},
Trans.  Amer.  Math.  Soc.  {\bf 348} (1996), 3591--3620.




\bibitem[G]{G} G.  Z.  Giambelli :
{\em Risoluzione del problema degli spazi secanti}, Mem.  R.  Accad.  Sci.
Torino (2) {\bf 52} (1902), 171--211.


\bibitem[HH]{HH} P. N. Hoffman and J. F. Humphreys :
{\em Projective Representations of the Symmetric Groups; $Q$-Functions 
and Shifted Tableaux}, Oxford Univ. Press, New York, 1992.




\bibitem[J]{J} C.  G.  J.  Jacobi : {\em De functionibus alternantibus
earumque divisione per productum e differentiis elementorum
conflatum}, J.  Reine Angew.  Math.  {\bf 22} (1841), 360--371.
Reprinted in Gesammelte Werke {\bf 3}, 439--452, Chelsea, New York,
1969.


\bibitem[Kr]{Kr} W.  Kra\'skiewicz :
{\em Reduced decompositions in hyperoctahedral groups},
C.  R.  Acad.  Sci.  Paris S\'er.  I Math.  {\bf 309} (1989), 903--907.

\bibitem[L]{L} T.  K.  Lam : {\em $B\sb n$ Stanley symmetric
    functions}, Proceedings of the 6th Conference on Formal Power
Series and Algebraic Combinatorics (New Brunswick, NJ, 1994),
Discrete Math.  {\bf 157}  (1996),  241--270.


\bibitem[LS]{LS} A.  Lascoux and M.-P.  Sch\"{u}tzenberger :
{\em Polyn\^{o}mes de Schubert}, C.  R.  Acad.  Sci.  Paris S\'er.  I
Math.  {\bf 294} (1982), 447--450.



\bibitem[M]{M} I.  Macdonald :
{\em Symmetric Functions and Hall Polynomials}, Second edition,
Clarendon Press, Oxford, 1995.



\bibitem[Pra]{Pra} P.  Pragacz :
{\em Algebro-geometric applications of Schur $S$- and $Q$-polynomials},
S\'{e}minare d'Alg\`{e}bre Dubreil-Malliavin 1989-1990, Springer Lecture
Notes in Math.  1478 (1991), 130--191.


\bibitem[PR]{PRpieri} P.  Pragacz and J.  Ratajski :
{\em A Pieri-type theorem for Lagrangian and odd orthogonal
Grassmannians}, J.  Reine Angew.  Math.  {\bf 476} (1996), 143--189.



\bibitem[S]{S2} I.  Schur :
{\em \"{U}ber die Darstellung der symmetrischen und der alternierenden
Gruppe durch gebrochene lineare Substitutionen}, J.  Reine Angew.
Math.  {\bf 139} (1911), 155--250.


\bibitem[T1]{T1} H.  Tamvakis : {\em Quantum cohomology of isotropic
Grassmannians}, Geometric Methods in Algebra and Number Theory,
311--338, Progress in Math.  235, Birkh\"auser, 2005.

\bibitem[T2]{T2} H.  Tamvakis : {\em Giambelli, Pieri, and tableau
formulas via raising operators}, J.  Reine Angew.  Math., to appear.

\bibitem[T3]{T3} H.  Tamvakis: {\em A Giambelli formula for classical
$G/P$ spaces}, Preprint (2009).

\bibitem[W]{W} D.  R.  Worley :
{\em A theory of shifted Young tableaux}, Ph.D.  thesis, MIT, 1984.

\bibitem[Y]{Y} A.  Young :
{\em On quantitative substitutional analysis VI}, Proc.  Lond.
Math.  Soc.  (2) {\bf 34} (1932), 196--230.


\end{thebibliography}
\end{document}